\documentclass[twoside,a4paper,reqno,11pt]{amsart} 
\usepackage[top=28mm,right=28mm,bottom=28mm,left=28mm]{geometry}
\usepackage{microtype}
\usepackage{amsfonts, amsmath, amssymb, amsbsy, mathrsfs, bm, latexsym, stmaryrd, hyperref, array, enumitem, textcomp, url}
\usepackage[table]{xcolor}

\renewcommand{\a}{\alpha}
\renewcommand{\b}{\beta}

\newcommand{\e}{\epsilon}
\renewcommand{\l}{\lambda}
 
\newcommand{\s}{\sigma}
\renewcommand{\O}{\Omega}

\newcommand{\la}{\langle}
\newcommand{\ra}{\rangle}
\newcommand{\leqs}{\leqslant}
\newcommand{\geqs}{\geqslant}
\newcommand{\normeq}{\trianglelefteqslant}

\newcommand{\vs}{\vspace{3mm}}

\makeatletter
\newcommand{\imod}[1]{\allowbreak\mkern4mu({\operator@font mod}\,\,#1)}
\makeatother

\theoremstyle{plain}
\newtheorem{theorem}{Theorem} 
 
\newtheorem{corol}[theorem]{Corollary}
\newtheorem{thm}{Theorem}[section] 
\newtheorem{lem}[thm]{Lemma}
\newtheorem{prop}[thm]{Proposition}

\newtheorem*{theorem*}{Theorem} 
\newtheorem*{conj*}{Conjecture}

\theoremstyle{definition}
\newtheorem{rem}[thm]{Remark}

\newtheorem{ex}[thm]{Example}

\begin{document}

\title[Base sizes for primitive groups with soluble stabilisers]{Base sizes for primitive groups with \\ soluble stabilisers}

\author{Timothy C. Burness}
\address{T.C. Burness, School of Mathematics, University of Bristol, Bristol BS8 1UG, UK}
\email{t.burness@bristol.ac.uk}

\date{\today} 

\begin{abstract}
Let $G$ be a finite primitive permutation group on a set $\Omega$ with point stabiliser $H$. Recall that a subset of $\O$ is a base for $G$ if its pointwise stabiliser is trivial. We define the base size of $G$, denoted $b(G,H)$, to be the minimal size of a base for $G$. Determining the base size of a group is a fundamental problem in permutation group theory, with a long history stretching back to the 19th century. Here one of our main motivations is a theorem of Seress from 1996, which states that $b(G,H) \leqs 4$ if $G$ is soluble. In this paper we extend Seress' result by proving that $b(G,H) \leqs 5$ for all finite primitive groups $G$ with a soluble point stabiliser $H$. This bound is best possible. We also determine the exact base size for all almost simple groups and we study random bases in this setting. For example, we prove that the probability that $4$ random elements in $\O$ form a base tends to $1$ as $|G|$ tends to infinity.
\end{abstract}

\maketitle

\section{Introduction}\label{s:intro}

Let $G$ be a permutation group on a set $\Omega$ and recall that a subset of $\Omega$ is a \emph{base} for $G$ if its pointwise stabiliser is trivial (that is, only the identity element fixes every point in the subset). The minimal cardinality of a base is called the \emph{base size} of $G$ and this invariant has been widely studied for more than a century, with numerous applications and connections to other areas of algebra and combinatorics. We refer the reader to the survey articles \cite{BC,LSh3} and \cite[Section 5]{Bur181} for more background on bases and their applications.

Determining the precise base size of a finite permutation group is a difficult problem, in general. Indeed, there is no known efficient algorithm for computing this number or for constructing a base of minimal size. In particular, a theorem of Blaha \cite[Theorem 3.1]{Blaha} implies that the problem of determining if the base size is at most a given integer is \textsf{NP}-complete. Therefore, it is natural to seek bounds on base sizes for interesting families of groups and there have been several major advances in this direction in recent years, particularly in the context of finite primitive groups.

Let $G \leqs {\rm Sym}(\O)$ be a finite primitive group of degree $n$ with point stabiliser $H$ and write $b(G,H)$ for the base size of $G$. Notice that if $G = S_n$ is the symmetric group in its natural action, then $b(G,H) = n-1$. Similarly, $b(G,H) = n-2$ for $G = A_n$. If $G$ is neither $S_n$ nor $A_n$, then a theorem of Bochert \cite{Boch} from 1889 shows that $b(G,H) \leqs n/2$. The best possible result here (up to a multiplicative constant) is due to Liebeck \cite{L10}, which states that either $b(G,H) < 9 \log_2 n$, or $n = \binom{m}{k}^r$ and $(A_m)^r \normeq G  \leqs S_m \wr S_r$, where the action of $S_m$ is on $k$-sets and the wreath product has the product action. This result, which relies on the Classification of Finite Simple Groups, extends earlier work of Babai \cite{Babai}, who proved that $b(G) < 4 \sqrt{n} \log_en$ if $G$ is simply primitive (Babai's proof does not use the Classification).

Further motivation for studying bases for finite primitive groups stems from several highly influential conjectures of Cameron, Kantor and Pyber from the early 1990s. As an immediate consequence of the definition of a base, we observe that $|G| \leqs n^{b(G,H)}$ and thus $b(G,H) \geqs \log |G| / \log n$. A conjecture of Pyber \cite{Pyber} asserts that there exists an absolute constant $c$ such that 
\[
b(G,H) \leqs c\frac{\log |G|}{\log n}
\]
for every primitive group $G$ of degree $n$. This conjecture has attracted the interest of various authors, with efforts to attack it organised according to the O'Nan-Scott theorem, which partitions the finite primitive groups into families depending on the structure and action of the socle of the group. By building on the work of several authors spanning more than 20 years, the proof of Pyber's conjecture was finally completed by Duyan, Halasi and Mar\'{o}ti \cite{DHM} in 2018. Also see \cite{HLM} for upper bounds with explicit constants. Stronger bounds have been established in some special cases. For example, if $G$ is soluble, then a striking theorem of Seress \cite{Seress} states that $b(G,H) \leqs 4$, which is best possible.

It is also possible to establish stronger bounds for some almost simple primitive groups (recall that $G$ is \emph{almost simple} if $G_0 \normeq G \leqs {\rm Aut}(G_0)$ for some nonabelian finite simple group $G_0$, which is the socle of $G$). Let us say that such a group $G \leqs {\rm Sym}(\O)$ is \emph{standard} if $G_0 = A_m$ is an alternating group and $\O$ is a set of subsets or partitions of $\{1, \ldots, m\}$, or if $G_0$ is a classical group and $\O$ is a set of subspaces (or pairs of subspaces) of the natural module for $G_0$ (otherwise, $G$ is \emph{non-standard}). For example, the natural action of $S_m$ is standard. In general, it is easy to see that if $G$ is standard of degree $n$ then $|G|$ is not bounded above by a fixed polynomial in $n$ and thus $b(G,H)$ can be arbitrarily large. However, if $G$ is non-standard then a conjecture of Cameron and Kantor \cite[p.142]{CK} asserts that $b(G,H) \leqs c$ for some absolute constant $c$ (they also conjectured that if $G$ is sufficiently large, then almost every $c$-tuple of points in $\O$ forms a base for $G$). This was subsequently refined by Cameron \cite[p.122]{CamPG}, who conjectured that $b(G,H) \leqs 7$, with equality if and only if $G$ is the Mathieu group ${\rm M}_{24}$ in its natural action on $24$ points. 

The original conjecture of Cameron and Kantor was proved by Liebeck and Shalev \cite{LSh2} using probabilistic methods and fixed point ratio estimates. By applying similar techniques, Cameron's refined conjecture was established in the sequence of papers \cite{Bur7,BGS,BLS,BOW}. The proof of Cameron's conjecture also reveals that if $G \leqs {\rm Sym}(\O)$ is non-standard and $\mathcal{P}(G,6)$ is the probability that a randomly chosen $6$-tuple of points in $\O$ forms a base for $G$, then $\mathcal{P}(G,6) \to 1$ as $|G| \to \infty$. Also see \cite{Bur18} for a classification of the non-standard groups with $b(G,H) = 6$ (there are infinitely many).

In this paper we extend some of this earlier work in several different directions. First recall Seress's theorem \cite{Seress}, which states that if $G$ is a finite primitive soluble group, then $b(G,H) \leqs 4$. In this setting, $G$ is an affine group and the point stabiliser $H$ is of course soluble itself. Given this result, it is natural to seek bounds on the base sizes of arbitrary primitive groups with soluble point stabilisers. It turns out that the base size of such a group is still bounded above by a small constant.

\begin{theorem}\label{t:main}
Let $G \leqs {\rm Sym}(\O)$ be a finite primitive permutation group with soluble point stabiliser $H$. Then $b(G,H) \leqs 5$.
\end{theorem}

The upper bound in Theorem \ref{t:main} is best possible and there are infinitely many groups that attain the bound. For example, if we take $G = S_5 \wr C_m$ in its product action on $5^m$ points, then $H = S_4 \wr C_m$ is soluble and $b(G,H) = 5$ for all $m \geqs 2$ (see Remark \ref{r:best}).

As a consequence of the O'Nan-Scott theorem, the primitive groups $G \leqs {\rm Sym}(\O)$ with soluble point stabilisers can be divided into three families: affine, almost simple and product-type groups of the form $T^r \normeq G \leqs L \wr S_r$, where $L \leqs {\rm Sym}(\Gamma)$ is an almost simple primitive group with socle $T$ and soluble point stabiliser and the action of $G$ on $\O = \Gamma^r$ is the product action. Moreover, Li and Zhang \cite{LZ} have determined all the almost simple primitive groups with this property, which relies on the extensive literature on maximal subgroups of almost simple groups. 

Suppose $G \leqs {\rm Sym}(\O)$ is almost simple and primitive with a soluble point stabiliser $H$. If $G$ is non-standard, then the proof of Cameron's conjecture yields $b(G,H) \leqs 6$ and the subsequent refinement in \cite{Bur18} gives $b(G,H) \leqs 5$ (in fact, by the main results in \cite{BGS,BOW}, the exact base size is known for all non-standard groups with socle an alternating or sporadic group). There are only partial results in the literature on base sizes for standard groups (for example, see \cite{BCN,BGL,Halasi,James, MS} for some results on bases for standard groups with an alternating socle). However, the soluble stabiliser hypothesis is rather restrictive and we can reduce the analysis of standard groups to symmetric and alternating groups of small degree and groups of Lie type of low rank (typically defined over small fields). These groups are amenable to direct calculation and we are able to determine the exact base size of every almost simple primitive group with a soluble point stabiliser. In particular, we can identify all of the groups with $b(G,H)=2$ and so this brings us a step closer towards a classification of the finite primitive groups with a base of size two, which is an ambitious project initiated by Jan Saxl in the 1990s.

Our main result for almost simple groups is Theorem \ref{t:as} below. In part (i)(b,c), we use the standard $P_m$ notation for maximal parabolic subgroups of classical groups; this is the stabiliser in $G$ of an $m$-dimensional totally isotropic subspace of the natural module for $G_0$. The tables referred to in part (ii) are presented at the end of the paper in Section \ref{s:tables} (see Remarks \ref{r:cases} and \ref{r:class} for information on the conventions adopted in these tables).

\begin{theorem}\label{t:as}
Let $G \leqs {\rm Sym}(\O)$ be a finite almost simple primitive group with socle $G_0$ and soluble point stabiliser $H$. Let $b=b(G,H)$ be the base size of $G$.
\begin{itemize}\addtolength{\itemsep}{0.2\baselineskip}
\item[{\rm (i)}] We have $b \leqs 5$, with equality if and only if one of the following holds:

\vspace{1mm}

\begin{itemize}\addtolength{\itemsep}{0.2\baselineskip}
\item[{\rm (a)}] $G = S_8$ and $H = S_4 \wr S_2$;
\item[{\rm (b)}] $G_0 = {\rm L}_{4}(3)$ and $H = P_2$;
\item[{\rm (c)}] $G_0 = {\rm U}_{5}(2)$ and $H = P_1$. 
\end{itemize}

\item[{\rm (ii)}] We have $b>2$ if and only if $(G,H,b)$ is one of the cases recorded in Tables \ref{tab:as1}--\ref{tab:as4}.
\end{itemize}
\end{theorem}

Let us record some immediate corollaries.

\begin{corol}\label{c:odd}
Let $G \leqs {\rm Sym}(\O)$ be a finite almost simple primitive group with socle $G_0$ and point stabiliser $H$. Suppose $|H|$ is odd and $b(G,H) > 2$. Then  $G_0 = {\rm L}_{2}(q)$, $q \equiv 3 \imod{4}$, $|G:G_0|$ is odd, $H = P_1$ is a Borel subgroup and 
\[
b(G,H) = \left\{\begin{array}{ll}
4 & \mbox{if $G \ne G_0$} \\
3 & \mbox{otherwise.}
\end{array}\right.
\]
\end{corol}

\begin{corol}\label{c:nilp}
Let $G \leqs {\rm Sym}(\O)$ be a finite almost simple primitive group with socle $G_0$ and point stabiliser $H$. Suppose $H$ is nilpotent and $b(G,H)>2$. Then $b(G,H) = 3$ and either
\begin{itemize}\addtolength{\itemsep}{0.2\baselineskip}
\item[{\rm (i)}] $G = {\rm Aut}(A_6)$ and $H$ is a Sylow $2$-subgroup of $G$; or 
\item[{\rm (ii)}] $G = {\rm PGL}_{2}(q)$, $q$ is a Mersenne prime and $H = D_{2(q+1)}$ is a Sylow $2$-subgroup of $G$.
\end{itemize}
\end{corol}

In the statement of the next result, we exclude the groups with socle $G_0 = {}^2G_2(3)' \cong {\rm L}_{2}(8)$ (here $b(G,H) \leqs 4$, with equality if and only if $G = {}^2G_2(3)$ and $H = 2^3{:}7{:}3$). See 
Table \ref{tab:as2} for a complete list of the exceptional groups with 
$b(G,H)=3$. 

\begin{corol}\label{c:ex}
Let $G \leqs {\rm Sym}(\O)$ be a finite almost simple primitive group with socle $G_0$ and soluble point stabiliser $H$. If $G_0$ is an exceptional group of Lie type, then $b(G,H) \leqs 3$, with equality only if $p \in \{2,3\}$ and $H$ is a parabolic subgroup.
\end{corol}

The proof of Theorem \ref{t:as} combines probabilistic and computational methods. Given a positive integer $c$ and a permutation group $G$ on a finite set $\O$, let 
\begin{equation}\label{e:pgc}
\mathcal{P}(G,c) = \frac{|\{(\a_1, \ldots, \a_c) \in \O^c \,:\, \bigcap_i G_{\a_i} = 1\}|}{|\O|^c}
\end{equation}
be the probability that a randomly chosen $c$-tuple of points in $\O$ forms a base for $G$. As in the proof of Cameron's base size conjecture, we can use fixed point ratios to estimate $\mathcal{P}(G,c)$, noting that $b(G,H) \leqs c$ if and only if $\mathcal{P}(G,c)>0$. In this way, we can establish the following asymptotic result for almost simple primitive groups.

\begin{theorem}\label{t:prob}
Let $(G_n)$ be a sequence of finite almost simple primitive permutation groups with soluble point stabilisers such that $|G_n| \to \infty$ as $n \to \infty$. 
\begin{itemize}\addtolength{\itemsep}{0.2\baselineskip}
\item[{\rm (i)}] We have $\mathcal{P}(G_n,4) \to 1$ as $n \to \infty$.
\item[{\rm (ii)}] Moreover, either $\mathcal{P}(G_n,3) \to 1$ as $n \to \infty$, or there exists an infinite subsequence of groups with socle ${\rm L}_{2}(q)$ and degree $q+1$. 
\end{itemize}
\end{theorem}

Let $G$ be a finite group, let $H$ be a soluble subgroup of $G$ and assume $G$ has no nontrivial soluble normal subgroups (so we may view $G$ as a transitive permutation group on the set of cosets of $H$). In this general setting, there is a conjecture attributed to Babai, Goodman and Pyber (cf. Conjecture 6.6 in \cite{BGP}) which asserts that  $b(G,H) \leqs 7$ (see Problem 17.41(a) in the Kourovka notebook \cite{Kou}; also see \cite[Problem 1]{Vdovin2}). By Theorem \ref{t:main}, this conjecture holds when $H$ is a maximal subgroup of $G$. In fact, the stronger bound $b(G,H) \leqs 5$ has been conjectured by Vdovin in \cite[Problem 17.41(b)]{Kou} and once again, our main theorem shows that this holds when $H$ is maximal. However, the general problem for non-maximal subgroups is still open.

In \cite{Vdovin2}, Vdovin essentially reduces his general conjecture to the almost simple groups and here there has been progress in some special cases. For example, Baikalov \cite{Bay2} has proved the conjecture for all soluble subgroups of symmetric and alternating groups and there are some partial results for groups of Lie type in \cite{Bay1, Vdovin3}. 

\vs

\noindent \textbf{Notation.} Let $G$ be a finite group and let $n$ be a positive integer. We will write $C_n$, or just $n$, for a cyclic group of order $n$ and $G^n$ will denote the direct product of $n$ copies of $G$. An unspecified extension of $G$ by a group $H$ will be denoted by $G.H$; if the extension splits then we write $G{:}H$. We use $[n]$ for an unspecified soluble group of order $n$. If $X$ is a subset of $G$, then $i_n(X)$ is the number of elements of order $n$ in $X$. We adopt the standard notation for simple groups of Lie type from \cite{KL}. In particular we write ${\rm L}_{n}^{\e}(q)$ for ${\rm PSL}_{n}(q)$ (when $\e=+$) and ${\rm PSU}_{n}(q)$ (when $\e=-$). The simple orthogonal groups are denoted ${\rm P\O}_{n}^{\e}(q)$, which differs from the notation used in the \textsc{Atlas} \cite{Atlas}. For positive integers $a$ and $b$, we write $(a,b)$ for the greatest common divisor of $a$ and $b$, while $\delta_{a,b}$ denotes the familiar Kronecker delta. All logarithms in this paper are base two.

\vs

\noindent \textbf{Organisation.} Let us say a few words on the layout of the paper. In Section \ref{s:prel} we discuss the probabilistic and computational methods that play a central role in the proofs of our main results. In Sections \ref{s:altsp}--\ref{s:excep}, which comprises the main bulk of the paper, we present proofs of Theorems \ref{t:as} and \ref{t:prob}, with the cases organised according to the possibilities for the socle $G_0$ and point stabiliser $H$ of $G$. The groups with socle an alternating or sporadic group are handled in Section \ref{s:altsp}. The two-dimensional linear groups with $G_0 = {\rm L}_{2}(q)$ require special attention and they are treated in Section \ref{s:psl2}. The remaining groups of Lie type are studied in Sections \ref{s:parab}--\ref{s:excep}, with the special cases where $H$ is a parabolic subgroup featuring in Section \ref{s:parab}. Finally, in Section \ref{s:main} we consider the affine and product-type primitive groups with soluble stabilisers and we combine Theorem \ref{t:as} with work of Seress \cite{Seress} to complete the proof of Theorem \ref{t:main}. The tables referred to in the statement of Theorem \ref{t:as} are presented in Section \ref{s:tables}.

\section{Preliminaries}\label{s:prel}

In this section we discuss some of the probabilistic and computational methods for calculating base sizes. These techniques will be applied repeatedly in the proofs of our main results.

\subsection{Bases}\label{ss:bases}

Let $G \leqs {\rm Sym}(\O)$ be a transitive permutation group on a finite set $\O$ with point stabiliser $H$. Let $b(G,H)$ denote the base size of $G$. As noted in Section \ref{s:intro}, the definition of a base implies that the elements of $G$ are distinguished by their action on a base and thus $|G| \leqs |\O|^{b(G,H)}$. This gives us the following useful lower bound on $b(G,H)$.

\begin{lem}\label{l:easy}
We have 
\[
b(G,H) \geqs \left\lceil \frac{\log |G|}{\log |\O|} \right\rceil.
\]
\end{lem}

In order to determine an upper bound $b(G,H) \leqs c$ we can either adopt a constructive approach with the aim of exhibiting a base of size $c$, or we can try to estimate the probability $\mathcal{P}(G,c)$ that a randomly chosen $c$-tuple in $\O$ forms a base for $G$ (see \eqref{e:pgc}), noting that $b(G,H) \leqs c$ if and only if $\mathcal{P}(G,c) > 0$. We will use both approaches in this paper, but we will predominantly seek to apply the probabilistic method whenever it is feasible to do so.

Probabilistic methods for studying bases were originally introduced by Liebeck and Shalev \cite{LSh2} in their proof of the Cameron-Kantor conjecture. The idea is very simple. Let $c$ be a positive integer and observe that $\{\a_1, \ldots, \a_c \} \subseteq \O$ is \emph{not} a base for $G$ if and only if there exists an element $x \in G$ of prime order such that $x \in G_{\a_i}$ for all $i$. Since we can interpret the \emph{fixed point ratio} of $x$, 
\[
{\rm fpr}(x,G/H) = \frac{|C_{\O}(x)|}{|\O|} = \frac{|x^G \cap H|}{|x^G|},
\]
as the probability that $x$ fixes a uniformly random element in $\O$ (here $C_{\O}(x)$ is the set of fixed points of $x$ on $\O$), it follows that  
\[
1 - \mathcal{P}(G,c) \leqs \sum_{x \in \mathcal{P}} {\rm fpr}(x,G/H)^c =: \mathcal{Q}(G,c),
\]
where $\mathcal{P}$ is the set of elements of prime order in $G$. Now  
 $|C_{\O}(x)| = |C_{\O}(x^g)|$ for all $g \in G$, whence 
 \begin{equation}\label{e:Q}
\mathcal{Q}(G,c) = \sum_{i=1}^{k} |x_i^{G}| \cdot \left(\frac{|x_i^{G} \cap H|}{|x_i^{G}|}\right)^c
\end{equation}
where $x_1, \ldots,x_k$ represent the $G$-classes of elements of prime order in $H$. We will repeatedly apply the following lemma.

\begin{lem}\label{l:base}
If $\mathcal{Q}(G,c)<1$ then $b(G,H) \leqs c$.	
\end{lem}

The following result is \cite[Lemma 2.1]{Bur7}, which provides a useful tool for bounding $\mathcal{Q}(G,c)$.

\begin{lem}\label{l:calc}
Suppose $x_{1}, \ldots, x_{m}$ represent distinct $G$-classes such that $\sum_{i}{|x_{i}^{G}\cap H|}\leqs A$ and $|x_{i}^{G}|\geqs B$ for all $i$. Then 
\[
\sum_{i=1}^{m} |x_i^{G}| \cdot \left(\frac{|x_i^{G} \cap H|}{|x_i^{G}|}\right)^c \leqs B(A/B)^c
\]
for every positive integer $c$.
\end{lem}

\subsection{Computational methods}\label{ss:comp}

We will use computational methods extensively in the proof of Theorem \ref{t:as} to handle small degree symmetric and alternating groups, as well as some low rank groups of Lie type defined over small fields. In all cases, we use {\sc Magma} V2.23-2 \cite{magma} to do the computations. Here we briefly describe the main techniques.

Let $G \leqs {\rm Sym}(\O)$ be an almost simple primitive group with socle $G_0$ and soluble point stabiliser $H$. Given $G$ as an abstract group, our initial aim is to construct $G$ as a permutation group of an appropriate degree (this is not necessarily the permutation representation of $G$ on $\O$). Typically we do this by first using the function \texttt{AutomorphismGroupSimpleGroup} to obtain $A = {\rm Aut}(G_0)$ as a permutation group and then we identify $G$ by inspecting the subgroups of $A$ containing $G_0$. For example, we can use the command \texttt{LowIndexSubgroups(A,m)}, which  returns a set of representatives of the $A$-classes of subgroups of $A$ of index at most $m$. 

Next we construct $H$ as a subgroup of $G$ in the same permutation representation. To do this, we usually use the command \texttt{MaximalSubgroups(G:IsSolvable:=true)}, which returns a set of representatives of the $G$-classes of soluble maximal subgroups of $G$ and it  is easy to identify the representative conjugate to $H$. For certain large groups of interest, the \texttt{MaximalSubgroups} function is ineffective and so in these cases we need to adopt a different approach. In the handful of cases where this issue arises, $G$ is a classical group and we can either use the \texttt{ClassicalMaximals} function to construct an appropriate maximal subgroup of the corresponding matrix group, or we can seek a direct construction of $H$ inside $G$. 

\begin{ex}\label{e:u63}
To illustrate the latter approach, suppose $G = {\rm Aut}({\rm U}_{6}(3))$ and $H$ is a maximal subgroup of type ${\rm GU}_{2}(3) \wr S_3$ (this case arises in the proof of Proposition \ref{p:nonpar1}). Here $|H \cap G_0| = 2^{13}3^4$ and we observe that $H = N_G(K)$, where $K$ is a subgroup of $G_0$ of order $2^{10}$. Given this, we can use the following code to construct $G$ and $H$ as permutation groups of degree $22204$:

{\small
\begin{verbatim}
G:=AutomorphismGroupSimpleGroup("U",6,3);
g:=Socle(G);
S:=SylowSubgroup(g,2);
S,f:=PCGroup(S);
N:=Subgroups(S:OrderEqual:=2^10);
exists(k){i : i in [1..#N] | #Normalizer(g,N[i]`subgroup@@f) eq 2^13*3^4};
H:=Normalizer(G,N[k]`subgroup@@f);
\end{verbatim}}
\end{ex}

\begin{ex}\label{e:o123}
Suppose $G = {\rm PGO}_{12}^{+}(3)$ and $H$ is a maximal subgroup of type ${\rm O}_{4}^{+}(3) \wr S_3$; this is also a genuine case that we will need to handle in the proof of Proposition \ref{p:nonpar1}. Write $G = L/Z$ and $H = K/Z$, where $L$ is the matrix group ${\rm GO}_{12}^{+}(3)$ and $Z = Z(L)$. We use \texttt{ClassicalMaximals} to construct $K$ (noting that $K$ is contained in Aschbacher's $\mathcal{C}_2$ collection of maximal subgroups of $L$, which explains why we set \texttt{classes:=\{2\}}) and then we take images modulo scalars to obtain $G$ and $H$ as permutation groups of degree $88816$:

{\small
\begin{verbatim}
L:=CGOPlus(12,3);
C:=ClassicalMaximals("O+",12,3: classes:={2}, normaliser:=true);
exists(i){i : i in [1..#C] | LMGIsSoluble(C[i]) eq true};
K:=C[i];
f,G,R:=PermutationRepresentation(L:ModScalars:=true);
H:=f(K);             
\end{verbatim}}
\end{ex}

Let us assume we have now constructed $G$ and $H$ as permutation groups. In most cases, we can compute $b(G,H)$ simply by combining the lower bound in Lemma \ref{l:easy} with a random search. More precisely, if $\lceil \log |G| / \log|\O| \rceil = c$ then Lemma \ref{l:easy} gives $b(G,H) \geqs c$ and by random search we will typically be able to find elements $x_1, \ldots, x_{c-1}$ in $G$ such that 
\[
H \cap H^{x_1} \cap \cdots \cap H ^{x_{c-1}} = 1,
\]
which gives the reverse inequality  $b(G,H) \leqs c$.

However, there are some situations where this approach is ineffective because $G$ does not have a base of size $c$. In other words,  
\[
c = \left\lceil \frac{\log |G|}{\log |\O|} \right\rceil < b = b(G,H).
\]
Here we establish the bound $b(G,H) \leqs b$ by random search and then to conclude we need to show that every $(b-1)$-point stabiliser is nontrivial. For example, if $G = S_8$ and $H = S_4 \wr S_2$, then $\lceil \log |G| /\log |\O| \rceil = 3$ and $b(G,H) \leqs 5$ by random search. By computing the order of every $4$-point stabiliser we deduce that $b(G,H) = 5$.  

To compute the order of every $(b-1)$-point stabiliser, we use the \texttt{CosetAction} function to construct $G$ as a permutation group on the set of cosets of $H$ and then we inspect stabiliser chains, working with representatives of the orbits of $k$-point stabilisers for $k<b-1$. This approach is straightforward to implement and it is effective for all but one case that arises in this paper. The exceptional case is described in the following example.

\begin{ex}\label{e:special}
Suppose $G_0 = {\rm P\O}_{8}^{+}(3)$ and $H$ is of type ${\rm O}_{4}^{+}(3) \wr S_2$. Here $\lceil \log |G| / \log |\O| \rceil = 2$ and by random search we deduce that $b(G,H) = 2$ if $|G:G_0| \leqs 4$. In the remaining cases, we claim that $b(G,H) = 3$. By random search, we get $b(G,H) \leqs 3$ and so it remains to show that every $2$-point stabiliser is nontrivial. But the method outlined above using \texttt{CosetAction} is ineffective since $|\O| = 14926275$ is prohibitively large. To resolve these cases, we use the double coset enumeration technique explained in \cite[Section 2.3.3]{BOW}. Here the aim is to find a set $T$ of distinct $(H,H)$ double coset representatives such that 
\begin{itemize}\addtolength{\itemsep}{0.2\baselineskip}
\item[{\rm (a)}] $|HxH| < |H|^2$ for all $x \in T$; and
\item[{\rm (b)}] $\sum_{x \in T} |HxH| > |G| - |H|^2$.
\end{itemize}
Indeed, if such a set $T$ exists, then $H$ does not have a regular orbit on $\O$ and we deduce that $b(G,H) \geqs 3$. As noted in \cite{BOW}, this approach can be implemented in {\sc Magma} and for the case above we quickly deduce that $b(G,H) = 3$ when $|G:G_0| \geqs 6$.
\end{ex}

\section{Alternating and sporadic groups}\label{s:altsp}

In this section we begin the proof of Theorem \ref{t:as} by handling the case where $G_0$ is either an alternating or sporadic simple group. Our main result is the following.

\begin{prop}\label{p:altsp}
Let $G \leqs {\rm Sym}(\O)$ be a finite almost simple primitive group with socle $G_0$ and soluble point stabiliser $H$. Set $b=b(G,H)$ and assume $G_0$ is either an alternating group or a sporadic simple group. 
\begin{itemize}\addtolength{\itemsep}{0.2\baselineskip}
\item[{\rm (i)}] We have $b \leqs 5$, with equality if and only if $G = S_8$ and $H = S_4 \wr S_2$.
\item[{\rm (ii)}] We have $b>2$ if and only if $(G,H,b)$ is one of the cases recorded in Table \ref{tab:as1}.
\end{itemize}
In addition, $\mathcal{P}(G,2) \to 1$ as $|G| \to \infty$.
\end{prop}

\begin{proof}
First assume $G_0$ is a sporadic simple group. Here the possibilities for $H$ are listed in \cite[Table 15]{LZ} and in each case $b(G,H)$ is computed in \cite{BOW}. The result follows by inspection.

Now assume $G_0 = A_n$ is an alternating group. The cases with $n \leqs 16$ are easily verified using {\sc Magma} \cite{magma} (see Section \ref{ss:comp}), so let us assume $n \geqs 17$. Then by inspecting \cite[Table 14]{LZ}, we see that $n = p$ and $H = {\rm AGL}_{1}(p) \cap G$ is the only possibility, where $p$ is a prime. Here $b(G,H) = 2$ by \cite[Theorem 1.1]{BGS}, so it just remains to show that $\mathcal{P}(G,2) \to 1$ as $p \to \infty$. Define $\mathcal{Q}(G,2)$ as in \eqref{e:Q} and observe that $|H| \leqs p(p-1)$ and every nontrivial element in $H$ has at most one fixed point on $\{1, \ldots, p\}$. By considering the involutions in $H$, we deduce that  
\[
|x^G| \geqs \frac{p!}{((p-1)/2)!2^{(p-1)/2}}
\]
for all $x \in H$ of prime order and one checks that this lower bound is greater than $p^5$ for $p \geqs 17$. Therefore, Lemma \ref{l:calc} implies that
\[
\mathcal{Q}(G,2) \leqs \frac{p^2(p-1)^2}{p^5} < p^{-1}
\]
and we conclude that $\mathcal{P}(G,2) \to 1$ as $|G| \to \infty$.
\end{proof}

\section{Two-dimensional linear groups}\label{s:psl2}

In this section we establish Theorem \ref{t:as} for the groups with socle $G_0 = {\rm L}_{2}(q)$. We begin by introducing some general notation, which we will use throughout this section. 

Let $V$ be the natural module for $G_0$ and write $q=p^f$ and $d = (2,q-1)$, where $p$ is a prime. Set $\tilde{G} = {\rm PGL}_{2}(q)$. Fix a basis $\{e_1,e_2\}$ for $V$ and write $\mathbb{F}_{q}^{\times} = \la \mu \ra$. Then 
\begin{equation}\label{e:aut}
{\rm Aut}(G_0) = \la G_0, \delta, \phi \ra,
\end{equation}
where $\delta \in \tilde{G}$ is the image (modulo scalars) of the diagonal matrix ${\rm diag}(\mu, 1) \in {\rm GL}_{2}(q)$ and $\phi$ is a field automorphism of order $f$ such that $(ae_1+be_2)^{\phi} = a^pe_1+b^pe_2$ for all $a,b \in \mathbb{F}_q$. For $g \in {\rm Aut}(G_0)$, if we write $\ddot{g}$ for the coset $G_0g$, then
\[
{\rm Out}(G_0) = \{ \ddot{g} \,:\, g \in {\rm Aut}(G_0) \} = \la \ddot{\delta} \ra \times \la \ddot{\phi} \ra = C_{d} \times C_{f}.
\]
If $H$ is a subgroup of $G$, then we set $H_0 = H \cap G_0$.

Since ${\rm L}_{2}(4) \cong {\rm L}_{2}(5) \cong A_5$ and ${\rm L}_{2}(9) \cong A_6$, we will assume $q \geqs 7$ and $q \ne 9$ (see Proposition \ref{p:altsp} for the excluded cases). The possibilities for $H$ are easily determined by inspecting \cite{LZ} (or by consulting \cite[Table 8.1]{BHR}) and they are recorded in Table \ref{tab:psl2}. Following \cite{KL}, we refer to the \emph{type} of $H$, which provides a rough description of the structure of $H$. Note that in the first row, $H$ is a parabolic subgroup of $G$ (as the notation indicates, it is the stabiliser in $G$ of a $1$-dimensional subspace of $V$).

\begin{table}
\[
\begin{array}{clll} \hline
\mbox{Case} & \mbox{Type of $H$}  & \mbox{Conditions} & \hspace{5mm} b(G,H) \\ \hline
{\rm (a)} & P_1 & & \hspace{5mm} \mbox{See Remark \ref{r:psl2}} \\
{\rm (b)} & {\rm GL}_{1}(q) \wr S_2 & & \left\{\begin{array}{ll}
3 & {\rm PGL}_{2}(q)<G \\
2 & \mbox{otherwise}
\end{array}\right. \\
{\rm (c)} & {\rm GL}_{1}(q^2) & & \left\{\begin{array}{ll}
3 & {\rm PGL}_{2}(q) \leqs G \\
2 & \mbox{otherwise}
\end{array}\right. \\
{\rm (d)} & {\rm GL}_{2}(3) & \mbox{$q = 3^k$, $k \geqs 3$ prime} & \hspace{5mm} 2 \\ 
{\rm (e)} & 2^{1+2}_{-}.{\rm O}_{2}^{-}(2) & q = p \geqs 7 & \hspace{5mm} 2+\delta_{7,q} \\ \hline
\end{array}
\]
\caption{The cases with $G_0 = {\rm L}_{2}(q)$}
\label{tab:psl2}
\end{table}

The main result of this section is the following.

\begin{prop}\label{p:psl2}
Let $G \leqs {\rm Sym}(\O)$ be a finite almost simple primitive group with socle $G_0 = {\rm L}_{2}(q)$ and soluble point stabiliser $H$, where $q \geqs 7$ and $q \ne 9$. Then $b(G,H)$ is recorded in the final column of Table \ref{tab:psl2}. In particular, $b(G,H) \leqs 4$ and $\mathcal{P}(G,c) \to 1$ as $q \to \infty$, where $c = 4$ if $H$ is of type $P_1$, otherwise $c=3$.
\end{prop}

\begin{rem}\label{r:psl2}
In case (a) we have $b(G,H) \in \{3,4\}$, with $b(G,H) = 3$ if and only if 
\begin{itemize}\addtolength{\itemsep}{0.2\baselineskip}
\item[{\rm (i)}] $G \leqs {\rm PGL}_{2}(q)$; or 
\item[{\rm (ii)}] $q$ is odd, $f$ is even and $G = \la G_0, \delta\phi^{f/2} \ra = G_0.\la \ddot{\delta}\ddot{\phi}^{f/2}\ra = G_0.2$.
\end{itemize}
Equivalently, $b(G,H) = 3$ if and only if $G=G_0$ or $G$ is sharply $3$-transitive.
\end{rem}

\begin{rem}\label{r:asy}
Let $G$ be as in Proposition \ref{p:psl2} with $q \geqs 11$ and set $b = b(G_0,H_0)$. Then   
\[
b = \left\{\begin{array}{ll}
3 & \mbox{if $H$ is of type $P_1$, or if $p=2$ and $H$ is of type ${\rm GL}_{1}(q^2)$} \\
2 & \mbox{otherwise} 
\end{array}\right.
\]
and the proof of Proposition \ref{p:psl2} shows that
\[
\mathcal{P}(G_0,b) \to \left\{\begin{array}{ll}
0 & \mbox{if $p=2$ and $H$ is of type ${\rm GL}_{1}(q) \wr S_2$} \\
1/2 & \mbox{if $p\ne 2$ and $H$ is of type ${\rm GL}_{1}(q) \wr S_2$ or ${\rm GL}_{1}(q^2)$} \\
1 & \mbox{otherwise}
\end{array}\right.
\]
as $q \to \infty$.
\end{rem}

We will prove Proposition \ref{p:psl2} with a sequence of lemmas. We refer the reader to \cite[Section 3.2]{BG} for a source of information on the conjugacy classes of elements of prime order in ${\rm Aut}(G_0)$. We start by considering case (a) in Table \ref{tab:psl2}.
 
\begin{lem}\label{l:l2p10}
If $G_0 = {\rm L}_{2}(q)$ and $H$ is of type $P_1$, then  $b(G,H) \leqs 4$ and $\mathcal{P}(G,4) \to 1$ as $q \to \infty$.
\end{lem}

\begin{proof}
Here $H_0 =(C_p)^f{:}C_{(q-1)/d}$ is a Borel subgroup of $G_0$ and we have $H = N_G(P)$, where $P$ is a Sylow $p$-subgroup of $G_0$. Note that $|\O| = q+1$ and we may identify $\O$ with the set of $1$-dimensional subspaces of $V$. In view of Lemma \ref{l:base}, it suffices to show that $\mathcal{Q}(G,4) < 1$ and $\mathcal{Q}(G,4) \to 0$ as $q$ tends to infinity. The cases with $q \leqs 32$ can be checked using {\sc Magma} \cite{magma} (see Section \ref{ss:comp}), so we will assume that $q>32$. Let $\chi$ be the corresponding permutation character of $\tilde{G} = {\rm PGL}_{2}(q)$, so $\chi(x) = |C_{\O}(x)|$ for all $x \in \tilde{G}$ and we note that $\chi = 1+{\rm St}$ is the sum of the trivial and Steinberg characters of $\tilde{G}$. We proceed by estimating the contribution to $\mathcal{Q}(G,4)$ from the different types of elements of prime order in $H$.

Suppose $x \in H$ has prime order $r$. If $x$ is unipotent then $r=p$, $|x^{\tilde{G}}| = q^2-1$ and $\chi(x) = 1$ (since every regular unipotent element is contained in a unique Borel subgroup, or recall that the Steinberg character  vanishes at nontrivial unipotent elements). Therefore, $|x^{\tilde{G}} \cap H| =  q-1$ and we deduce that the contribution to $\mathcal{Q}(G,4)$ from unipotent elements is equal to 
\[
\a_1 = \frac{(q-1)^4}{(q^2-1)^3} = \frac{q-1}{(q+1)^3}.
\]

Next assume $x$ is a semisimple involution, so $q$ is odd. If $x$ is the image of a diagonalisable matrix in ${\rm GL}_{2}(q)$ (that is, if $C_{G_0}(x)$ is the normaliser of a split torus), then $|x^G| = \frac{1}{2}q(q+1)$ and $x$ fixes exactly two $1$-spaces, so $\chi(x) = 2$ and $|x^G\cap H| = q$. On the other hand, if $C_{G_0}(x)$ is the normaliser of a non-split torus, then $\chi(x) = 0$ and $x^G \cap H$ is empty. It follows that the contribution from semisimple involutions is given by
\[
\a_2 = \frac{q^4}{\left(\frac{1}{2}q(q+1)\right)^3} = \frac{8q}{(q+1)^3}.
\]
Now assume $x \in H$ is semisimple and $r \geqs 3$. Here $r$ divides $q-1$, $|x^{\tilde{G}}| = q(q+1)$ and $\chi(x) = 2$, so $|x^{\tilde{G}} \cap H| = 2q$. Since there are $\frac{1}{2}(r-1)$ distinct $\tilde{G}$-classes of such elements, we conclude that the combined contribution to $\mathcal{Q}(G,4)$ from semisimple elements of odd order is equal to 
\[
\a_3 = \sum_{r \in \pi} \frac{1}{2}(r-1) \cdot \frac{16q}{(q+1)^3},
\]
where $\pi$ is the set of odd prime divisors of $q-1$. Now $r \leqs q-1$ and $|\pi| < \log q$, so
\[
\a_3 < \frac{8q(q-2)\log q}{(q+1)^3} = \a_3'.
\]

Finally, let us assume $q=q_0^r$ and $x$ is a field automorphism of order $r$. Here $C_{H_0}(x)$ is a Borel subgroup of $C_{G_0}(x)$ (see the proof of 
\cite[Lemma 6.1]{LLS2}, for example) and thus
\[
|x^{G_0} \cap H| = \frac{q(q-1)}{(1+\delta_{2,r})q_0(q_0-1)},\;\; 
|x^{G_0}| =  \frac{q(q^2-1)}{(1+\delta_{2,r})q_0(q_0^2-1)}.
\]
Since there are $r+\delta_{2,r}-1$ distinct $G_0$-classes of field automorphisms of order $r$ in ${\rm Aut}(G_0)$, the combined contribution from field automorphisms is equal to 
\[
\a_4 = \sum_{r \in \pi'} (r-1) \cdot \frac{(q_0+1)^3}{(q+1)^3} \cdot \frac{q(q-1)}{q_0(q_0-1)},
\]
where $\pi'$ is the set of prime divisors of $f = \log_pq$. 

We have now shown that 
\[
\mathcal{Q}(G,4)  =  \a_1+(1-\delta_{2,p})\a_2+\a_3+\a_4
\]
and it is straightforward to check that this is less than $1$ if $32 < q<10000$. Therefore, for the remainder of the proof we may assume that $q>10000$. (Note that $\mathcal{Q}(G,4)>1$ if $q=32$, which explains why we used  {\sc Magma} to handle the cases with $q \leqs 32$.)

If $q_0=2$ then $\pi' = \{r\}$ and  
\[
\a_4 = (r-1) \cdot \frac{27}{2} \cdot \frac{2^r(2^r-1)}{(2^r+1)^3},
\] 
which is less than $q^{-1/2}$ since $r >13$. Now assume $q_0 \geqs 3$. Here one checks that  
\[
\frac{(q_0+1)^3}{(q+1)^3} \cdot \frac{q(q-1)}{q_0(q_0-1)} < 4q^{-\left(1 - \frac{1}{r}\right)}
\]
and thus    
\[
(r-1) \cdot \frac{(q_0+1)^3}{(q+1)^3} \cdot \frac{q(q-1)}{q_0(q_0-1)} < 4q^{-\frac{1}{2}}
\]
for all $r \in \pi'$. We deduce that $\a_4 < 4q^{-1/2}\log\log q = \a_4'$ since  $|\pi'|<\log\log q$.

In conclusion, if $q>10000$ then
\[
\mathcal{Q}(G,4) < \a_1+\a_2 + \a_3' + \a_4' < 5q^{-\frac{1}{2}}\log\log q
\]
and the result follows. 
\end{proof}

\begin{lem}\label{l:l2p1}
If $G_0 = {\rm L}_{2}(q)$ and $H$ is of type $P_1$, then $b(G,H) \in \{3,4\}$ and $b(G,H) = 3$ if and only if 
\begin{itemize}\addtolength{\itemsep}{0.2\baselineskip}
\item[{\rm (i)}] $G \leqs {\rm PGL}_{2}(q)$; or 
\item[{\rm (ii)}] $q$ is odd, $f$ is even and $G = \la G_0, \delta\phi^{f/2} \ra = G_0.2$.
\end{itemize}
In addition, if $b(G,H) = 3$ then $\mathcal{P}(G,3) \to 1$ as $q \to \infty$.
\end{lem}

\begin{proof}
First observe that $\log |G| /\log |\O|>2$, so by combining Lemmas \ref{l:easy} and \ref{l:l2p10} we deduce that $b(G,H) \in \{3,4\}$. As before, we may identify $\O$ with the set of $1$-dimensional subspaces of the natural module $V$ for $G_0$. Given this identification, it is straightforward to check that   
\[
\{ \la e_1 \ra, \la e_2 \ra, \la e_1+e_2 \ra, \la e_1+\mu e_2 \ra\}
\]
is a base for $G$ of size $4$.

First assume $q$ is even, so $G_0$ is $3$-transitive on $\O$ and thus every $3$-point stabiliser in $G$ has order $|G:G_0|$. Therefore, $b(G,H) = 3$ if and only if $G = G_0$, in which case 
\begin{equation}\label{e:pG3}
\mathcal{P}(G,3) = \frac{|G|}{|\O|^3} = \frac{q(q^2-1)}{(q+1)^3}
\end{equation}
and we see that $\mathcal{P}(G,3) \to 1$ as $q \to \infty$. 

Now assume $q$ is odd. Let $\a,\b,\gamma \in \O$ be three distinct points and observe that $G_0$ is $2$-transitive, but not $3$-transitive on $\O$. Since ${\rm PGL}_{2}(q)$ is $3$-transitive, it follows that every $3$-point stabiliser in $G_0$ is trivial. Therefore, the $2$-point stabiliser $(G_0)_{\a,\b}$ has $4$ orbits on $\O$, namely $\{a\}$, $\{\b\}$ and two regular orbits $\Gamma_1$ and $\Gamma_2$, each of size $\frac{1}{2}(q-1)$. 

Suppose $G$ is $3$-transitive. Then $G_{\a,\b}$ is transitive on $\Gamma_1 \cup \Gamma_2$ and thus $|G_{\a,\b,\gamma}| = \frac{1}{2}|G:G_0|$. Therefore, $b(G,H) = 3$ if and only if $G = G_0.2$ is sharply $3$-transitive, which implies that either $G = {\rm PGL}_{2}(q)$, or $f$ is even and $G = \la G_0, \delta\phi^{f/2}\ra$ (note that $\phi$ fixes $\la e_1 \ra$, $\la e_2 \ra$ and $\la e_1+e_2 \ra$, so $\la G_0, \phi^{f/2}\ra$ is not $3$-transitive). In these cases, every $3$-point stabiliser is trivial and \eqref{e:pG3} holds. Finally, if $G$ is not $3$-transitive, then each $\Gamma_i$ is an orbit for $G_{\a,\b}$, so $|G_{\a,\b,\gamma}| = |G:G_0|$ and we deduce that $b(G,H) = 3$ if and only if $G = G_0$.
\end{proof}

\begin{lem}\label{l:l2base}
If $G_0 = {\rm L}_{2}(q)$ and $H$ is of type ${\rm GL}_{1}(q) \wr S_2$ or ${\rm GL}_{1}(q^2)$, then $b(G,H) \leqs 3$ and $\mathcal{P}(G,3) \to 1$ as $q \to \infty$.
\end{lem}

\begin{proof}
Here $H_0 = D_{2(q-\e)/d}$ and $|\O| = \frac{1}{2}q(q+\e)$, where $\e = 1$ if $H$ is of type ${\rm GL}_{1}(q) \wr S_2$, otherwise $\e=-1$. We proceed by  estimating the contributions to $\mathcal{Q}(G,3)$ from the various elements of prime order in $H$. Both cases are very similar, so for brevity we will assume that $H$ is of type ${\rm GL}_{1}(q) \wr S_2$. Let $x \in H$ be an element of prime order $r$ and recall that $i_2(H)$ denotes the number of involutions in $H$.

If $x$ is unipotent, then $r=p=2$, $|x^G| = q^2-1 = b_1$ and $|x^G \cap H| = i_2(H) = q-1 = a_1$. Similarly, if $x$ is a semisimple involution, then $|x^G| \geqs \frac{1}{2}q(q-1) = b_2$ and we note that $i_2(H) \leqs q = a_2$. Next suppose $x$ is semisimple and $r \geqs 3$, so $r$ divides $q-1$, $|x^{G_0}| = q(q+1)$ and $|x^{G_0} \cap H| = 2$. Since $G_0$ has $\frac{1}{2}(r-1) \leqs \frac{1}{2}(q-2)$ distinct conjugacy classes of such elements, it follows that the combined contribution to $\mathcal{Q}(G,3)$ from semisimple elements of odd order is at most
\[
\sum_{r \in \pi} \frac{1}{2}(r-1) \cdot \frac{8}{q^2(q+1)^2} < \frac{4(q-2)\log q}{q^2(q+1)^2} = \a_1,
\]
where $\pi$ is the set of odd prime divisors of $q-1$. 

Finally, suppose $q=q_0^r$ and $x$ is a field automorphism of order $r$. If $r=2$ then $|x^G| \geqs \frac{1}{2}q^{1/2}(q+1) = b_3$ and an easy calculation shows that $H$ contains at most $a_3 = 2q^{1/2}$ of these elements. Now assume $r$ is odd, so 
\[
|x^G \cap H| = \frac{q-1}{q_0-1},\;\; |x^G| = \frac{q(q^2-1)}{q_0(q_0^2-1)}
\]
and there are $r-1$ distinct conjugacy classes of field automorphisms of order $r$. If $q_0 = 2$ then $q=2^r$ and the contribution from field automorphisms is equal to 
\[
(r-1) \cdot \frac{36(2^r-1)}{2^{2r}(2^r+1)^2} < 2^{-r} = q^{-1}.
\]
Similarly, if $q_0 \geqs 3$ then the combined contribution from odd order field automorphisms is given by
\[
\sum_{r \in \pi'}(r-1) \cdot \frac{q_0^2(q_0+1)^2}{q^2(q+1)^2}\cdot \frac{q-1}{q_0-1} < \sum_{r \in \pi'}3(r-1)q^{-3\left(1-\frac{1}{r}\right)} < q^{-1}\log\log q = \a_2,
\]
where $\pi'$ is the set of odd prime divisors of $f = \log_pq$.

In conclusion,
\[
\mathcal{Q}(G,3) < \sum_{i=1}^{3}a_i^3/b_i^2 + \a_1 +\a_2 < 2q^{-\frac{1}{2}}
\]
for all $q>37$. In addition, this upper bound gives $\mathcal{Q}(G,3) < 1$ if $q>13$. The remaining groups with $q \leqs 13$ can be checked using {\sc Magma}.
\end{proof}

\begin{lem}\label{l:l2c2}
If $G_0 = {\rm L}_{2}(q)$ and $H$ is of type ${\rm GL}_{1}(q) \wr S_2$, then $b(G,H) \leqs 3$, with equality if and only if ${\rm PGL}_{2}(q) < G$.
\end{lem}

\begin{proof}
Here $H_0 = D_{2(q-1)/d}$, $|\O| = \frac{1}{2}q(q+1)$ and we may identify $\O$ with the set of distinct pairs of $1$-dimensional subspaces of $V$. By Lemma \ref{l:l2base}, we have $b(G,H) \leqs 3$. In fact, we claim that 
$\{\a,\b,\gamma\}$ is a base for $G$, where   
\[
\a = \{ \la e_1 \ra, \la e_2 \ra\},\; \b = \{ \la e_1 \ra, \la e_1+e_2 \ra\},\; \gamma = \{ \la e_1 \ra, \la e_1+\mu e_2 \ra\}.
\]
To see this, suppose $x = A \phi^j$ fixes $\a$, $\b$ and $\gamma$, where $A \in {\rm GL}_{2}(q)$ and $0 \leqs j < f$. We need to show that $A \in Z({\rm GL}_{2}(q))$ and $j=0$, which is a routine calculation. For example, one checks that $x$ fixes $\a$ and $\b$ if and only if $A \in Z({\rm GL}_{2}(q))$, and then it also fixes $\gamma$ if and only if $\mu = \mu^{p^j}$. Since $\mu$ is a generator for $\mathbb{F}_{q}^{\times}$, it follows that $j=0$ and this justifies the claim.

As explained in \cite[Example 2.5]{BG_saxl} (also see \cite[Table 2]{FI}), if $G = {\rm PGL}_{2}(q)$ then $H$ has a unique regular orbit on $\O$ and thus $b(G,H) = 2$. As an immediate consequence, we deduce that $b(G,H)= 3$ if ${\rm PGL}_{2}(q)<G$ (indeed, if $G_{\a}$ has a regular orbit, then the stabiliser of $\a$ in ${\rm PGL}_{2}(q)$ has at least $|G:{\rm PGL}_{2}(q)|$ regular orbits). Let us also observe that   
\[
\mathcal{P}({\rm PGL}_{2}(q),2) = \frac{|G|}{|\O|^2} = \frac{4(q-1)}{q(q+1)},
\]
which tends to $0$ as $q$ tends to infinity. 

Since ${\rm PGL}_{2}(q)$ has a trivial $2$-point stabiliser, we immediately deduce that $b(G,H) = 2$ if $G = G_0$. Moreover, by arguing as in the proof of \cite[Lemma 7.9]{BH_UDN} for example, one can show that if $q$ is odd then $(G_0)_{\a}$ has exactly $\frac{1}{4}(q+m)$ regular orbits, where $m = 7$ if $q \equiv 1 \imod{4}$, otherwise $m=5$. Therefore, if $q$ is odd then 
\[
\mathcal{P}({\rm L}_{2}(q),2) = \frac{(q-1)(q+m)}{2q(q+1)},
\]
which tends to $\frac{1}{2}$. 

Finally, to complete the proof we may assume that $q$ is odd and $G \cap {\rm PGL}_{2}(q) = G_0$. Here either $G = \la G_0, \phi^j \ra$ for some $j$ with $0 \leqs j < f$, or $G = \la G_0, \delta\phi^{j}\ra$ with $0<j<f$ and $f/(f,j)$ even. In both cases, we claim that $\{\a,\b\}$ is a base for $G$, where
\[
\a = \{\la e_1 \ra, \la e_2 \ra\},\;\; \b = \{\la e_1-e_2 \ra, \la e_1+\mu e_2\ra\}.
\]
To see this, let $x = AB^i\phi^j$, where $A \in {\rm SL}_{2}(q)$, $B = {\rm diag}(\mu,1) \in {\rm GL}_{2}(q)$ and either $i=0$ and $0 \leqs j < f$, or $1 \leqs i < q-1$ and $0<j<f$. It suffices to show that $x$ fixes $\a$ and $\b$ if and only if $A = \pm I_2$ and $i=j=0$. So let us assume $x$ fixes $\a$ and $\b$. Since $x$ fixes $\a$, it follows that $AB^i$ is either diagonal or anti-diagonal. 

Suppose $AB^i = {\rm diag}(a\mu^i,a^{-1})$ is diagonal. If $x$ fixes both spaces in $\b$, then 
\begin{align*}
(e_1-e_2)^x & = a\mu^ie_1-a^{-1}e_2 = \l(e_1-e_2) \\
(e_1+\mu e_2)^x & = a\mu^ie_1+a^{-1}\mu^{p^j}e_2 = \eta(e_1+\mu e_2)
\end{align*}
for some $\l,\eta \in \mathbb{F}_{q}^{\times}$. The first condition gives $a^2 = \mu^{-i}$ and using the second we deduce that $\mu^{p^j-1} =1$. Since $\mu$ has (multiplicative) order $q-1$, it follows that $j=0$ and thus $i=0$ and $a^2=1$, so $A = \pm I_2$ as required.
Similarly, if $x$ interchanges the two $1$-spaces in $\b$, then we deduce that $\mu^{p^j+1} = 1$, which contradicts the fact that $\mu$ has order $q-1$.

Now suppose $AB^i = \left(\begin{smallmatrix} 0 & a \\ -a^{-1}\mu^i & 0 \end{smallmatrix}\right)$ is anti-diagonal. If $x$ fixes both spaces in $\b$, then
\begin{align*}
(e_1-e_2)^x & = -ae_1-a^{-1}\mu^ie_2 = \l(e_1-e_2) \\
(e_1+\mu e_2)^x & = a\mu^{p^j}e_1-a^{-1}\mu^ie_2 = \eta(e_1+\mu e_2)
\end{align*}
for scalars $\l,\eta \in \mathbb{F}_{q}^{\times}$. These conditions imply that $\mu^{p^j+1} = 1$, which is a contradiction as above. Finally, if 
$x$ interchanges both spaces in $\b$ then we get $\mu^{i-1}=a^2$ and $\mu^{p^j-1}=1$. The latter condition implies that $j=0$, which forces $i=0$ and thus $\mu = a^{-2}$ is a square in $\mathbb{F}_q$. Once again we have reached a contradiction since $\mu$ is a generator for $\mathbb{F}_q^{\times}$.
\end{proof}

\begin{lem}\label{l:l2c3}
If $G_0 = {\rm L}_{2}(q)$ and $H$ is of type ${\rm GL}_{1}(q^2)$, then $b(G,H) \leqs 3$, with equality if and only if ${\rm PGL}_{2}(q) \leqs G$.
\end{lem}

\begin{proof}
Here $H_0 = D_{2(q+1)/d}$, $|\O| = \frac{1}{2}q(q-1)$ and Lemma \ref{l:l2base} gives $b(G,H) \leqs 3$ and $\mathcal{P}(G,3) \to 1$ as $q \to \infty$. The subdegrees for the action of ${\rm PGL}_{2}(q)$ are presented in \cite[Table 2]{FI} and we see that there is no suborbit of size $2(q+1)$. Therefore, $b(G,H) = 3$ if ${\rm PGL}_{2}(q) \leqs G$.

To complete the proof, we may assume that $q$ is odd and $G \cap {\rm PGL}_{2}(q) = G_0$. The subdegrees for the action of $G = G_0$ are computed in \cite[Lemma 7.9]{BH_UDN} and we deduce that $b(G,H) = 2$ and 
\[
\mathcal{P}(G,2) = \frac{(q+1)(q-m)}{2q(q-1)},
\]
where $m =1$ if $q \equiv 1 \imod{4}$ and $m=3$ if $q \equiv 3 \imod{4}$. In particular, $\mathcal{P}(G,2) \to \frac{1}{2}$ as $q \to \infty$. As in the proof of the previous lemma, it now remains to consider the following two cases:
\begin{itemize}\addtolength{\itemsep}{0.2\baselineskip}
\item[{\rm (a)}] $G = \la G_0, \phi^j\ra$ with $1 \leqs j < f$;
\item[{\rm (b)}] $G = \la G_0, \delta\phi^{j}\ra$ with $1 \leqs j < f$ and $f/(f,j)$ even.  
\end{itemize}

In both cases, we claim that $b(G,H) = 2$. To show this, it will be useful to identify $G_0$ with the unitary group $X_0 = {\rm U}_{2}(q)$ and $\O$ with the set of orthogonal pairs of non-degenerate $1$-dimensional subspaces of the natural module $U$ for $X_0$ over $\mathbb{F}_{q^2}$. Fix an orthonormal basis $\{u,v\}$ for $U$ with respect to the defining unitary form on $U$ and set $\a = \{\la u \ra, \la v \ra \} \in \O$. Observe that
\[
\O = \{\a\} \cup \{ \omega_{\xi} \,:\, \xi \in \mathbb{F}_{q^2}^{\times},\, \xi^{q+1} \ne -1\},
\]
where $\omega_{\xi} = \{ \la u+\xi v \ra, \la u-\xi^{-q}v\ra\}$. Note that $\omega_{\xi} = \omega_{-\xi^{-q}}$.

For the remainder of this proof, we will abuse notation by writing $\phi$ for the field automorphism of $X_0$ that corresponds to the map $\eta \mapsto \eta^{p}$ on $\mathbb{F}_{q^2}$. In particular, we will assume that  
\[
(au+bv)^{\phi} = a^pu + b^pv
\]
for all $a,b \in \mathbb{F}_{q^2}$. Now $X_0 \cap \la \phi \ra = \la \phi^f \ra$ and $\la X_0, \phi \ra = X_0.f$. With this set up, the two cases we need to consider are as described in (a) and (b) above, with $G_0$ replaced by $X_0$. Note that in (b),  the diagonal automorphism $\delta$ is the image of a diagonal matrix ${\rm diag}(\l^{q-1},1) \in {\rm GU}_{2}(q)$ with respect to the basis $\{u,v\}$ for $U$, where $\mathbb{F}_{q^2}^{\times} = \la \l \ra$. 

We claim that $\{\a,\b\}$ is a base for $G$, where 
\[
\b = \{ \la u+ \l v \ra, \la u - \l^{-q}v\ra\}.
\]
To see this, let $x = AB^i\phi^j$, where 
\[
A = \left(\begin{array}{cc}
a & b \\
c & d 
\end{array}\right) \in {\rm SU}_{2}(q),\;\; B = \left(\begin{array}{cc}
\l^{q-1} & 0 \\
0 & 1
\end{array}\right) \in {\rm GU}_{2}(q)
\]
and $0 \leqs j < 2f$ with $j \ne f$. In addition, assume that either $i=0$, or $1\leqs i < q+1$ and $0<j<2f$. Then to justify the claim, it suffices to show that $x$ fixes $\a$ and $\b$ if and only if  $A = \pm I_2$ and $i=j=0$.

Let us assume $x$ fixes $\a$ and $\b$. Since $x$ fixes $\a$, it is of the form  
\[
\left(\begin{array}{cc}
a\l^{i(q-1)} & 0 \\
0 & a^{-1} 
\end{array}\right)\phi^j \; \mbox{ or } \; \left(\begin{array}{cc}
0 & a \\
-a^{-1}\l^{i(q-1)} & 0 
\end{array}\right)\phi^j,
\]
according to whether or not $x$ fixes or interchanges the two $1$-spaces in $\a$. Note that $a^{q+1} = 1$ since $A \in {\rm SU}_{2}(q)$. This gives us two cases to consider.

Suppose $A$ is diagonal and $x$ fixes the two subspaces comprising $\b$. By direct calculation, we deduce that
\begin{equation}\label{e:a}
a^2 = \l^{p^j - i(q-1)-1} = \l^{q-qp^j-i(q-1)},
\end{equation}
whence $\l^{(q+1)(p^j-1)} = 1$ and thus $q^2-1$ divides $(q+1)(p^j-1)$ (recall that $\mathbb{F}_{q^2}^{\times} = \la \l \ra$). Since $j \ne f$ we immediately deduce that $j=0$ is the only possibility. Therefore $i=0$ (recall that $i \geqs 1$ only if $j>0$) and thus \eqref{e:a} implies that $a^2 = 1$, so $A = \pm I_2$. Similarly, if $A$ is diagonal and $x$ interchanges the spaces in $\b$, then $\l^{(q+1)(p^j+1)} = 1$ and this is incompatible with the fact that $\l$ has (multiplicative) order $q^2-1$.

Now assume $A$ is anti-diagonal. If $x$ fixes the two spaces in $\b$ then 
$\l^{(q+1)(p^j+1)} = 1$, which is a contradiction as above. On the other hand, if $x$ swaps the spaces in $\b$ then  
\[
a^2 = \l^{i(q-1)-p^j+q} = \l^{i(q-1)+qp^j-1}
\]
and thus $\l^{(q+1)(p^j-1)}=1$. As above, it follows that $i=j=0$ and thus $a^2 = \l^{q-1}$. But $a^{q+1}=1$ so we have $\l^{(q^2-1)/2}=1$ and once again we have reached a contradiction.

This justifies the claim and we conclude that $b(G,H) = 2$ in cases (a) and (b) above. This completes the proof of the lemma.
\end{proof}

\begin{lem}\label{l:l2sub}
Suppose $G_0 = {\rm L}_{2}(q)$, where $q = 3^k$ and $k$ is an odd prime. If $H$ is of type ${\rm GL}_{2}(3)$, then $b(G,H) = 2$ and 
$\mathcal{P}(G,2)\to 1$ as $q \to \infty$. 
\end{lem}

\begin{proof}
The case $q = 27$ can be checked using {\sc Magma}, so let us assume $q \geqs 3^5$. Here $H_0 = {\rm L}_{2}(3) \cong A_4$ and $|H| \leqs 24k = a_1$. Now $|x^G| \geqs \frac{1}{2}q(q-1) = b_1$ for all $x \in H$ of prime order (minimal if $x$ is an involution) and thus $\mathcal{Q}(G,2)<a_1^2/b_1$. It is routine to check that this upper bound is less than $q^{-1/2}$ if $q>3^5$ and it is less than $1$ if $q = 3^5$. 
\end{proof}

\begin{lem}\label{l:s4}
Suppose $G_0 = {\rm L}_{2}(q)$ and $H$ is of type $2^{1+2}_{-}.{\rm O}_{2}^{-}(2)$, where $q = p \geqs 7$. Then $b(G,H) = 2+\delta_{7,q}$ and 
$\mathcal{P}(G,2)\to 1$ as $q \to \infty$. 
\end{lem}

\begin{proof}
Here $q = p \geqs 7$ and $H_0 = A_4.c$, where $c=2$ if $p \equiv \pm 1 \imod{8}$, otherwise $c=1$ (see \cite[Proposition 4.6.7]{KL}). Therefore, $|H| \leqs 24 = a_1$ and we note that $|x^G|\geqs \frac{1}{2}q(q-1) = b_1$ for all $x \in H$ of prime order. This yields $\mathcal{Q}(G,2) \leqs a_1^2/b_1$, which is less than $q^{-1/2}$ if $q >109$, and it is less than $1$ if $q > 31$. The remaining cases with $q \leqs 31$ can be checked using {\sc Magma}. 
\end{proof}

This completes the proof of Proposition \ref{p:psl2}.

\section{Groups of Lie type: Parabolic actions}\label{s:parab}

To complete the proof of Theorem \ref{t:as}, we may assume $G$ is an almost simple group of Lie type over $\mathbb{F}_q$ with socle $G_0 \ne {\rm L}_{2}(q)$. We partition these groups into three collections according to $G_0$ and the structure of the maximal subgroup $H$. In this section, we consider the groups where $H$ is a parabolic subgroup; the remaining cases are handled in  Sections \ref{s:class} (classical groups) and \ref{s:excep} (exceptional groups).

\begin{rem}\label{r:isom}
In order to avoid unnecessary repetition, if $G_0$ is a classical group then we will assume it is one of the following:
\[
{\rm L}_{n}^{\e}(q),\, n \geqs 3; \; {\rm PSp}_{4}(q), \, n \geqs 4;\; {\rm P\O}_{n}^{\e}(q),\, n \geqs 7.
\]
In addition, we will assume that $G_0 \ne {\rm L}_{3}(2), {\rm L}_{4}(2), {\rm PSp}_{4}(2)'$ or ${\rm PSp}_{4}(3)$, which is justified by the existence of the following exceptional isomorphisms (see \cite[Proposition 2.9.1]{KL}):
\[
{\rm L}_{3}(2) \cong {\rm L}_{2}(7),\; {\rm L}_{4}(2) \cong A_8,\; {\rm PSp}_{4}(2)' \cong A_6,\; {\rm PSp}_{4}(3) \cong {\rm U}_{4}(2).
\]
Similarly, if $G_0$ is an exceptional group, then we will assume $G_0 \ne {}^2G_2(3)', G_2(2)'$ since ${}^2G_2(3)' \cong {\rm L}_{2}(8)$ and $G_2(2)' \cong {\rm U}_{3}(3)$.
\end{rem}

The main result of this section is the following.  

\begin{prop}\label{p:par}
Let $G \leqs {\rm Sym}(\O)$ be a finite almost simple primitive group with socle $G_0$ and soluble point stabiliser $H$. Set $b=b(G,H)$ and assume $G_0 \ne {\rm L}_{2}(q)$ is a group of Lie type and $H$ is a maximal parabolic subgroup of $G$.  
\begin{itemize}\addtolength{\itemsep}{0.2\baselineskip}
\item[{\rm (i)}] We have $3 \leqs b \leqs 5$, with $b=5$ if and only if $G_0 = {\rm L}_{4}(3)$ and $H$ is of type $P_2$, or $G_0 = {\rm U}_{5}(2)$ and $H$ is of type $P_1$.
\item[{\rm (ii)}] The precise value of $b$ is recorded in Tables \ref{tab:as2} ($G_0$ exceptional) and  \ref{tab:as3} ($G_0$ classical).
\end{itemize}
In addition, $\mathcal{P}(G,3) \to 1$ as $|G| \to \infty$.
\end{prop}

\begin{rem}\label{r:nota}
We adopt the standard notation from \cite{KL} for maximal parabolic subgroups. In particular, if $G_0$ is a classical group with natural module $V$, then $P_m$ denotes the stabiliser of an $m$-dimensional totally singular subspace of $V$. Similarly, if $G_0 = {\rm L}_{n}(q)$, then $P_{m,n-m}$ is the stabiliser of a flag $0<W<U<V$, where $\dim W = m < n/2$ and $\dim U = n-m$. If $G_0 = {\rm P\O}_{8}^{+}(q)$ then we write $P_{1,3,4}$ for a parabolic subgroup $H$ of $G$ such that $H \cap G_0 = L/Z$ and  
\[
L = [q^{11}]{:}[(q-1)/d]^2.\frac{1}{d}{\rm GL}_{2}(q).d^2 < \O_{8}^{+}(q)
\]
with $d = (2,q-1)$ and $Z = Z(\O_{8}^{+}(q))$. Note that in this case, $H$ is maximal and soluble if and only if $q \in \{2,3\}$ and $G \not\leqs {\rm PGO}_{8}^{+}(q)$ (see \cite[Table 8.50]{BHR}).
\end{rem}

To get started, we first determine the cases that we need to consider. As before, we set $H_0= H \cap G_0$. In Table \ref{tab:par}, we write $\phi$ for a field automorphism of $G_0$ of order $f = \log_pq$.

\begin{lem}\label{l:par1}
Let $G$ be a finite almost simple group of Lie type over $\mathbb{F}_q$ with socle $G_0 \ne {\rm L}_{2}(q)$ and a soluble maximal parabolic subgroup $H$. Then one of the following holds:
\begin{itemize}\addtolength{\itemsep}{0.2\baselineskip}
\item[{\rm (i)}] $G_0 \in \{ {\rm L}_n^{\e}(q), {\rm L}_{6}(q), {\rm PSp}_{6}(q), \O_7(q), {\rm P\O}_{8}^{+}(q)\}$ with $n \leqs 5$ and $q \in \{2,3\}$.
\item[{\rm (ii)}] $G_0$ is an exceptional group and one of the following holds:

\vspace{1mm}

\begin{itemize}\addtolength{\itemsep}{0.2\baselineskip}
\item[{\rm (a)}] $G = G_2(3)$ and $H=[3^5]{:}{\rm GL}_{2}(3)$.
\item[{\rm (b)}] $G_0 = {}^3D_4(q)$, $H_0 = [q^{11}]{:}((q^3-1) \circ {\rm SL}_{2}(q)).(2,q-1)$ and $q \in \{2,3\}$.
\item[{\rm (c)}] $G_0 = {}^2F_4(2)'$ and $H_0 = [2^9]{:}5{:}4$ or $[2^{10}]{:}S_3$.
\item[{\rm (d)}] $G = F_4(2).2$ and $H = [2^{22}]{:}S_3^2.2$.
\end{itemize}
\item[{\rm (iii)}] $(G,H)$ is one of the cases recorded in Table \ref{tab:par}.
\end{itemize}
\end{lem}

\begin{table}
\[
\begin{array}{clll} \hline
\mbox{Case} & G_0 & \mbox{Type of $H$} & \mbox{Conditions} \\ \hline
{\rm (a)} & {\rm L}_{3}(q) & P_{1,2} & G \not\leqs \la {\rm PGL}_{3}(q), \phi \ra \\
{\rm (b)} & {\rm U}_{3}(q) & P_1 & \\
{\rm (c)} & {\rm Sp}_{4}(q) & [q^4]{:}C_{q-1}^2 & \mbox{$q=2^f \geqs 4$ and $G \not\leqs \la G_0, \phi \ra$} \\
{\rm (d)} & G_2(q) & [q^6]{:}C_{q-1}^2 & \mbox{$q=3^f \geqs 3$ and  $G \not\leqs \la G_0, \phi \ra$} \\
{\rm (e)} & {}^2B_2(q) & [q^{2}]{:}C_{q-1} & q=2^{2m+1} \geqs 8 \\
{\rm (f)} & {}^2G_2(q) & [q^{3}]{:}C_{q-1} & q=3^{2m+1} \geqs 27 \\ \hline
\end{array}
\]
\caption{Parabolic actions}
\label{tab:par}
\end{table}

\begin{proof}
This follows by inspection of \cite[Tables 16-19]{LZ} for $G_0$ classical and \cite[Table 20]{LZ} for $G_0$ exceptional. 
\end{proof}

\begin{prop}\label{p:par1}
Proposition \ref{p:par} holds in cases (i) and (ii) of Lemma \ref{l:par1}.
\end{prop}

\begin{proof}
For the case in part (ii)(d), \cite[Theorem 3]{BLS} gives $b(G,H) = 3$ (here $H = P_{2,3}$ in the notation of \cite{BLS}). All of the remaining groups can be handled using {\sc Magma} (see Section \ref{ss:comp}). 
\end{proof}

For the remainder of this section, we may assume $(G,H)$ belongs to one of the infinite families recorded in Table \ref{tab:par}. Notice that in each case, $H = N_G(P)$ where $P$ is a Sylow $p$-subgroup of $G_0$. As before, if $G_0$ is a classical group then we refer the reader to \cite[Section 3]{BG} for information on the conjugacy classes of elements of prime order in ${\rm Aut}(G_0)$.

\begin{lem}\label{l:psl3}
Suppose $G_0 = {\rm L}_{3}(q)$ and $H$ is of type $P_{1,2}$. Then either $b(G,H)=3$, or $G = {\rm L}_{3}(4).D_{12}$ and $b(G,H)=4$. Moreover, $\mathcal{P}(G,3) \to 1$ as $q \to \infty$.
\end{lem}

\begin{proof}
Write $q=p^f$ and set $d=(3,q-1)$. As recorded in Table \ref{tab:par}, the maximality of $H$ implies that $G$ contains graph or graph-field automorphisms of $G_0$. We have 
\[
H_0  = [q^3]{:}[(q-1)^2/d],\;\;  |\O| = (q^2+q+1)(q+1)
\]
and one checks that $\log |G| / \log |\O| >2$ (recall that $q \geqs 3$). The cases with $q \leqs 2^7$ can be checked using {\sc Magma}. (Note that if $q \geqs 5$, then it suffices to show that $b(G,H) \leqs 3$ for $G = {\rm Aut}(G_0)$, which is easily checked by random search, noting that $H = N_G(P)$ for a Sylow $p$-subgroup $P$ of $G_0$.) For the remainder of the proof, we will assume that $q > 2^7$. Our aim is to show that $\mathcal{Q}(G,3)<1$ (and also $\mathcal{Q}(G,3) \to 0$ as $q$ tends to infinity).

First assume $x \in G_0$ is an element of prime order $r$. Let $\chi$ be the corresponding permutation character of $G_0$, so $\chi(x) = |C_{\O}(x)|$. The character table of $G_0$ is presented in \cite[Table 2]{SF} and we observe that 
\[
\chi = \chi_1 + 2\chi_{q(q+1)} + \chi_{q^3}
\]
as a sum of unipotent characters (in the notation of \cite[Table 2]{SF}). Here $\chi_1$ and $\chi_{q^3}$ are the trivial and Steinberg characters of $G_0$, respectively.

Suppose $x$ is unipotent and let $J_i$ denote a standard unipotent Jordan block of size $i$. If $x$ has Jordan form $[J_2,J_1]$ on the natural module, then we read off $\chi(x) = 2q+1$ (note that there is an error in \cite[Table 2]{SF}: the Steinberg character vanishes on all nontrivial unipotent elements, so $\chi_{q^3}(x) = 0$ and not $q$ as stated in the table). Similarly, 
$\chi(x) = 1$ if $x = [J_3]$. For $x = [J_2,J_1]$ we have $|x^{G_0}| = (q+1)(q^3-1)$ and we deduce that $|x^{G_0} \cap H_0| = 2q^2-q-1$. On the other hand, if $x$ is regular then $|x^{G_0}| = q(q^2-1)(q^3-1)/d$ and we get $|x^{G_0} \cap H_0| = q(q-1)^2/d$. Therefore, the combined contribution to $\mathcal{Q}(G,3)$ from unipotent elements is 
\[
\a = \frac{(q^2-q-1)^3}{(q+1)^2(q^3-1)^2} + \frac{q^3(q-1)^6}{q^2(q^2-1)^2(q^3-1)^2} < q^{-2}.
\]

Next assume $x \in G_0$ is semisimple and note that we may assume $r$ divides $q-1$ (otherwise $x^G \cap H$ is empty). If $r=2$ then $|x^{G_0}| = q^2(q^2+q+1)$ and $\chi(x) = 3(q+1)$, which gives $|x^{G_0} \cap H_0| = 3q^2$. Therefore, the contribution from semisimple involutions is equal to
\[
\b_1 = \frac{(3q^2)^3}{q^4(q^2+q+1)^2} = \frac{27q^2}{(q^2+q+1)^2}.
\]

If $x$ is regular then $|x^{G_0}| = q^3(q+1)(q^2+q+1)$ and $\chi(x) = 6$, so $|x^{G_0} \cap H_0| = 6q^3$. Let $n(r)$ be the number of $G_0$-classes of regular semisimple elements of order $r$. Then $n(3) = 1$ and $n(r) = \frac{1}{6}(r-1)(r-2)$ if $r \geqs 5$. Therefore, the combined contribution to $\mathcal{Q}(G,3)$ from these elements is equal to
\[
\b_2 = \left(\delta + \frac{1}{6}\sum_{r \in \pi} (r-1)(r-2) \right) \cdot \frac{(6q^3)^3}{(q^3(q+1)(q^2+q+1))^2}, 
\]
where $\pi$ is the set of primes $r \geqs 5$ dividing $q-1$ and we set $\delta = 1$ if $d = 3$, otherwise $\delta = 0$. Similarly, if $x \in G_0$ is non-regular then $|x^{G_0}| = q^2(q^2+q+1)$ and $\chi(x) = 3(q+1)$, which gives $|x^{G_0} \cap H_0| = 3q^2$. Since there are $r-1$ distinct $G_0$-classes of such elements if $r \geqs 5$ (and none if $r=3$), the contribution here is equal to 
\[
\b_3 = \left(\delta+\sum_{r \in \pi}(r-1) \right)\cdot \frac{27q^2}{(q^2+q+1)^2},
\]
where $\delta$ and $\pi$ are defined as above. Therefore, the combined contribution from all semisimple elements in $G_0$ is equal to $\b_0 = (1-\delta_{2,p})\b_1+\b_2+\b_3$. 

For $2^7< q< 1000$, we calculate that $\b_0 < \frac{1}{7}$. Now assume $q > 1000$. If $q-1$ is a prime, then $q$ is even, $\delta = 0$, $\pi = \{q-1\}$ and it is routine to check that $\b_0<70q^{-1}$.
Now assume $q-1$ is composite. Since $|\pi| < \log q$ and $r \leqs \frac{1}{2}(q-1)$, we deduce that 
\[
\delta + \frac{1}{6}\sum_{r \in \pi} (r-1)(r-2) < 1 + \frac{1}{24}(q-3)(q-5)\log q
\]
and
\[
\delta+\sum_{r \in \pi}(r-1) < 1+\frac{1}{2}(q-3)\log q.
\]
These estimates yield upper bounds on $\b_2$ and $\b_3$ and one checks that $\b_0< 250q^{-1}$.

To complete the analysis of semisimple elements, it remains to consider the contribution from elements of order $3$ in ${\rm PGL}_{3}(q) \setminus G_0$, so let us assume $d=3$. There are four $G_0$-classes of such elements; two of the classes are represented by elements that are the images of non-regular elements of order $3$ in ${\rm GL}_{3}(q)$, while the latter two are the images of elements of order $9$ that do not fix any $1$-spaces over $\mathbb{F}_q$ (in particular, $x^G \cap H$ is empty for these elements). If $x$ is the image of a non-regular element of order $3$ then $|x^{G_0}| = q^2(q^2+q+1)$ and we calculate that $|x^{G_0} \cap H| = 3q^2$ (this can be computed directly and it also follows from the fact that $\chi(y) = 3(q+1)$ for all non-regular semisimple elements $y \in G_0$), so the contribution from these elements is equal to $2\b_1$. 

Therefore, the entire contribution to $\mathcal{Q}(G,3)$ from semisimple elements is equal to 
\[
\b = (1 - \delta_{2,p}+2\delta_{3,d})\b_1+\b_2+\b_3
\]
and we conclude that $\b<\frac{1}{7}$ if $2^7<q < 1000$ and $\b<250q^{-1}$ if $q > 1000$.

Next assume $x \in G$ is a field automorphism of prime order $r$, so $q=q_0^r$. Set $\tilde{G} = {\rm PGL}_{3}(q)$ and $\tilde{H} = [q^3]{:}C_{q-1}^2 = N_{\tilde{G}}(P)$. Then
\[
|x^{\tilde{G}}| = \frac{q^3(q^2-1)(q^3-1)}{q^{3/r}(q^{2/r}-1)(q^{3/r}-1)} = f(q,r)
\]
and as noted in the proof of \cite[Lemma 6.1]{LLS2}, we have
\[
|x^{\tilde{G}} \cap \tilde{H}x| = \frac{q^3(q-1)^2}{q^{3/r}(q^{1/r}-1)^2} = g(q,r).
\]
Therefore, the contribution to $\mathcal{Q}(G,3)$ from field automorphisms is
\[
\varphi=\sum_{r \in \pi} (r-1) \cdot g(q,r)^3f(q,r)^{-2},
\]
where $\pi$ is the set of prime divisors of $\log_pq=f$. One checks that $\varphi<\frac{1}{2}$ if $2^7<q<10000$, so let us assume $q > 10000$. (It is worth noting here that $\varphi>1$ if $q=2^7$, which explains why we used {\sc Magma} to handle this case.) Set $e(q,r) = (r-1) \cdot g(q,r)^3f(q,r)^{-2}$.

If $q_0 \in \{2,3\}$ then $\varphi = e(q_0^r,r) < 3q^{-1/2}$. Now assume $q_0 \geqs 4$ and observe that 
\[
|x^{\tilde{G}} \cap \tilde{H}x| < 2q^{5\left(1-\frac{1}{r}\right)},\;\; |x^{\tilde{G}}|>q^{8\left(1-\frac{1}{r}\right)}
\]
and thus 
\[
e(q,r) < (r-1)\cdot 8q^{-\left(1-\frac{1}{r}\right)} = 8(r-1)q_0^{-(r-1)}.
\]
For $r \geqs 3$, this implies that $e(q,r)<8q^{-1/2}$ and direct calculation gives $e(q,2) < 2q^{-1/2}$. Since $|\pi| < \log\log q$, we conclude that 
\[
\varphi < 8q^{-\frac{1}{2}}\log\log q
\]
for $q>10000$.

Next suppose $x \in G$ is an involutory graph-field automorphism. Here $q=q_0^2$,  
\[
|x^{\tilde{G}}\cap \tilde{H}x| = \frac{q^3(q-1)^2}{q^{3/2}(q-1)} = q^{3/2}(q-1)
\]
(since a Borel subgroup of $C_{\tilde{G}}(x) = {\rm PGU}_{3}(q^{1/2})$ has order $q^{3/2}(q-1)$) and
\[
|x^{\tilde{G}}| = \frac{q^3(q^2-1)(q^3-1)}{q^{3/2}(q-1)(q^{3/2}+1)} = q^{3/2}(q+1)(q^{3/2}-1).
\]
Therefore, the contribution from these elements is equal to 
\[
\frac{|x^{\tilde{G}}\cap \tilde{H}x|^3}{|x^{\tilde{G}}|^2} = \frac{q^{3/2}(q-1)^3}{(q+1)^2(q^{3/2}-1)^2} < q^{-\frac{1}{2}}.
\]

Finally, let us assume $x$ is an involutory graph automorphism of $G_0$. Without loss of generality, replacing $x$ by a conjugate if necessary, we may assume that $x$ is the inverse-transpose map. We claim that $|C_{\O}(x)| = q+1$, which implies that $|x^{\tilde{G}} \cap \tilde{H}| = q^2(q-1)$. Since $|x^{\tilde{G}}| = q^2(q^3-1)$, it follows that the contribution from graph automorphisms is at most 
\[
\frac{(q^2(q-1))^3}{(q^2(q^3-1))^2} = \frac{q^2(q-1)}{(q^2+q+1)^2} < q^{-1}.
\] 

To establish the claim, it is helpful to identify $\O$ with the set of flags $0< U < W < V$ of the natural module $V$ for $G_0$. Let $\{e_1,e_2,e_3\}$ be a basis for $V$. Now $x$ maps the $1$-space $U = \la u \ra$ to the $2$-space $U^{\perp} = \{ v \in V \,:\, u^{T}v = 0 \mbox{ for all } u \in U\}$. Therefore, $x$ fixes a flag $0< U < W < V$ if and only if $U < U^{\perp}$, whence $|C_{\O}(x)|$ is the number of $1$-spaces $\la a_1e_1+a_2e_2+a_3e_3 \ra$ with $a_1^2+a_2^2+a_3^2 = 0$. 

If $q$ is even, then $a_1^2+a_2^2+a_3^2 = 0$ if and only if $a_3 = a_1+a_2$, so there are $q^2-1$ choices for $a_1e_1+a_2e_2+a_3e_3$ and thus $q+1$ distinct $1$-spaces with the desired property. For $q$ odd, we see that 
$|C_{\O}(x)|$ is the number of totally isotropic $1$-spaces in a $3$-dimensional orthogonal space. Therefore, $|C_{\O}(x)| = |{\rm SO}_{3}(q):L|$ where $L$ is a Borel subgroup of ${\rm SO}_{3}(q)$, which once again gives  $|C_{\O}(x)| = q+1$ as claimed. 

We conclude that if $2^7 < q < 10000$, then 
\[
\mathcal{Q}(G,3) < \frac{1}{2}+\frac{1}{7} + q^{-\frac{1}{2}} + q^{-1} + q^{-2} < 1
\]
and thus $b(G,H) = 3$. Similarly, if $q>10000$ then the above estimates imply that 
\[
\mathcal{Q}(G,3) < (1+8\log\log q)q^{-\frac{1}{2}} + 251q^{-1} + q^{-2}
\]
and the result follows.
\end{proof}

\begin{lem}\label{l:psu3}
If $G_0 = {\rm U}_{3}(q)$ and $H$ is of type $P_{1}$, then $b(G,H)=3$ and $\mathcal{P}(G,3) \to 1$ as $q \to \infty$.
\end{lem}

\begin{proof}
This is very similar to the proof of the previous lemma. Write $q=p^f$ and $d=(3,q+1)$. Note that $q \geqs 3$ and
\[
H_0 = [q^3]{:}C_{(q^2-1)/d},\;\; |\O| = q^3+1.
\]
We have $\log |G| /\log |\O| > 2$, so $b(G,H) \geqs 3$. The cases with $q \leqs 8$ can be checked using {\sc Magma}, so for the remainder of the proof we will assume that $q>8$.

Let $\chi$ be the corresponding permutation character of $G_0$. The character table of $G_0$ is given in \cite[Table 2]{SF} and we observe that
\[
\chi = \chi_1 + \chi_{q^3}
\]
is the sum of the trivial and Steinberg characters of $G_0$. Let $x \in G_0$ be an element of prime order $r$.

If $x$ is unipotent then $\chi(x) = 1+0$ (as noted in the proof of the previous lemma, there is a misprint in \cite[Table 2]{SF}), so $|x^{G_0}\cap H_0| = q-1$ if $x = [J_2,J_1]$ and $|x^{G_0} \cap H_0| = q(q^2-1)/d$ if $x = [J_3]$. It follows that the contribution to $\mathcal{Q}(G,3)$ from unipotent elements is 
\[
\frac{(q-1)^3}{(q-1)^2(q^3+1)^2} + \frac{q^3(q^2-1)^3}{q^2(q^2-1)^2(q^3+1)^2} < q^{-3}.
\]

Next suppose $x$ is semisimple and note that we may assume $r$ divides $q^2-1$. If $r=2$ then $|x^{G_0}| = q^2(q^2-q+1)$ and $\chi(x) = q+1$, which gives $|x^{G_0} \cap H_0| = q^2$. Therefore, the contribution from semisimple involutions is equal to
\[
\b_1 = \frac{q^6}{q^4(q^2-q+1)^2} = \frac{q^2}{(q^2-q+1)^2}.
\]
Now assume $r \geqs 3$. Let $n(r)$ be the number of $G_0$-classes of regular semisimple elements of order $r$. If $r$ divides $q-1$ then $x$ is regular, $|x^{G_0}| = q^3(q^3+1)$, $n(r) = \frac{1}{2}(r-1)$ and $\chi(x) = 2$, which gives $|x^{G_0} \cap H_0| = 2q^3$. Therefore, 
\[
\b_2 = \frac{1}{2}\sum_{r \in \pi} (r-1) \cdot \frac{8q^3}{(q^3+1)^2} <  \frac{4(q-2)q^3\log q}{(q^3+1)^2} =\b_2'
\]
is the contribution from these elements, 
where $\pi$ is the set of primes $r \geqs 3$ dividing $q-1$. 

Now assume $r$ divides $q+1$. For now, let us also assume that $r \geqs 5$. If $x$ is regular, then $\chi(x) = 0$ so we may assume $x$ is non-regular. Then $|x^{G_0}| = q^2(q^2-q+1)$, $n(r) = r-1$ and $\chi(x) = q+1$, so $|x^{G_0} \cap H_0| = q^2$ and the contribution from these elements is equal to
\[
\b_3 = \sum_{r \in \pi'} (r-1) \cdot \frac{q^2}{(q^2-q+1)^2} < \frac{q^2\log q}{(q^2-q+1)^2} = \b_3',
\]
where $\pi'$ is the set of primes $r \geqs 5$ dividing $q+1$.

To complete the analysis of semisimple elements, let us assume $r=d=3$. Suppose $x \in G_0$ and observe that $|H_0|$ is divisible by $3$ if and only if $q \equiv -1 \imod{9}$. So let us assume $q \equiv -1 \imod{9}$. If $x$ is regular, then $\chi(x)=0$. There are also two non-regular classes of elements $x \in G_0$ of order $3$ with $|x^{G_0}| = q^2(q^2-q+1)$ and $\chi(x) = q+1$ (in the notation of \cite[Table 2]{SF}, these elements are of type $C_4^{(k)}$). Here we get $|x^{G_0} \cap H_0| = q^2$. In addition, there are two classes of elements of order $3$ in ${\rm PGU}_{3}(q) \setminus G_0$, but none of them fix a $1$-dimensional subspace of the natural module for $G_0$ (indeed, on lifting to ${\rm GU}_{3}(q)$, none of these elements have an eigenvalue in $\mathbb{F}_{q^2}$). It follows that the total contribution to $\mathcal{Q}(G,3)$ from elements of order $3$ when $d=3$ is at most  
\[
\b_4 = \frac{2q^2}{(q^2-q+1)^2}.
\]

We conclude that the combined contribution from semisimple elements is less than
\[
\b_1+\b_2'+\b_3'+\b_4 = \frac{(3+\log q)q^2}{(q^2-q+1)^2} +  \frac{4(q-2)q^3\log q}{(q^3+1)^2} < 2q^{-1}
\]
for all $q > 8$.

Next let us assume $x$ is a field automorphism of prime order $r$, so $q=q_0^r$ and $r$ is odd. Set $\tilde{G} = {\rm PGU}_{3}(q)$ and $\tilde{H} = [q^3]{:}C_{q^2-1} = N_{\tilde{G}}(P)$, where $P$ is a Sylow $p$-subgroup of $G_0$. Then
\[
|x^{\tilde{G}}| = \frac{q^3(q^2-1)(q^3+1)}{q^{3/r}(q^{2/r}-1)(q^{3/r}+1)} = f(q,r) > q^{8\left(1-\frac{1}{r}\right)} 
\]
and
\[
|x^{\tilde{G}} \cap \tilde{H}x| = \frac{q^3(q^2-1)}{q^{3/r}(q^{2/r}-1)} = g(q,r) < 2q^{5\left(1-\frac{1}{r}\right)},
\]
so the total contribution to $\mathcal{Q}(G,3)$ from field automorphisms is
\[
\varphi=\sum_{r \in \pi} (r-1) \cdot g(q,r)^3f(q,r)^{-2},
\]
where $\pi$ is the set of odd prime divisors of $\log_pq=f$. One checks that $\varphi<\frac{1}{2}$ if $8<q<10000$, so let us assume $q > 10000$. Set $e(q,r) = (r-1) \cdot g(q,r)^3f(q,r)^{-2}$, so  
\[
e(q,r) < (r-1)\cdot 8q^{-\left(1-\frac{1}{r}\right)} = 8(r-1)q_0^{-(r-1)} < 8q^{-\frac{1}{2}}
\]
and we conclude that 
\[
\varphi < 8q^{-\frac{1}{2}}\log\log q
\]
for $q>10000$. 
 
Finally, let us assume $x \in G$ is an involutory graph automorphism. Fix a standard unitary basis $\{e_1,v, f_1\}$ for the natural module $V$, where $e_1$ and $f_1$ are isotropic, $(e_1,f_1) = (v,v)=1$ and $(e_1,v) = (f_1,v) = 0$ with respect to the defining unitary form $(\, , \,)$ on $V$. It will be convenient to identify $\O$ with the set of totally isotropic $1$-dimensional subspaces of $V$. Without loss of generality, we may assume that $x$ corresponds to the involutory automorphism of $\mathbb{F}_{q^2}$, so $x$ sends the subspace $\la ae_1+bv+cf_1 \ra$ of $V$ to $\la a^qe_1+b^qv+c^qf_1\ra$. The $1$-space $\la ae_1+bv+cf_1 \ra$ is totally isotropic if and only if $ac^q+b^{q+1}+ca^q = 0$, and it is fixed by $x$ if and only if $a,b,c \in \mathbb{F}_q$. Therefore, $|C_{\O}(x)|$ is equal to the number of $1$-spaces $\la ae_1+bv+cf_1 \ra$ with $a,b,c \in \mathbb{F}_q$ and $2ac+b^2 = 0$. 

If $q$ is even then $b=0$ and there are $(q^2-1)/(q-1) = q+1$ choices for $(a,c)$, whence $|C_{\O}(x)|=q+1$. Now assume $q$ is odd. If $b=0$ then either $a$ or $c$ is $0$, so $\la e_1 \ra$ and $\la f_1 \ra$ are the only options. If $b \ne 0$, then we may assume $b=1$ by scaling, so $ac = -\frac{1}{2}$ and there are $q-1$ possibilities for $(a,c)$. So once again we get $|C_{\O}(x)| = 2+(q-1) = q+1$. Now $|x^{\tilde{G}}| = q^2(q^3+1)$ and it follows that $|x^{\tilde{G}} \cap H| =  q^2(q+1)$. Therefore, the contribution to $\mathcal{Q}(G,3)$ from graph automorphisms is equal to 
\[
\frac{q^2(q+1)}{(q^2-q+1)^2}, 
\]
which is less than $2q^{-1}$ for $q>8$.

To conclude, we observe that the above estimates imply that 
\[
\mathcal{Q}(G,3) < \frac{1}{2}+4q^{-1}+q^{-3} < 1
\]
if $8<q<10000$ and 
\[
\mathcal{Q}(G,3) < 8q^{-\frac{1}{2}}\log\log q+ 4q^{-1}+q^{-3}
\]
if $q>10000$. The result follows.
\end{proof}

\begin{lem}\label{l:sp4}
If $G_0 = {\rm Sp}_{4}(q)$ and $H$ is of type $[q^4]{:}C_{q-1}^2$, then $b(G,H) = 3$ and $\mathcal{P}(G,3) \to 1$ as $q \to \infty$.
\end{lem}

\begin{proof}
Here $q = 2^f \geqs 4$ and the maximality of $H$ implies that $G$ contains graph automorphisms. We have 
\[
H_0 = [q^4]{:}C_{q-1}^2, \;\; |\O| = (q+1)^2(q^2+1),
\] 
so $\log |G| / \log |\O|>2$ and thus $b(G,H) \geqs 3$. For $q \leqs 32$, it is easy to check that $b(G,H) \leqs 3$ using {\sc Magma} \cite{magma}. For the remainder, we may assume that $q \geqs 64$. 

Write $G_0 = \bar{G}_{\s} = {\rm Sp}_{4}(q)$, where $\bar{G} = {\rm Sp}_{4}(k)$, $k$ is the algebraic closure of $\mathbb{F}_2$ and $\s$ is a Steinberg endomorphism of $\bar{G}$. Set $H_0 = H \cap G_0$. Then $H_0 = \bar{H}_{\s}$, where $\bar{H}$ is a $\s$-stable Borel subgroup of $\bar{G}$, and we fix a $\s$-stable maximal torus $\bar{T}$ of $\bar{G}$ contained in $\bar{H}$ such that $\bar{T}_{\s} = C_{q-1}^2$. Let $\chi$ be the permutation character corresponding  to the action of $G_0$ on $\O$. Since $H_0$ is a Borel subgroup of $G_0$, it follows that $\chi = R^{\bar{G}}_{\bar{T}}(1_{\bar{T}})$ is the Deligne-Lusztig character of $G_0$ corresponding to the trivial conjugacy class in the Weyl group $W = D_8$ of $G_0$. By adopting the notation in \cite[Table 2.8]{GM}, we can express
\begin{equation}\label{e:decomp}
\chi = \theta_0 + 2\theta_9 + \theta_{11} + \theta_{12} + \theta_{13}
\end{equation}
as a sum of unipotent characters of $G_0$ (in this notation, $\theta_0$ and $\theta_{13}$ are the trivial and Steinberg characters of $G_0$). Let $x \in G_0$ be an element of prime order $r$.

First assume $r=2$, so $x$ is of type $b_1$, $a_2$ or $c_2$ with respect to  the notation in \cite{AS}. Since $G$ contains graph automorphisms, we note that $b_1$ and $a_2$ are $G$-conjugate. The values of the unipotent characters in \eqref{e:decomp} at unipotent elements are recorded in \cite[Table 2.10]{GM}. If $x$ is of type $b_1$ or $a_2$, then $|x^G| = 2(q^4-1)$ and $\chi(x) = (q+1)^2$, which implies that $|x^{G} \cap H| = 2(q^2-1)$. Similarly, if $x$ is of type $c_2$ then $|x^G| = (q^2-1)(q^4-1)$ and $\chi(x) = 2q+1$, which gives $|x^G \cap H| = (2q+1)(q-1)^2$. We conclude that the contribution to $\mathcal{Q}(G,3)$ from unipotent elements is precisely 
\[
\frac{(2(q^2-1))^3}{(2(q^4-1))^2} + \frac{((2q+1)(q-1)^2)^3}{((q^2-1)(q^4-1))^2} < 2q^{-2}.
\]

Now assume $r$ is odd and divides $q-1$ (note that $x^G \cap H$ is empty if $r$ does not divide $q-1$). Suppose $x$ is regular, so $r \geqs 5$, $|x^{G_0}| = q^4(q+1)^2(q^2+1)$ and $C_{\bar{G}}(x)$ is a maximal torus. Here we calculate that $\chi(x) = |W| = 8$ (for example, this is easily computed via \cite[Lemma 2.2.23]{GM}), which gives $|x^{G_0} \cap H_0| = 8q^4$. Since there are $\binom{(r-1)/2}{2} = \frac{1}{8}(r-1)(r-3)$ distinct $G_0$-classes of regular semisimple elements of order $r$, it follows that the combined contribution from regular semisimple elements is precisely
\[
\sum_{r \in \pi} \frac{1}{8}(r-1)(r-3) \cdot \frac{(8q^4)^3}{(q^4(q+1)^2(q^2+1))^2} = \sum_{r \in \pi} (r-1)(r-3) \cdot \frac{64q^4}{(q+1)^4(q^2+1)^2},
\]
where $\pi$ is the set of odd prime divisors of $q-1$. Since $|\pi| < \log q$ and $r \leqs q-1$, this is at most
\[
\b_1 = \frac{64(q-2)(q-4)q^4\log q}{(q+1)^4(q^2+1)^2}.
\]

Now assume $r$ is odd and $x$ is non-regular, so $|x^{G_0}| = q^3(q+1)(q^2+1)$. Using \cite[Lemma 2.2.23]{GM} we calculate that $\chi(x) = 4(q+1)$, which yields $|x^{G_0} \cap H_0| = 4q^3$. Since there are $r-1$ distinct $G_0$-classes of such elements, the contribution here is equal to 
\[
\sum_{r \in \pi} (r-1) \cdot \frac{(4q^3)^3}{(q^3(q+1)(q^2+1))^2} = \sum_{r \in \pi} 64(r-1) \cdot \frac{q^3}{(q+1)^2(q^2+1)^2},
\]
which is at most
\[
\b_2 = \frac{64(q-2)q^3\log q}{(q+1)^2(q^2+1)^2}.
\]

We conclude that the total contribution to $\mathcal{Q}(G,3)$ from semisimple elements is less than $\b_1+\b_2 < 12q^{-1}$.

Next assume $x \in G$ is a field automorphism of order $r$, so $q=q_0^r$ and 
\[
|x^{G_0}| = \frac{q^4(q^2-1)(q^4-1)}{q^{4/r}(q^{2/r}-1)(q^{4/r}-1)} = f(q,r).
\]
As before, $C_{H_0}(x)$ is a Borel subgroup of $C_{G_0}(x) = {\rm Sp}_{4}(q_0)$ and this implies that 
\[
|x^{G_0} \cap H_0x| = \frac{q^4(q-1)^2}{q^{4/r}(q^{1/r}-1)^2} = g(q,r).
\]
Since there are $r-1$ distinct $G_0$-classes of field automorphisms of order $r$ in ${\rm Aut}(G_0)$, it follows that the combined contribution to $\mathcal{Q}(G,3)$ from field automorphisms is equal to
\[
\varphi = \sum_{r \in \pi}(r-1) \cdot g(q,r)^3f(q,r)^{-2},
\]
where $\pi$ is the set of prime divisors of $f = \log q$. Set $e(q,r) = (r-1) \cdot g(q,r)^3f(q,r)^{-2}$. If $2^6 \leqs q \leqs 2^{11}$ then it is easy to check that $\varphi<\frac{1}{10}$, so let us assume that $q \geqs 2^{12}$.

If $f=r$ then $\varphi = e(q,r)$ and one checks that this is less than $q^{-1}$. Now assume $f$ is composite, so $q_0 = q^{1/r} \geqs 4$ for each $r \in \pi$.  This implies that 
\[
f(q,r) > q^{10\left(1-\frac{1}{r}\right)},\;\; g(q,r)<2q^{6\left(1-\frac{1}{r}\right)}
\]
and thus
\[
e(q,r) < 8(r-1)q^{-2\left(1-\frac{1}{r}\right)} = 8(r-1)q_0^{-2(r-1)} \leqs 4q^{-1}
\]
for all $r \geqs 3$. Since $e(q,2)<2q^{-1}$, we deduce that 
\[
\varphi < 2q^{-1}+|\pi| \cdot 4q^{-1} < 2q^{-1}+4q^{-1}\log\log q
\]
for all $q \geqs 2^{12}$.

Finally, let us assume $x$ is an involutory graph automorphism of $G_0$, so $f$ is odd and 
\[
|x^G\cap H| = \frac{q^4(q-1)^2}{q^2(q-1)} = q^2(q-1),\;\; |x^G| = \frac{|{\rm Sp}_{4}(q)|}{|{}^2B_2(q)|} = q^2(q^2-1)(q+1)
\]
since $C_{H_0}(x)$ is a Borel subgroup of $C_{G_0}(x) = {}^2B_2(q)$. Therefore, the contribution from these elements is
$q^2(q-1)(q+1)^{-4} < q^{-1}$.

By bringing the above bounds together, we conclude that if $q \geqs 64$ then
\[
\mathcal{Q}(G,3) < 13q^{-1}+2q^{-2}+\eta,
\]
where $\eta = \frac{1}{10}$ if $q \leqs 2^{11}$, otherwise $\eta = 2q^{-1}+4q^{-1}\log\log q$. Therefore, $\mathcal{Q}(G,3) <1$ for all $q \geqs 64$ and we also observe that $\mathcal{Q}(G,3) \to 0$ as $q \to \infty$.
\end{proof}

\begin{lem}\label{l:g2}
If $G_0 = G_2(q)$ and $H$ is of type $[q^6]{:}C_{q-1}^2$, then $b(G,H)=3$ and $\mathcal{P}(G,3) \to 1$ as $q \to \infty$.
\end{lem}

\begin{proof}
Here $q = 3^f$ and the maximality of $H$ implies that $G$ contains graph automorphisms. The cases $q \in \{3,9\}$ can be handled using {\sc Magma}, so let us assume $q \geqs 27$. Note that 
\[
H_0 = [q^6]{:}C_{q-1}^2,\;\; |\O| = (q+1)(q^5+q^4+q^3+q^2+q+1)
\]
and $\log |G| / \log |\O|>2$, whence $b(G,H) \geqs 3$.

Let $\chi$ be the corresponding permutation character of $G_0$. As explained in \cite[Section 2]{LLS2}, we can decompose $\chi$ as a sum
\[
\chi = R_{\phi_{1,6}}+R_{\phi_{1,3}'}+R_{\phi_{1,3}''}+2R_{\phi_{2,2}}+2R_{\phi_{2,1}}+R_{\phi_{1,0}}
\]
where each $R_{\phi}$ is an almost character of $G_0$ labelled by an  irreducible character $\phi$ of the Weyl group of $G_0$ (here we are using the labelling given in \cite[Section 13.2]{Carter}). Let $x \in G_0$ be an element of prime order $r$. As usual, we may assume $r$ divides $|H_0|$, so either $r=3$ or $r$ divides $q-1$.

First assume $r=3$, so $x$ is unipotent. The restriction of the $R_{\phi}$ to unipotent elements are called the Green functions of $G_0$, which are polynomials in $q$ with non-negative coefficients. The full character table of $G_0$ is available in \cite{Eno} and all of the relevant Green functions have been computed (see \cite{Lub}, for example). This allows us to read off $\chi(x)$ for each element $x \in G_0$ of order $3$ and we obtain the following results, where we use the labels from \cite[Table 22.2.6]{LS_book} for the unipotent classes in the ambient algebraic group $\bar{G} = G_2$ (note that there are two $G$-classes of elements of type $G_2(a_1)$):
\[
\begin{array}{llll} \hline
 & \chi(x) & |x^G| & |x^G \cap H| \\ \hline
A_1 & (q+1)(q^2+q+1) & 2(q^6-1) & 2(q^2+q+1)(q-1) \\
(\tilde{A}_1)_3 & 2q^2+2q+1 & (q^2-1)(q^6-1) & (2q^2+2q+1)(q-1)^2 \\
G_2(a_1) & 2q+1 & \frac{1}{2}q^2(q^2-1)(q^6-1) & \frac{1}{2}q^2(q-1)^2(2q+1) \\ \hline
\end{array}
\]
We deduce that the contribution to $\mathcal{Q}(G,3)$ from unipotent elements is less than $q^{-2}$.

Now assume $r \ne 3$, so $r$ divides $q-1$ and $C_{\bar{G}}(x)$ is either $A_1\tilde{A}_1$ ($r=2$ only), $A_1T_1$, $\tilde{A}_1T_1$ or $T_2$, where $\tilde{A}_1$ denotes an $A_1$ subgroup generated by short root subgroups and $T_i$ is an $i$-dimensional torus. We refer the reader to \cite{Lu2} for a convenient source of information on the semisimple conjugacy classes in $G_0$ (the original reference is \cite{Chang}). In each case, we compute $\chi(x)$ by applying \cite[Corollary 3.2]{LLS2}, which can be implemented in {\sc Magma} (alternatively, one can also do this using \cite[Lemma 2.2.23]{GM}). In this way, we obtain the results presented below, where $n$ denotes the number of $G_0$-classes of semisimple elements with the given centraliser:
\[
\begin{array}{lllll} \hline
 & |C_{G_0}(x)| & n & \chi(x) &  |x^{G_0} \cap H_0| \\ \hline
A_1\tilde{A}_1 & q^2(q^2-1)^2  & 1 & 3(q+1)^2 &  3q^4 \\
A_1T_1 & q(q-1)(q^2-1) & \frac{1}{2}(q-3) & 6(q+1) &  6q^5 \\
\tilde{A}_1T_1 & q(q-1)(q^2-1) & \frac{1}{2}(q-3) & 6(q+1) &  6q^5 \\
T_2 & (q-1)^2 & \frac{1}{12}(q^2-8q+15) & 12 &  12q^6 \\ \hline
\end{array}
\]
One checks that the total contribution to $\mathcal{Q}(G,3)$ from semisimple elements is less than $q^{-2}$.

Now suppose $x \in G$ is a field automorphism of prime order $r$. Then $q=q_0^r$ and we have 
\[
|x^G \cap H| = \frac{q^6(q-1)^2}{q_0^6(q_0-1)^2} < 2q^{8\left(1-\frac{1}{r}\right)},\;\; |x^G| = \frac{|G_2(q)|}{|G_2(q_0)|} > q^{14\left(1-\frac{1}{r}\right)}
\]
so the contribution to $\mathcal{Q}(G,3)$ from these elements is less than
\[
\sum_{r \in \pi} (r-1)\cdot 8q^{-4\left(1-\frac{1}{r}\right)} < |\pi|\cdot q^{-1} < q^{-1}\log\log q,
\]
where $\pi$ is the set of prime divisors of $f = \log_3 q$.

Finally, suppose $x \in G$ is an involutory graph automorphism. Here $f$ is odd and $C_{H_0}(x)$ is a Borel subgroup of $C_{G_0}(x) = {}^2G_2(q)$, so
\[
|x^G \cap H| = q^3(q-1),\;\; |x^G| = q^3(q-1)(q^3-1)
\]
and the contribution from graph automorphisms is equal to $q^3(q-1)(q^3-1)^{-2}< q^{-2}$. We conclude that if $q \geqs 27$, then
\[
\mathcal{Q}(G,3) < 3q^{-2}+q^{-1}\log\log q
\]
and the result follows.
\end{proof}

Finally, we turn to the cases labelled (e) and (f) in Table \ref{tab:par}. 

\begin{lem}\label{l:pex}
Suppose $G_0 = {}^2B_2(q)$ or ${}^2G_2(q)$ and $H$ is of type $[q^2]{:}C_{q-1}$ or $[q^3]{:}C_{q-1}$, respectively. Then $b(G,H)=3$ and $\mathcal{P}(G,3) \to 1$ as $q \to \infty$.
\end{lem}

\begin{proof}
This follows immediately from \cite[Theorem 3(i)]{BLS}. 
\end{proof}

This completes the proof of Proposition \ref{p:par}.

\section{Classical groups: Non-parabolic actions}\label{s:class}

We are now ready to complete the proof of Theorem \ref{t:as} for classical groups by handling the remaining cases where $G_0 \ne {\rm L}_{2}(q)$ and $H$ is a non-parabolic subgroup of $G$. We continue to assume (as we may) that $G_0$ satisfies the conditions presented in Remark \ref{r:isom}. Our main result is the following.

\begin{prop}\label{p:nonpar}
Let $G \leqs {\rm Sym}(\O)$ be a finite almost simple primitive classical group with socle $G_0$ and soluble point stabiliser $H$. Set $b=b(G,H)$ and assume $G_0 \ne {\rm L}_{2}(q)$ and $H$ is non-parabolic.  
\begin{itemize}\addtolength{\itemsep}{0.2\baselineskip}
\item[{\rm (i)}] We have $b \leqs 4$, with $b>2$ if and only if $(G,H,b)$ is one of the cases in Table \ref{tab:as4}.
\item[{\rm (ii)}] In addition, $\mathcal{P}(G,2) \to 1$ as $|G| \to \infty$.
\end{itemize}
\end{prop}

In order to prove this result, we first need to determine the possibilities for $G$ and $H$. 

\begin{lem}\label{l:nonpar1}
Let $G$ be a finite almost simple classical group over $\mathbb{F}_q$ with socle $G_0 \ne {\rm L}_{2}(q)$ and a maximal soluble non-parabolic subgroup $H$. Then one of the following holds:
\begin{itemize}\addtolength{\itemsep}{0.2\baselineskip}
\item[{\rm (i)}] $G_0 \in \{  {\rm L}_{3}^{\e}(q), {\rm L}_{4}^{\e}(q), {\rm L}_{5}^{\e}(q), {\rm PSp}_{6}(q) \} \cup \{{\rm L}_{3}(4), \O_7(3), {\rm PSp}_{8}(3), \O_{8}^{+}(2)\}$ with $q \in \{2,3\}$.
\item[{\rm (ii)}] $G_0 = {\rm L}_{n}^{\e}(3)$, $n \in \{6,8\}$ and $H$ is of type ${\rm GL}_{2}^{\e}(3) \wr S_{n/2}$.
\item[{\rm (iii)}] $G_0 = {\rm U}_{n}(2)$, $n \in \{6,9,12\}$ and $H$ is of type ${\rm GU}_{3}(2) \wr S_{n/3}$.
\item[{\rm (iv)}] $G_0 = {\rm P\O}_{n}^{+}(3)$, $n \in \{8,12,16\}$ and $H$ is of type ${\rm O}_{4}^{+}(3) \wr S_{n/4}$.
\item[{\rm (v)}] $(G,H)$ is one of the cases recorded in Table \ref{tab:nonpar}.
\end{itemize}
\end{lem}

\begin{table}
\[
\begin{array}{clll} \hline
\mbox{Case} & G_0 & \mbox{Type of $H$} & \mbox{Conditions} \\ \hline
{\rm (a)} & {\rm L}_{n}^{\e}(q) & {\rm GL}_{1}^{\e}(q^n) & \mbox{$n \geqs 3$ prime; $(n,q,\e) \ne (3,3,-)$, $(5,2,-)$} \\
{\rm (b)} & {\rm L}_{n}^{\e}(q) & {\rm GL}_{1}^{\e}(q) \wr S_n   & \mbox{$n = 3,4$; $q \geqs 5$ if $\e=+$} \\
{\rm (c)} & {\rm P\O}_{8}^{+}(q) & {\rm O}_{2}^{\e}(q) \wr S_4  & \mbox{$q \geqs 5$ if $\e=+$} \\ 
{\rm (d)} & {\rm P\O}_{8}^{+}(q) & {\rm O}_{2}^{-}(q^2) \times {\rm O}_{2}^{-}(q^2)  & G \not\leqs \la {\rm PGO}_{8}^{+}(q), \phi \ra \\
{\rm (e)} & {\rm Sp}_{4}(q) & {\rm O}_{2}^{\e}(q) \wr S_2 &  \mbox{$q \geqs 4$ even, $G \not\leqs \la G_0, \phi \ra$} \\
{\rm (f)} & {\rm Sp}_{4}(q) & {\rm O}_{2}^{-}(q^2) &  \mbox{$q \geqs 4$ even, $G \not\leqs \la G_0, \phi \ra$} \\
{\rm (g)} & {\rm U}_{3}(q) & {\rm GU}_{3}(2) & \mbox{$q=2^k$, $k \geqs 3$ prime} \\
{\rm (h)} & {\rm L}_{3}^{\e}(q) & 3^{1+2}.{\rm Sp}_{2}(3) & q = p \equiv \e \imod{3} \\ \hline
\end{array}
\]
\caption{Non-parabolic actions of classical groups}
\label{tab:nonpar}
\end{table}

\begin{proof}
This follows by inspecting \cite[Tables 16--19]{LZ}.
\end{proof}

\begin{prop}\label{p:nonpar1}
Proposition \ref{p:nonpar} holds in cases (i)--(iv) of Lemma \ref{l:nonpar1}.
\end{prop}

\begin{proof}
All of the groups arising in part (i) can be handled using {\sc Magma} in the usual manner (via \texttt{AutomorphismGroupSimpleGroup} and \texttt{MaximalSubgroups}). 

Now let us consider the cases in (ii), (iii) and (iv). The case $G_0 = {\rm U}_{6}(2)$ with $H$ of type ${\rm GU}_{3}(2) \wr S_2$ can be treated in the same way as those in (i) and one checks that $b(G,H) = 3$. The special case where $G_0 = {\rm P\O}_{8}^{+}(3)$ and $H$ is of type ${\rm O}_{4}^{+}(3) \wr S_2$ was discussed in Example \ref{e:special} and we find that $b(G,H) \leqs 3$, with equality if and only if $|G:G_0| \geqs 6$.

In all of the remaining cases, we claim that $b(G,H) = 2$. To prove this, we may assume that $G = {\rm Aut}(G_0)$. If $G_0 = {\rm L}_{6}(3)$ then we can use the \texttt{MaximalSubgroups} function in {\sc Magma} to construct $H$ and we quickly deduce that $b(G,H) = 2$ by random search. In the remaining cases, we have
\[
G_0 \in \{ {\rm U}_{6}(3), {\rm L}_{8}^{\e}(3), {\rm U}_{9}(2), {\rm U}_{12}(2), {\rm P\O}_{12}^{+}(3), {\rm P\O}_{16}^{+}(3)\}
\]
and the \texttt{MaximalSubgroups} function is ineffective.

As noted in Example \ref{e:u63}, if $G_0 = {\rm U}_{6}(3)$ and $H$ is of type ${\rm GU}_{2}(3) \wr S_3$, then $H = N_{G}(K)$ for some subgroup $K$ of order $2^{10}$ and we can use this observation to construct $G$ and $H$ as permutation groups of degree $22204$ (see Example \ref{e:u63} for the details). It is then straightforward to find an element $x \in G$ with $H \cap H^x = 1$. The cases with $G_0 = {\rm L}_{8}^{\e}(3)$ can be handled in an entirely similar fashion, using the fact that $H = N_{G}(K)$ with $|K| = 2^{12-\e}$. Similarly, if $G_0 = {\rm U}_{9}(2)$ and $H$ is of type ${\rm GU}_{3}(2) \wr S_3$, then $H = N_G(K)$ for a subgroup $K<G_0$ of order $3^8$ and we can treat this case in the same way. We refer the reader to Example \ref{e:o123} for the case where $G_0 = {\rm P\O}_{12}^{+}(3)$ and $H$ is of type ${\rm O}_{4}^{+}(3) \wr S_3$.

The final two cases are more difficult to handle computationally and we will show that $b(G,H) =2$ by establishing the bound $\mathcal{Q}(G,2)<1$. 

First assume $G_0 = {\rm U}_{12}(2)$ and $H$ is of type ${\rm GU}_{3}(2) \wr S_4$. By \cite[Proposition 4.2.9]{KL} we have
\[
H_0 = [3^3].{\rm U}_{3}(2)^4.3^3.S_4,\;\; 
H \cap {\rm PGU}_{12}(2) \leqs [3^3].({\rm PGU}_{3}(2) \wr S_4).
\]
Let $N = [3^3]$ be the normal subgroup of $H_0$ and let us view $H$ as the stabiliser in $G$ of an orthogonal decomposition $V = V_1 \perp V_2 \perp V_3 \perp V_4$ of the natural module, where each $V_i$ is a non-degenerate $3$-space. Let $x \in H$ be an element of order $r$, so $r \in \{2,3\}$. We refer the reader to \cite[Section 3.3]{BG} for information on the conjugacy classes of prime order elements in ${\rm Aut}(G_0)$, which we will repeatedly use in the following analysis. Note that $|H|<2^{42} = a_1$. Recall that if $X$ is a subset of a finite group, then $i_r(X)$ denotes the number of elements of order $r$ in $X$.

First assume $r=3$, so $x$ is semisimple. If some conjugate of $x$ induces a nontrivial permutation of the $V_i$, then each cube root of unity arises as an eigenvalue of $x$ on $V$ with multiplicity at least $3$ and we deduce that  
\[
|x^G| \geqs \frac{|{\rm GU}_{12}(2)|}{|{\rm GU}_{6}(2)||{\rm GU}_{3}(2)|^2} > 2^{89} = b_1.
\]
Now assume every element in $x^G \cap H$ fixes each $V_i$ and observe that there are at most
\[
|N|\cdot (1 + i_3({\rm PGU}_{3}(2)^4)) < 2^{31} = a_2
\]
such elements in $H$. Therefore, the contribution to $\mathcal{Q}(G,2)$ from these elements $x \in H$ of order $3$ with $|x^G|>3.2^{62} = b_2$ is less than $a_2^2/b_2 = \frac{1}{3}$. So let us assume $|x^G| \leqs 3.2^{62}$. Then one checks that the possibilities for $x$, up to ${\rm Aut}(G_0)$-conjugacy, are as follows (here $\mathbb{F}_4^{\times} = \la \omega \ra$):
\[
\begin{array}{llll} \hline
i & x & a_i & b_i \\ \hline
3 & [I_{11},\omega] & 48 & 2^{21} \\
4 & [I_{10},\omega I_2] & 912 & 2^{39} \\
5 & [I_{10}, \omega, \omega^2] & 1824 & 2^{40} \\
6 & [I_9, \omega I_3] & 8644 & 2^{53} \\
7 & [I_9, \omega I_2, \omega^2] & 22512 & 2^{56} \\ \hline
\end{array}
\]

In this table, we also record bounds $|x^{G_0} \cap H| \leqs a_i$ and $|x^{G_0}| > b_i$. For example, if $x$ is the image of $[I_{10},\omega, \omega^2] \in {\rm GU}_{12}(2)$ then
\[
|x^{G_{0}}| = \frac{|{\rm GU}_{12}(2)|}{|{\rm GU}_{10}(2)||{\rm GU}_{1}(2)|^2} > 2^{40}
\]
and we calculate that
\[
|x^{G_0} \cap H| \leqs \binom{4}{1}\frac{|{\rm GU}_{3}(2)|}{|{\rm GU}_{1}(2)|^3} + 2\binom{4}{2}\left(\frac{|{\rm GU}_{3}(2)|}{|{\rm GU}_{2}(2)||{\rm GU}_{1}(2)|}\right)^2 = 1824.
\]
We conclude that the combined contribution to $\mathcal{Q}(G,2)$ from elements of order $3$ is less than
\[
a_1^2/b_1 + a_2^2/b_2 + 2\sum_{i=3}^{7}a_i^2/b_i < \frac{1}{2}.
\]

Now let us assume $x \in H$ is an involution. Suppose for now that $x$ is unipotent, so $x$ has Jordan form $[J_2^k,J_1^{12-2k}]$ for some $1 \leqs k \leqs 6$. If $x = [J_2, J_1^{10}]$ then 
\[
|x^G \cap H| \leqs \binom{4}{1}\frac{|{\rm GU}_{3}(2)|}{2^3|{\rm GU}_{1}(2)|^2} = 36 = a_8,\;\; |x^G| =  \frac{|{\rm GU}_{12}(2)|}{2^{21}|{\rm GU}_{10}(2)||{\rm GU}_{1}(2)|} > 2^{21} = b_8.
\]
Similarly, if $x = [J_2^2,J_1^{8}]$ then $|x^G \cap H| \leqs \binom{4}{2}9^2 = 486 = a_9$ and $|x^G|>2^{39} = b_9$. 
For all other unipotent involutions, one checks that $|x^G|>2^{53}=b_{10}$ and we note that 
\[
i_2(H_0) \leqs i_2({\rm PGU}_{3}(2) \wr S_4) = 279567 < 2^{19} = a_{10}.
\]
Therefore, the total contribution to $\mathcal{Q}(G,2)$ from unipotent involutions is less than 
\[
\sum_{i=8}^{10}a_i^2/b_i < 2^{-10}.
\]

To complete the argument in this case,  we may assume $x \in H$ is an involutory graph automorphism. If $C_{G_0}(x) = {\rm Sp}_{12}(2)$, then $|x^G|>2^{63} = b_{11}$ and the proof of \cite[Proposition 2.7]{Bur3} gives the bound $|x^G \cap H| \leqs |{\rm GU}_{3}(2)|^2 < 2^{19} = a_{11}$. On the other hand, if $C_{G_0}(x) \ne {\rm Sp}_{12}(2)$, then $|x^G|>2^{75} = b_{12}$ and we note that $H$ contains fewer than
\[
|H_0/N| = |{\rm U}_{3}(2)|^4.3^3.|S_4| < 2^{35} = a_{12}
\]
involutory graph automorphisms. Therefore, the contribution from graph automorphisms is less than $a_{11}^2/b_{11}+ a_{12}^2/b_{12} < 2^{-4}$ and we conclude that
\[
\mathcal{Q}(G,2) < 2^{-1}+2^{-10}+2^{-4} < 1,
\]
which implies that $b(G,H) = 2$.

Finally, let us assume $G_0 = {\rm P\O}_{16}^{+}(3)$ and $H$ is of type ${\rm O}_{4}^{+}(3) \wr S_4$. We may view $H$ as the stabiliser of an orthogonal decomposition $V = V_1 \perp V_2 \perp V_3 \perp V_4$, where each $V_i$ is a $4$-dimensional non-degenerate plus-type space. By \cite[Proposition 4.2.11]{KL} we have
\[
H_0 = 2^3.{\rm P\O}_{4}^{+}(3)^4.2^6.S_4,\;\; H \leqs 2^3.{\rm PO}_{4}^{+}(3)^4.2.S_4.
\]
Let $N = 2^3$ be the normal subgroup of $H_0$ and observe that
\[
i_2(H) \leqs |N|\cdot (1+i_2({\rm PGO}_{4}^{+}(3) \wr S_4)) < 3^{20},\;\; i_3(H) \leqs i_3({\rm PGO}_{4}^{+}(3) \wr S_4) < 3^{19}.
\]
Let $x \in H$ be an element of prime order $r$, so $r \in \{2,3\}$. See \cite[Section 3.5]{BG} for detailed information on the conjugacy classes of elements of prime order in orthogonal groups.

First assume $r=3$, so $x$ is unipotent. If $x$ has Jordan form $[J_2^2,J_1^{12}]$, then $|x^G|> 3^{26} = b_{1}$ and we calculate that there are at most 
\[
a_1 =\binom{4}{1}\frac{|{\rm O}_{4}^{+}(3)|}{3|{\rm Sp}_{2}(3)|} = 288
\]
of these elements in $H$. Similarly, if $x = [J_3,J_1^{13}]$ then $|x^G|> 3^{27}= b_2$ and $H$ contains at most $a_2 = 256$ such elements. For all other elements of order $3$ we have $|x^G|> 3^{44}= b_3$ (minimal if $x$ has Jordan form $[J_2^4,J_1^8]$) and $i_3(H) < 3^{19} = a_3$ as noted above.

Now assume $x \in H$ is an involution. Since $i_2(H) < 3^{20} = a_4$, it follows that the contribution to $\mathcal{Q}(G,2)$ from the elements with $|x^G| > 3^{47} = b_4$ is less than $a_4^2/b_4 = 3^{-7}$. Now assume $|x^G| \leqs 3^{47}$, which implies that $x$ is the image of an involution in ${\rm GO}_{16}^{+}(3)$ of the form $[-I_{\ell},I_{16-\ell}]$ with $\ell \leqs 3$. In particular, $x$ fixes each $V_i$ in the above orthogonal decomposition stabilised by $H$. Set $m = i_2({\rm PGO}_{4}^{+}(3)) = 123$.

If $x = [-I_{1},I_{15}]$ then 
\[
|x^G| \geqs \frac{|{\rm O}_{16}^{+}(3)|}{2|{\rm O}_{15}(3)|} >3^{14} = b_5
\]
and there are fewer than $a_5 = \binom{4}{1}m = 492$ such elements in $H$. Similarly, if $x = [-I_2,I_{14}]$ then $|x^G|>3^{27} = b_6$ and $H$ contains at most $a_6 = \binom{4}{2}m^2 + \binom{4}{1}m = 91266$ such elements. Finally, if $x = [-I_3,I_{13}]$ then $|x^G| > 3^{37} = b_7$ and there are less than $a_7 = \binom{4}{3}m^3+2\binom{4}{2}m^2+\binom{4}{1}m = 7625508$ of these elements in $H$.

We conclude that
\[
\mathcal{Q}(G,2) < \sum_{i=1}^{7}a_i^2/b_i < 1
\]
and the result follows.
\end{proof}

For the remainder of this section, we consider each of the infinite families in Table \ref{tab:nonpar} in turn. We continue to assume that $G_0$ satisfies the conditions described in Remark \ref{r:isom}. For example, in the statement of the next lemma, the case $G_0 = {\rm L}_{3}(2)$ is excluded.

\begin{lem}\label{l:c3}
Suppose $G_0 = {\rm L}_{n}^{\e}(q)$ and $H$ is of type ${\rm GL}_{1}^{\e}(q^n)$, where $n \geqs 3$ is a prime. Then $b(G,H) \leqs 3$, with equality if and only if $G = {\rm L}_{3}(3).2$. In addition, $\mathcal{P}(G,2) \to 1$ as $|G| \to \infty$.
\end{lem}

\begin{proof}
Here $H_0 = C_{m}{:}C_{n}$, where $m = (q^n-\e)/d(q-\e)$ and $d = (n,q-\e)$ (see \cite[Proposition 4.3.6]{KL}). Let $V$ be the natural module for $G_0$ and let $x \in H$ be an element of prime order $r$.

If $x$ is unipotent, then $r=p=n$ and $x$ has Jordan form $[J_p]$ on $V$. Similarly, if $x$ is semisimple then it embeds in $G$ as a regular element (see \cite[Lemma 5.3.2]{BG}, for example). Therefore, 
\[
|x^G|>\frac{1}{2n}\left(\frac{q}{q+1}\right)^nq^{n^2-n} = b_1
\]
for all unipotent and semisimple elements in $H$ and we note that 
\[
|H \cap {\rm PGL}_{n}^{\e}(q)| \leqs n\left(\frac{q^n-1}{q-1}\right) = a_1.
\] 

Now assume $x$ is either a field automorphism of odd prime order or an involutory graph automorphism of $G_0$. Then $|x^G|>\frac{1}{2n}q^{n^2/2+n/2-1} = b_2$ and we observe that
\[
|H| \leqs 2n\left(\frac{q^n-1}{q-1}\right)\log q = a_2.
\]
Finally, suppose $\e=+$, $q=q_0^2$ and $x$ is an involutory field or graph-field automorphism. Here $|x^G|>\frac{1}{2n}q^{(n^2-1)/2} = b_3$ and $H$ contains at most 
\begin{equation}\label{e:phi}
\frac{q^{n/2}-1}{q^{1/2}-1} + \frac{q^{n/2}+1}{q^{1/2}+1} < 2\left(\frac{q^{n/2}-1}{q^{1/2}-1}\right)= a_3
\end{equation}
of these elements.  

In view of the above bounds, we conclude that $\mathcal{Q}(G,2) < \sum_{i=1}^{3}a_i^2/b_i$. For $n \geqs 5$, one checks that this upper bound is less than $q^{-n/2}$ unless $(n,q) = (7,2)$, or $n=5$ and $q \leqs 16$. Moreover, it is  less than $1$ unless $n=5$ and $q \leqs 4$; these remaining cases can be checked using {\sc Magma}.

To complete the proof, we may assume $n=3$. We can use {\sc Magma} to verify the result for $q \leqs 19$, so let us assume $q>19$. Let $x \in H$ be an element of prime order $r$. If $x$ is semisimple or unipotent, then 
\[
|x^G|\geqs \frac{1}{3}q^3(q-1)(q^2-q+1)=b_1
\]
and we note that $|H \cap {\rm PGL}_{3}^{\e}(q)| \leqs 3(q^2+q+1) = a_1$. Next assume $x$ is a field automorphism and $r$ is odd. Here $r \geqs 5$ (since every element in $H$ of order $3$ is contained in ${\rm PGL}_{3}^{\e}(q)$), so $|x^G|>\frac{1}{6}q^{32/5} = b_2$ and there are at most $a_2=3(q^2+q+1)\log q$ of these elements in $H$. If $x$ is an involutory graph automorphism, then $|x^G| \geqs \frac{1}{3}q^2(q^3-1) = b_3$ and $x$ inverts the normal subgroup $C_{(q^2+\e q+1)/d}$ of $H_0$. Since this torus has odd order, we deduce that $H$ contains  at most $a_3=q^2+q+1$ involutory graph automorphisms. Finally, suppose $\e=+$ and $x$ is an involutory field or graph-field automorphism. Here $q=q_0^2$, 
\[
|x^G| \geqs \frac{1}{3}q^{3/2}(q+1)(q^{3/2}-1) = b_4
\]
and as noted above (see \eqref{e:phi}) there are fewer than $a_4=2(q+q^{1/2}+1)$ of these elements in $H$. 

We conclude that if $n=3$ then $\mathcal{Q}(G,2) < \sum_{i=1}^{4}a_i^2/b_i$. It is routine to check that this upper bound is less than $1$ if $q > 19$, and it is less than $q^{-1/2}$ if $q > 73$. The result follows.
\end{proof}

\begin{lem}\label{l:c2_1}
Suppose $G_0 = {\rm L}_{3}^{\e}(q)$ and $H$ is of type ${\rm GL}_{1}^{\e}(q) \wr S_3$. Then $b(G,H) \leqs 3$, with equality if and only if $G_0 = {\rm U}_{3}(3)$, or if $G_0 = {\rm U}_{3}(4)$ and $G \ne G_0$. In addition, $\mathcal{P}(G,2) \to 1$ as $q \to \infty$.
\end{lem}

\begin{proof}
Write $q=p^f$ and set $d=(3,q-\e)$. Here $H_0 = [(q-\e)^2/d].S_3$ is the stabiliser in $G_0$ of a direct sum decomposition
$V = V_1 \oplus V_2 \oplus V_3$
of the natural module into $1$-spaces (more precisely, this is an orthogonal decomposition into non-degenerate $1$-spaces when $\e=-$). As noted in \cite[Table 8.3]{BHR}, if $\e=+$ then $q \geqs 5$ (otherwise $H$ is non-maximal). The cases with $q \leqs 27$ can be checked using {\sc Magma}, so we will assume $q>27$. 

Let $x \in H$ be an element of prime order $r$. First assume $x$ is unipotent, so $r =p \in \{2,3\}$. If $r=3$ then $x$ acts transitively on the $V_i$, whence $x$ has Jordan form $[J_3]$ on $V$ and $|x^G| \geqs q(q^2-1)(q^3-1)=b_1$. Moreover, there are at most $a_1=2(q+1)^2$ of these elements in $H$. Similarly, if $r=2$ then $x$ acts as a transposition on the $V_i$ and it has Jordan form $[J_2,J_1]$ on $V$. Therefore, $|x^G| \geqs (q^2-1)(q^2-q+1) = b_2$ and $H$ contains at most $a_2=3(q+1)$ such elements. 

Next assume $x$ is semisimple. If $r=2$ then $q$ is odd, $|x^G| \geqs q^2(q^2-q+1) = b_3$ and we note that there is a unique class of involutions in ${\rm PGL}_{3}^{\e}(q)$. Since 
\[
i_2(H \cap {\rm PGL}_{3}^{\e}(q)) \leqs i_2(C_{q-\e}^2) + 3(q-\e) = 3(q+1-\e),
\]
it follows that $H$ contains at most $a_3=3(q+2)$ semisimple involutions. Now assume $r=3$. If $q \not\equiv \e \imod{3}$, then $x$ cyclically permutes the $V_i$, so $|x^G| \geqs q^3(q^3-1) = b_4$ and there are at most $a_4=2(q+1)^2$ of these elements in $H$. Suppose now that $r=3$ and $q \equiv \e \imod{3}$. If $x$ is not regular, then $|x^G| \geqs q^2(q^2-q+1) = b_5$ and $x$ fixes each $V_i$, so there are at most $i_3(C_{q-\e}^2) = 8 = a_5$ such elements in $H$. On the other hand, if $x$ is regular, then $|x^G| \geqs \frac{1}{3}q^3(q-1)(q^2-q+1) = b_6$ and we note that $i_3(H \cap {\rm PGL}_{3}^{\e}(q)) \leqs 8+2(q+1)^2 = a_6$.  

To complete the analysis of semisimple elements, let us assume $x$ has order $r \geqs 5$, so $r$ is a divisor of $q-\e$ and $x$ fixes each $V_i$. Let $\pi$ be the set of such primes. First assume $x$ is regular. Then up to conjugacy, $x$ is the image of an element in ${\rm SL}_{3}^{\e}(q)$ of the form $[1,\omega,\omega^{-1}]$, where $\omega \in \mathbb{F}_{q^u}$ is a primitive $r$-th root of unity (here $u=1$ if $\e=+$, otherwise $u=2$). Now $|x^G| \geqs q^3(q-1)(q^2-q+1)$ and we calculate that $|x^{G_0} \cap H| \leqs 6$. Since there are $(r-1)/2$ distinct $G_0$-classes of this form, it follows that the contribution to $\mathcal{Q}(G,2)$ from these elements is at most
\[
\sum_{r \in \pi} \frac{1}{2}(r-1) \cdot \frac{36}{q^3(q-1)(q^2-q+1)} < \frac{18\log q}{q^2(q-1)(q^2-q+1)}.
\]
Here we are using the fact that $|\pi|< \log q$ and $r \leqs q+1$ for all $r \in \pi$. Similarly, if $x$ is non-regular then $|x^G| \geqs q^2(q^2-q+1)$, $|x^{G_0} \cap H| \leqs 3$ and the contribution from these elements is less than
\[
\sum_{r \in \pi} (r-1) \cdot \frac{9}{q^2(q^2-q+1)} < \frac{9\log q}{q(q^2-q+1)}.
\]
It follows that the combined contribution to $\mathcal{Q}(G,2)$ from semisimple elements of order at least $5$ is less than $2q^{-2}$.

Next assume $x$ is a field automorphism of order $r \geqs 5$. Here $q \geqs 32$, $|x^G|>\frac{1}{6}q^{32/5} = b_7$ and there are at most
\[
\sum_{r \in \pi'}(r-1) \cdot (q-\e)^2 < (q+1)^2\log q = a_7
\]
such elements in $H$, where $\pi'$ is the set of prime divisors $r \geqs 5$ of $f = \log_pq$. 

Now suppose $x$ is a field automorphism of order $3$, so $q=q_0^3$ and 
\[
|x^G| \geqs \frac{1}{3}q^2(q^{4/3}+q^{2/3}+1)(q^2-q+1) = b_8.
\]
A straightforward calculation shows that there are at most
\[
a_8 = 4(q+1)(q^{1/3}+1) + 2(q^{2/3}+q^{1/3}+1)^2
\]
such elements in $H$. For example, suppose $\e=+$ and consider the coset $C_{q-1}^2\rho x$, where $\rho = (1,2,3) \in S_3$. We may identify $C_{q-1}^2$ with the subgroup $\{(a,b,a^{-1}b^{-1}) \,:\, a,b \in \mathbb{F}_{q}^{\times}\}< C_{q-1}^3$ and we may assume that the action of $x$ on $C_{q-1}^2$ is given by $(a,b,a^{-1}b^{-1})^x = (a^{q_0},b^{q_0},a^{-q_0}b^{-q_0})$. If $z = (a,b,a^{-1}b^{-1})\rho x$, then
\[
z^3 = (a^{1-q_0}b^{q_0(q_0-1)}, a^{q_0(1-q_0)}b^{1-q_0^2},a^{q_0^2-1}b^{q_0-1})
\]
and we deduce that $i_3(C_{q-1}^2\rho x) = (q-1)(q_0-1)$. Similarly, there are $(q-1)(q_0-1)$ elements of order $3$ in each of the cosets of $C_{q-1}^2$ containing $(1,2,3)x^2$, $(1,3,2)x$ and $(1,3,2)x^2$, and we calculate that there are $(q^{2/3}+q^{1/3}+1)^2$ elements of order $3$ in both $C_{q-1}^2x$ and $C_{q-1}^2x^2$. It follows that if $\e=+$ then $H$ contains at most
\[
4(q-1)(q^{1/3}-1) + 2(q^{2/3}+q^{1/3}+1)^2 < a_8
\]
field automorphisms of order $3$. A similar argument applies when $\e=-$. 

Now assume $\e=+$, $q=q_0^2$ and $x$ is an involutory field or graph-field automorphism. Here 
\[
|x^G| \geqs \frac{1}{3}q^{3/2}(q+1)(q^{3/2}-1) = b_9
\]
and by counting the number of involutions in each relevant coset of $C_{q-1}^2$ (noting that an involutory graph automorphism inverts $C_{q-1}^2$), we deduce that $H$ contains at most 
\[
a_9 = (q^{1/2}+1)^2+(q^{1/2}-1)^2+6(q-1) = 8q-4
\]
of these elements. Finally, if $x$ is an involutory graph automorphism, then $|x^G| \geqs \frac{1}{3}q^2(q^3-1) = b_{10}$ and one can check that there are at most $a_{10} = (q+1)^2+3(q+1) = q^2+5q+4$ such elements in $H$.

We conclude that if $q>27$ then
\[
\mathcal{Q}(G,2)< 2q^{-2} + \sum_{i=1}^{10}a_i^2/b_i < 1
\]
and thus $b(G,H) = 2$. Moreover, this upper bound is less than $q^{-1/2}$ if $q >89$. 
\end{proof}

\begin{lem}\label{l:c2_2}
Suppose $G_0 = {\rm L}_{4}^{\e}(q)$ and $H$ is of type ${\rm GL}_{1}^{\e}(q) \wr S_4$. Then 
\[
b(G,H) = \left\{\begin{array}{ll}
4 & \mbox{if $G_0 = {\rm U}_{4}(2)$} \\
3 & \mbox{if $G = {\rm U}_{4}(3).D_8$ } \\
2 & \mbox{otherwise}
\end{array}\right.
\]
and $\mathcal{P}(G,2) \to 1$ as $q \to \infty$.
\end{lem}

\begin{proof}
Set $q=p^f$ and $d = (4,q-\e)$. Here $H_0 = [(q-\e)^3/d].S_4$ is the stabiliser of an appropriate  direct sum decomposition $V = V_1 \oplus V_2 \oplus V_3 \oplus V_4$
of the natural module for $G_0$. As noted in \cite[Table 8.8]{BHR}, the maximality of $H$ implies that $q \geqs 5$ when $\e=+$. The cases with $q \leqs 8$ can be checked using {\sc Magma} so we will assume $q>8$ for the remainder of the proof.

Let $x \in H$ be an element of prime order $r$. First assume $x$ is unipotent, so $r=p \in \{2,3\}$. Suppose $p=2$. If $x$ has Jordan form $[J_2,J_1^2]$ on $V$, then $|x^G| \geqs (q^3+1)(q^2+1)(q-1) = b_1$ and we note that $x$ acts as a transposition on $\{V_1, \ldots, V_4\}$, whence 
\[
|x^G \cap H| \leqs \binom{4}{2}(q-\e) \leqs 6(q+1) =  a_1.
\]
Similarly, if $x = [J_2^2]$ then $|x^G|>\frac{1}{4}q^8=b_2$ and there are at most $a_2=3(q+1)^2$ of these elements in $H$ (here $x$ induces a double transposition on the $V_i$). Now assume $p=3$. Here $x = [J_3,J_1]$ is the only possibility and we have $|x^G|>\frac{1}{2}(q+1)^{-1}q^{11} = b_3$ and $|x^G \cap H| \leqs 8(q+1)^2 = a_3$. 

Next assume $x$ is a semisimple involution, so $q$ is odd. There are three conjugacy classes of involutions in ${\rm PGL}_{4}^{\e}(q)$, labelled $t_1$, $t_2$ and $t_2'$ in \cite[Sections 3.2.2 and 3.3.2]{BG}. First assume $x$ is of type $t_1$, so up to conjugacy $x$ is the image of an element in ${\rm GL}_{4}^{\e}(q)$ of the form $[-I_{1},I_3]$. Now $|x^G| \geqs q^3(q^2+1)(q-1)=b_4$ and we calculate that there are at most $a_4 = \binom{4}{1}+\binom{4}{2}(q+1) = 6q+10$ of these elements in $H$. Similarly, if $x$ is of type $t_2$ then $|x^G| \geqs \frac{1}{2}q^4(q-1)(q^3-1) = b_5$ and there are at most 
\[
a_5 = \binom{4}{2}+2\binom{4}{2}(q+1)+3(q+1)^2 = 3q^2+18q+21
\]
such elements in $H$. Finally, if $x$ is of type $t_2'$ then $|x^G| \geqs b_6 = b_5$ and $x$ induces a double transposition on the $V_i$, whence $H$ contains at most $a_6 = 3(q+1)^2$ of these involutions.

Now let us turn to the contribution to $\mathcal{Q}(G,2)$ from semisimple elements of odd prime order $r$. First assume $r=3$ and $q \equiv -\e \imod{3}$, so $x$ must induce a $3$-cycle on the $V_i$. Then up to conjugacy, $x$ is the image of a matrix of the form 
$[I_2,\omega, \omega^2] \in {\rm SL}_{4}(k)$, where $k = \bar{\mathbb{F}}_{q}$ and $\omega \in k$ is a primitive cube root of unity, and we obtain the bounds $|x^G| >\frac{1}{2}q^{10} = b_7$ and $|x^G \cap H| \leqs 8(q+1)^2= a_7$. Now assume $q \equiv \e \imod{3}$. Here there are four $G_0$-classes of elements of order $3$. If $x$ is of type $[I_3, \omega]$ or $[I_3, \omega^2]$ then $|x^G| \geqs q^3(q^2+1)(q-1) = b_8$ and in total there are at most $a_8 = 2\binom{4}{1}=8$ of these elements in $H$. Similarly, if $x = [ I_2, \omega I_2]$ then $|x^G \cap H| \leqs 6 = a_9$ and $|x^G| \geqs q^4(q^2+1)(q^2-q+1) = b_9$. Finally, if $x = [I_2, \omega,\omega^2]$ then $|x^G| >\frac{1}{2}(q+1)^{-2}q^{12} = b_{10}$ and $H$ contains at most $a_{10} = 2\binom{4}{2}+8(q+1)^2 = 8q^2+16q+20$ of these elements.

Now assume $r \geqs 5$, so $r$ divides $q-\e$ and $x$ fixes each $V_i$. If $x$ is of the form $[I_3,\omega]$ then $|x^{G_0} \cap H| = 4 = a_{11}$ and $|x^{G_0}| \geqs q^3(q^2+1)(q-1)=b_{11}$. Similarly, if $x = [I_2, \omega I_2]$ then $|x^{G_0} \cap H| = 6 = a_{12}$ and $|x^{G_0}| \geqs q^4(q^2-q+1)(q^2+1) = b_{12}$. There are $r-1$ distinct $G_0$-classes of elements of each type. If $x$ is any other element of order $r$, then $|x^G|>\frac{1}{2}(q+1)^{-2}q^{12} = b_{13}$ and we note that there are less than $a_{13} = (q+1)^3$ semisimple elements in $H$ of order at least $5$. Therefore, the combined contribution to $\mathcal{Q}(G,2)$ from semisimple elements of order at least $5$ is less than
\[
\a = a_{13}^2/b_{13} +\sum_{r \in \pi}(r-1)\cdot \left(a_{11}^2/b_{11} + a_{12}^2/b_{12}\right),
\]
where $\pi$ is the set of primes $r \geqs 5$ dividing $q-\e$. Since $r \leqs q+1$ and $|\pi|<\log q$, we deduce that
\[
\a < a_{13}^2/b_{13} + q\left(a_{11}^2/b_{11} + a_{12}^2/b_{12}\right)\log q < 2q^{-3}
\]
for all $q \geqs 9$.

To complete the proof, we may assume $x$ is a field, graph or graph-field automorphism. First assume $q=q_0^r$ and $x$ is a field automorphism of order $r$. If $r \geqs 3$ then $q \geqs 27$ (recall that we are assuming $q \geqs 9$), $|x^G|>\frac{1}{8}q^{10} = b_{14}$ and we observe that there are fewer than $a_{14} = 24(q+1)^3\log q$ of these elements in $H$. Now assume $r=2$, so $q=q_0^2$, $\e=+$ and 
\[
|x^G| \geqs \frac{1}{4}q^3(q+1)(q^{3/2}+1)(q^2+1) = b_{15}.
\]
By carefully counting the number of involutions in the relevant cosets of $C_{q-1}^3$, we deduce that $H$ contains at most 
\[
a_{15} = (q^{1/2}+1)^3 + \binom{4}{2}(q-1)(q^{1/2}+1) + 3(q-1)(q^{1/2}-1)
\]
involutory field automorphisms. For example, if $z \in C_{q-1}^3\rho x$, say 
\[
z = (a,b,c,a^{-1}b^{-1}c^{-1})\rho x,
\]
where $\rho = (1,2)(3,4) \in S_4$ and $a,b,c \in \mathbb{F}_{q}^{\times}$, then
\[
z^2 = (ab^{q_0},ba^{q_0}, a^{-q_0}b^{-q_0}c^{1-q_0},a^{-1}b^{-1}c^{q_0-1})
\]
and thus $z^2 = 1$ if and only if $b = a^{-q_0}$ and $c = \l a^{-1}$ with $\l^{q_0-1}=1$. Therefore, each coset of the form $C_{q-1}^3\rho x$, where $\rho \in S_4$ is a double transposition, contains $(q-1)(q^{1/2}-1)$ involutions.

Similarly, if $x$ is an involutory graph-field automorphism then 
\[
|x^G| \geqs \frac{1}{4}q^3(q+1)(q^{3/2}-1)(q^2+1) = b_{16}
\]
and there are at most 
\[
a_{16} = (q^{1/2}+1)^3 + \binom{4}{2}(q-1)(q^{1/2}-1) + 3(q-1)(q^{1/2}+1)
\]
of these elements in $H$.

Finally, let us assume $x\in G$ is an involutory graph automorphism. Let $\tau$ be the inverse-transpose graph automorphism of $G_0$ (note that if $\e=-$, then this is induced by the order two field automorphism of $\mathbb{F}_{q^2}$). As explained in \cite[Sections 3.2.5 and 3.3.5]{BG}, we have $C_{G_0}(\tau) = {\rm PSO}_{4}^{\e'}(q).2$ if $q$ is odd (for some choice of sign $\e'$) and 
$C_{G_0}(\tau) = C_{{\rm Sp}_{4}(q)}(t)$ if $q$ is even, where $t \in {\rm Sp}_{4}(q)$ is a transvection. If $x \in H$ is an involutory graph automorphism with $C_{G_0}(x)' = {\rm PSp}_{4}(q)$, then $|x^G| \geqs \frac{1}{2}q^2(q^3-1) = b_{17}$ and we observe that $x$ is contained in a coset of the form $C_{q-\e}^3\rho\tau$, where $\rho \in S_4$ is a double transposition. Now $\tau$ inverts the torus $C_{q-\e}^3$ and we calculate that there are at most $2(q-\e)$ involutions in each of these cosets and so in total there are at most $a_{17} = 6(q+1)$ of these graph automorphisms in $H$. On the other hand, if $C_{G_0}(x)' \ne {\rm PSp}_{4}(q)$, then $|x^G| \geqs \frac{1}{2}q^4(q^2-1)(q^3-1)=b_{18}$ and by counting the involutions in the cosets $C_{q-\e}^3\tau$ and $C_{q-\e}^3\rho\tau$, where $\rho \in S_4$ is a transposition, we deduce that $H$ contains at most $a_{18} = (q+1)^3 + \binom{4}{2}(q+1)^2$ of these graph automorphisms.

If we now bring together the above bounds, we deduce that if $q \geqs 9$ then 
\[
\mathcal{Q}(G,2) < 2q^{-3} + \sum_{i=1}^{10}a_i^2/b_i + \eta a_{14}^2/b_{14} + \sum_{i=15}^{18}a_i^2/b_i,
\]
where $\eta = 1$ if $q=q_0^r$ with $r \geqs 3$, otherwise $\eta = 0$. One can check that this upper bound is less than $1$ for all $q \geqs 9$. In addition, it is less than $q^{-1/2}$ if $q \geqs 29$.
\end{proof}

\begin{lem}\label{l:c2_3}
Suppose $G_0 = {\rm P\O}_{8}^{+}(q)$ and $H$ is of type ${\rm O}_{2}^{\e}(q) \wr S_4$. Then 
\[
b(G,H)= \left\{\begin{array}{ll}
3 & \mbox{if $(\e,q) = (-,2)$} \\
2 & \mbox{otherwise}
\end{array}\right.
\]
and $\mathcal{P}(G,2) \to 1$ as $q \to \infty$.
\end{lem}

\begin{proof}
Let $V$ be the natural module for $G_0$ and write $q=p^f$ with $p$ a prime. Here $H$ is the stabiliser in $G$ of an orthogonal decomposition 
\begin{equation}\label{e:vi}
V = V_1 \perp V_2 \perp V_3 \perp V_4,
\end{equation}
where each $V_i$ is a non-degenerate $2$-space of type $\e$. The precise structure of $H_0$ is given in \cite[Proposition 4.2.11]{KL} and we note that $|H_0| = 2^{m}.24(q-\e)^4$, where $m=1+2\delta_{2,p}$. If $q<5$ then the maximality of $H$ implies that $\e=-$ and using {\sc Magma} one checks that $b(G,H) = 2+\delta_{2,q}$. Therefore, for the remainder we will assume that $q \geqs 5$. We refer the reader to \cite[Section 3.5]{BG} for information on the conjugacy classes of elements of prime order in ${\rm Aut}(G_0)$.

Let $x \in H$ be an element of prime order $r$. First assume $r \geqs 5$, so either $x$ is semisimple and $r$ divides $q-\e$, or $x$ is a field automorphism and $q=q_0^r$. Note that $q \geqs 8$ since $q-\e$ is indivisible by $r$ when $q =5$ or $7$. If $x$ is semisimple, then 
\[
|x^G| \geqs \frac{|{\rm SO}_{8}^{+}(q)|}{|{\rm SO}_{6}^{-}(q)|{\rm GU}_{1}(q)|} > \frac{1}{2}q^{12} = b_1
\]
and plainly there are fewer than $a_1 = (q+1)^4$ such elements in $H$.  Similarly, if $x$ is a field automorphism then $|x^G|>\frac{1}{8}q^{112/5} = b_2$ and we note that $|H| \leqs 2^4.72(q+1)^4 =  a_2$. It follows that the combined contribution to $\mathcal{Q}(G,2)$ from elements of order at least $5$ is less than $a_1^2/b_1 +a_2^2/b_2< q^{-3}$.

Next assume $x \in H$ is a unipotent element of order $3$. Here $p=3$ and $x$ acts as a $3$-cycle on the summands in \eqref{e:vi}, which implies that $x$ has Jordan form $[J_3^2,J_1^2]$ on $V$. Therefore, $|x^G|>\frac{1}{8}q^{18}=b_3$ and $H$ contains at most $8|{\rm O}_{2}^{\e}(q)|^2 \leqs 32(q+1)^2=a_3$ such elements. Since $q \geqs 9$, it follows that the contribution from these elements is less than $a_3^2/b_3 < q^{-9}$. 

Now assume $p \ne 3$ and $x \in H$ is a semisimple element of order $3$, so $x \in H \cap G_0 = H_0$. Suppose $q \not\equiv \e \imod{3}$. Since $|{\rm O}_{2}^{\e}(q)|$ is indivisible by $3$, it follows that $x$ must induce a $3$-cycle on the set of spaces in the decomposition \eqref{e:vi}. Therefore $\dim C_V(x) = 4$, $|x^G|>\frac{1}{2}q^{18} = b_4$ and we note that $i_3(H_0) \leqs 32(q+1)^2 = a_4$.

Now suppose $q \equiv \e \imod{3}$. There are three ${\rm Aut}(G_0)$-classes of elements of order $3$ in $G_0$, represented by 
\[
[I_2, \omega I_3, \omega^2 I_3],\;\; [I_4, \omega I_2, \omega^2 I_2],\;\; [I_6, \omega, \omega^2]
\]
(modulo scalars), where $\omega \in \mathbb{F}_{q^2}$ is a primitive cube root of unity (the ${\rm Aut}(G_0)$-class of the latter element splits into three $G_0$-classes, so there are five $G_0$-classes in total). First assume $x$ is of type $[I_{2}, \omega I_3, \omega^2I_3]$, so $|x^G|>\frac{1}{2}(q+1)^{-1}q^{19} = b_5$. Here we calculate that there are at most
\[
16|{\rm O}_{2}^{\e}(q)|^2 + 2^3\binom{4}{1} \leqs 64(q+1)^2+32 = a_5
\]
such elements in $H$. Similarly, if $x$ is the image of $[I_{4}, \omega I_2, \omega^2I_2]$ then $|x^G|>\frac{1}{2}q^{18} = b_6$ and $H$ contains at most
\[
 8|{\rm O}_{2}^{\e}(q)|^2 + 2^2\binom{4}{2} \leqs 32(q+1)^2 +24 = a_6
\]
of these elements. Finally, suppose $x$ is of type $[I_{6}, \omega, \omega^2]$. Here we have $|x^G|>\frac{1}{2}q^{12} = b_7$ and $|x^{G_0} \cap H_0| \leqs 2\binom{4}{1} = 8$, so $|x^{{\rm Aut}(G_0)} \cap H| \leqs 24 =a_7$. 

Since $a_4^2/b_4 < q^{-8}$ and $\sum_{i=5}^{7}a_i^2/b_i < q^{-7}$ for all $q \geqs 5$, we conclude that the contribution to $\mathcal{Q}(G,2)$ from semisimple or unipotent elements of order $3$ is less than $q^{-7}$.

Next let us consider the contribution from semisimple or unipotent involutions (including involutory graph automorphisms). It will be useful to observe that ${\rm O}_{2}^{\e}(q) \cong D_{2(q-\e)}$.

First assume $p=2$, so $q \geqs 8$. There are five classes of unipotent involutions in ${\rm Aut}(G_0)$, represented by the elements 
\[
b_1, \; a_2, \; c_2, \; b_3, \; c_4
\]
in the notation of Aschbacher and Seitz \cite{AS}. We claim that the total contribution to $\mathcal{Q}(G,2)$ from these elements is less than $\sum_{i=1}^{5}r_i^2/s_i< q^{-2}$, where the terms $r_i$ and $s_i$ are defined in the following table:
\[
\begin{array}{llll} \hline
i & x & r_i & s_i \\ \hline
1 & b_1 & 12(q+1) & \frac{1}{2}q^7 \\
2 & a_2 & 12(q+1) & \frac{1}{2}q^{10} \\
3 & c_2 & 18(q+1)^2 & \frac{1}{2}q^{12} \\
4 & b_3 & 12(q+1)^2(q+7) & \frac{1}{2}q^{15} \\
5 & c_4 & (q+1)^3(q+13) & \frac{1}{2}q^{16} \\ \hline
\end{array}
\]

Here $r_i$ is an upper bound on $|x^{{\rm Aut}(G_0)} \cap H|$ and 
$s_i$ is a lower bound on $|x^{G_0}|$ (see the proof of \cite[Proposition 3.22]{Bur2}, for example), so the claim follows from Lemma \ref{l:calc}. 

For instance, suppose $x$ is a $c_2$-type involution. Here the ${\rm Aut}(G_0)$-class of $x$ is a union of three distinct $G_0$-classes, labelled $c_2$, $a_4$ and $a_4'$ in \cite{AS}. If $x$ is $G_0$-conjugate to $c_2$, then $x$ fixes each summand $V_i$ in \eqref{e:vi}, acting nontrivially on exactly two of the summands. Since $i_2({\rm O}_{2}^{\e}(q)) = q-\e$, it follows that $|x^{G_0} \cap H| \leqs \binom{4}{2}(q-\e)^2 \leqs 6(q+1)^2$. Similarly, if $x$ is $G_0$-conjugate to $a_4$ or $a_4'$ then $x$ induces a double transposition on the $V_i$ and there are at most $3|{\rm O}_{2}^{\e}(q)|^2 \leqs 12(q+1)^2$ such elements in $H$. We conclude that $|x^{{\rm Aut}(G_0)} \cap H| \leqs 18(q+1)^2$, which explains the expression for $r_3$ given in the above table. Similar reasoning applies in the other cases.

Now assume $p \ne 2$ and $x$ is a semisimple involution. If $x$ is a graph automorphism of type $[-I_{1}, I_{7}]$, then  
$|x^G|>\frac{1}{4}q^7=b_8$ and we calculate that $H$ contains at most 
\[
a_8=3\binom{4}{1}(q+1) = 12(q+1)
\]
of these involutions. Next assume $x$ is of type $[-I_{3},I_5]$, which represents the other ${\rm Aut}(G_0)$-class of involutory graph automorphisms. Here $|x^G|>\frac{1}{4}q^{15}=b_9$ and by carefully considering the conjugacy classes of involutions in ${\rm O}_{2}^{\e}(q) \wr S_4 < {\rm O}_{8}^{+}(q)$ we deduce that  
\begin{align*}
a_9 & =
3\left(2\binom{4}{2}(q+1) + \binom{4}{3}(q+1)^3+\binom{4}{2}2(q+1)\cdot \binom{2}{1}(q+1)\right) \\
& = 12(q^2+8q+10)(q+1)
\end{align*}
is an upper bound on the total number of involutions in $H$ of this form. 

Next assume $x$ is ${\rm Aut}(G_0)$-conjugate to an involution of the form $[-I_{2},I_{6}]$. There are two such ${\rm Aut}(G_0)$-classes, each of which splits into three $G_0$-classes, giving six $G_0$-classes in total. Now $|x^G|>\frac{1}{4}q^{12}=b_{10}$ and we see that there are at most
\[
a_{10} = 6\left(\binom{4}{1}+\binom{4}{2}(q+1)^2+\binom{4}{2}2(q+1)\right) = 12(3q^2+12q+11)
\]
such elements in $H$. Finally, let us assume $x$ is ${\rm Aut}(G_0)$-conjugate to an involution of the form $[-I_{4},I_4]$; there are two such ${\rm Aut}(G_0)$-classes, one of which splits into three $G_0$-classes. Now $|x^G|>\frac{1}{8}q^{16} = b_{11}$ and we calculate that $H$ contains at most
\begin{align*}
a_{11} & = 4\left(\binom{4}{2}+12(q+1)^3+(q+1)^4+\binom{4}{2}2(q+1)\left(2+ (q+1)^2\right) + 3(2(q+1))^2\right) \\
& = 4q^4+112q^3+360q^2+496q+268
\end{align*}
involutions of this type.

Putting all of the above estimates together, we conclude that the contribution to $\mathcal{Q}(G,2)$ from semisimple involutions is less than 
\[
\sum_{i=8}^{11}a_i^2/b_i < 2q^{-1}
\]
for all $q \geqs 5$. Given the previous estimate for unipotent involutions when $p=2$, it follows that the total contribution from semisimple or unipotent involutions (including involutory graph automorphisms) is less than $2q^{-1}$.

To complete the proof, we need to consider field and graph-field automorphisms of order $2$ and $3$, as well as graph automorphisms of order $3$.

Suppose $x$ is an involutory field or graph-field automorphism, so $q \geqs 9$ and $|x^G|>\frac{1}{4}q^{14}=b_{12}$. By applying the upper bounds on $|x^G \cap H|$ presented in the proof of \cite[Proposition 2.11]{Bur3}, we deduce that $H$ contains at most
\[
a_{12}  = 2\left((2q^{1/2})^4 + \binom{4}{2}|{\rm O}_{2}^{+}(q)|\cdot (2q^{1/2})^2+3|{\rm O}_{2}^{+}(q)|^2\right)  = 152q^2-144q+24
\]
of these elements. 

Finally, let us assume $x$ is a field, graph or graph-field automorphism of order $3$. First assume $x \in H$ is a triality graph automorphism with $C_{G_0}(x) = G_2(q)$, so $|x^G|>\frac{1}{8}q^{14} = b_{13}$. Fix a set of simple roots $\{a_1, \ldots, \a_4\}$ for the ambient simple algebraic group $\bar{G} = D_4$, labelled in the usual way (so $\a_2$ corresponds to the central node in the corresponding Dynkin diagram). We may assume $x$ cyclically permutes the roots $\a_1, \a_3$ and $\a_4$, so it induces a $3$-cycle on the factors of a standard maximal torus of $\bar{G}$. It follows that $x$ acts as a $3$-cycle on the summands $V_i$ in \eqref{e:vi} and then by counting elements of order $3$ in the coset $(H \cap {\rm PGO}_{8}^{+}(q))x$ we deduce that there are at most  
\[
\frac{4!}{3}\cdot 3|{\rm O}_{2}^{\e}(q)|^2 \leqs 96(q+1)^2 = a_{13}
\]
of these specific graph automorphisms in $H$. For all other field, graph and graph-field automorphisms of order $3$ we have $|x^G|>\frac{1}{8}q^{20} = b_{14}$ and we note that 
\[
|H| \leqs 2^4.72(q+1)^4\log q = a_{14}.
\]

Therefore, the total contribution from field and graph-field automorphisms of order $2$ and $3$, together with graph automorphisms of order $3$, is less than
\[
\eta a_{12}^2/b_{12} + a_{13}^2/b_{13} + a_{14}^2/b_{14}, 
\]
where $\eta=1$ if $q=q_0^2$, otherwise $\eta = 0$. For $q \geqs 7$, one can check this is less than $q^{-2}$. If $q = 5$ then we can remove the $\log q$ factor in the expression for $a_{14}$ and in this way we deduce that the contribution is less than $\frac{1}{4}$.

Finally, by bringing together the above estimates, we conclude that 
\[
\mathcal{Q}(G,2) < q^{-3} + q^{-7} + 2q^{-1} + \mu
\]
for all $q \geqs 5$, where $\mu = q^{-2}$ if $q \geqs 7$ and $\mu = \frac{1}{4}$ if $q=5$.  Therefore $\mathcal{Q}(G,2)<1$ and thus $b(G,H) = 2$. We also deduce that $\mathcal{P}(G,2) \to 1$ as $q$ tends to infinity.
\end{proof}

\begin{lem}\label{l:08+}
Suppose $G_0 = {\rm P\O}_{8}^{+}(q)$ and $H$ is of type ${\rm O}_{2}^{-}(q^2) \times {\rm O}_{2}^{-}(q^2)$. Then $b(G,H)=2$ and $\mathcal{P}(G,2) \to 1$ as $q \to \infty$.
\end{lem}

\begin{proof}
Set $d=(2,q-1)$ and note that 
\[
H_0 = (D_{\frac{2}{d}(q^2+1)} \times D_{\frac{2}{d}(q^2+1)}).2^2 < (\O_{4}^{-}(q) \times \O_{4}^{-}(q)).2^2
\]
and the maximality of $H$ implies that $G$ contains triality graph or graph-field automorphisms (see \cite[Table 8.50]{BHR}). Let us also observe that $H = N_G(P)$, where $P$ is a Sylow $\ell$-subgroup of $G_0$ and $\ell$ is an odd prime divisor of $q^2+1$. Given this, it is easy to check the cases with $q \leqs 7$ using {\sc Magma}, so for the remainder of the proof we will assume that $q \geqs 8$. Let $x \in H$ be an element of prime order $r$.

First assume $x \in H \cap {\rm PGO}_{8}^{+}(q)$. If $r$ is odd, then $x$ is semisimple, $r$ divides $q^2+1$ and
\[
|x^{G_0}| \geqs \frac{|{\rm SO}_{8}^{+}(q)|}{|{\rm SO}_{4}^{-}(q)||{\rm GU}_{1}(q^2)|}  > \frac{1}{2}q^{20}.
\]
Now suppose $r=2$. As explained in the proof of \cite[Proposition 3.4]{Bur4}, every involution in $H \cap {\rm PGO}_{8}^{+}(q)$ is contained in ${\rm Inndiag}(G_0)$, which is the subgroup of ${\rm Aut}(G_0)$ generated by the inner and diagonal automorphisms of $G_0$. As a consequence, if $p=2$ then $x$ is $G$-conjugate to $c_2$ or $c_4$ (in the notation of \cite{AS}), which implies that $|x^{G}|>\frac{3}{2}q^{12}$. Similarly, if $p \ne 2$ then
\[
|x^{G}| \geqs 3\left(\frac{|{\rm O}_{8}^{+}(q)|}{|{\rm O}_{6}^{-}(q)||{\rm O}_{2}^{-}(q)|}\right)  > \frac{1}{2}q^{12} = b_1.
\]
Since $|H \cap {\rm PGO}_{8}^{+}(q)| \leqs 32(q^2+1)^2 = a_1$, it follows that the combined contribution to $\mathcal{Q}(G,2)$ from elements in $H \cap {\rm PGO}_{8}^{+}(q)$ is less than $a_1^2/b_1$.

Finally, let us assume $x \in H \setminus {\rm PGO}_{8}^{+}(q)$, so $x$ is a field, graph or graph-field automorphism. If $x$ is a field or graph-field automorphism of odd order, then $|x^G|>\frac{1}{8}q^{56/3} = b_2$ and we note that $|H| \leqs 96(q^2+1)^2\log q = a_2$. Similarly, if $x$ is an involutory field or graph-field automorphism, then $|x^G|>\frac{1}{8}q^{14} = b_3$ and there are at most $2|H \cap {\rm PGO}_{8}^{+}(q)| \leqs 64(q^2+1)^2 = a_3$ of these elements in $H$. Finally, suppose $x$ is a triality graph automorphism. Here $|x^G|>\frac{1}{8}q^{14} = b_4$ and again we observe that $H$ contains at most $64(q^2+1)^2 = a_4$ of these elements.

To summarise, we have shown that
\[
\mathcal{Q}(G,2) < \sum_{i=1}^{4}a_i^2/b_i
\]
and one checks that this upper bound is less than $1$ for $q \geqs 8$ (in addition, it is less than $q^{-1/2}$ if $q \geqs 11$). The result follows.
\end{proof}

\begin{lem}\label{l:sp4_2}
Suppose $G_0 = {\rm Sp}_{4}(q)$ and $H$ is of type ${\rm O}_{2}^{\e}(q) \wr S_2$ or ${\rm O}_{2}^{-}(q^2)$. Then $b(G,H)=2$ and $\mathcal{P}(G,2) \to 1$ as $q \to \infty$.
\end{lem}

\begin{proof}
In both cases, $q \geqs 4$ is even and $G$ contains graph automorphisms. If $q \leqs 32$ then the desired result can be checked using {\sc Magma}, so we will assume $q \geqs 64$. 

First assume $H$ is of type ${\rm O}_{2}^{\e}(q) \wr S_2$, so 
\[
H_0 = {\rm O}_{2}^{\e}(q) \wr S_2  = D_{2(q-\e)} \wr S_2 < {\rm Sp}_{2}(q) \wr S_2 < G_0.
\]
Let $x \in H$ be an element of prime order $r$. As noted in the proof of \cite[Proposition 3.1]{Bur4}, if $x$ is a unipotent involution then $|x^{G}\cap H|  = 4(q-\e) = a_1$ and $|x^{G}| = 2(q^4-1) = b_1$ if $x$ is $G$-conjugate to a long root element, otherwise $|x^{G} \cap H| \leqs (q+1)^2 = a_2$ and $|x^{G}| = (q^2-1)(q^4-1) = b_2$.
If $x$ is semisimple, then $r$ divides $q-\e$,  
\[
|x^G| \geqs \frac{|{\rm Sp}_{4}(q)|}{|{\rm GU}_{2}(q)|} = q^3(q-1)(q^2+1) = b_3
\]
and we note that $|H_0| \leqs 8(q+1)^2 = a_3$. Similarly, if $x$ is a field automorphism of odd order, then $|x^G|>q^{20/3} = b_4$ and plainly there are fewer than $a_4 = 8(q+1)^2\log q$ field automorphisms in $H$. 

Now assume $x$ is an involutory field or graph automorphism (note that $G$ contains one or the other, but not both). First assume $x$ is a field automorphism, so $\log q$ is even. Here $|x^G| = q^2(q+1)(q^2+1) = b_5$ and we calculate that $|x^G \cap H| \leqs 6q-2 = a_5$. Indeed, if $\e=-$ then $|x^G \cap H| \leqs |{\rm O}_{2}^{-}(q)| = 2(q+1)$, whereas if $\e=+$ we get
\[
|x^G \cap H| \leqs |{\rm O}_{2}^{+}(q)| + \left(\frac{|{\rm O}_{2}^{+}(q)|}{|{\rm O}_{2}^{+}(q^{1/2})|} + \frac{|{\rm O}_{2}^{+}(q)|}{|{\rm O}_{2}^{-}(q^{1/2})|}\right)^2 = 6q-2.
\]
Finally, suppose $x$ is an involutory graph automorphism, so $\log q$ is odd. Here we have $C_{G_0}(x) = {}^2B_2(q)$, so $|x^G| = q^2(q+1)(q^2-1) = b_6$ and we note that $H$ contains fewer than $a_6 = 8(q+1)^2$ of these elements. 

We conclude that
\[
\mathcal{Q}(G,2) < \sum_{i=1}^{4}a_i^2/b_i + \a a_5^2/b_5 + (1-\a) a_6^2/b_6,
\]
where $\a=1$ if $\log q$ is even, otherwise $\a=0$. One checks that this upper bound is less than $1$ for all $q \geqs 64$. In addition, it is less than $q^{-1/2}$ if $q \geqs 2^{12}$.

A very similar argument applies when $H$ is of type ${\rm O}_{2}^{-}(q^2)$ and we omit the details.
\end{proof}

\begin{lem}\label{l:c5}
Suppose $G_0 = {\rm U}_{3}(q)$ and $H$ is of type ${\rm GU}_{3}(2)$, where $q = 2^k$ and $k \geqs 3$ is a prime. Then $b(G,H)=2$ and $\mathcal{P}(G,2) \to 1$ as $q \to \infty$.
\end{lem}

\begin{proof}
By \cite[Proposition 4.5.3]{KL} we have $H_0 = {\rm PGU}_{3}(2)$ if $k=3$, otherwise $H_0 = {\rm U}_{3}(2)$. 
The cases $k \in \{3,5\}$ can be checked directly using {\sc Magma}, so let us assume $k \geqs 7$. Now $|H| \leqs 2k|{\rm PGU}_{3}(2)| = 432k=a_1$ and $|x^G| \geqs (q-1)(q^3+1) = b_1$ for all $x \in G$ of prime order (minimal if $x \in G_0$ is an involution). Therefore, $\mathcal{Q}(G,2) \leqs a_1^2/b_1< 4q^{-1}$
and the result follows.
\end{proof}

\begin{lem}\label{l:c6}
Suppose $G_0 = {\rm L}_{3}^{\e}(q)$ and $H$ is of type $3^{1+2}.{\rm Sp}_{2}(3)$. Then $b(G,H)=2$ and $\mathcal{P}(G,2) \to 1$ as $q \to \infty$.
\end{lem}

\begin{proof}
Here $q=p \equiv \e \imod{3}$ and \cite[Proposition 4.6.5]{KL} gives $H_0 = 3^2.Q_8$ if $q \equiv 4\e,7\e \imod{9}$, otherwise $H_0 = 3^2.{\rm Sp}_{2}(3)$. The cases with $q \leqs 23$ can be checked using {\sc Magma}, so let us assume $q>23$. Now $|H| \leqs 432=a_1$ and $|x^G| \geqs (q-1)(q^3-1) = b_1$ for all $x \in G$ of prime order (minimal if $\e=+$ and $x$ is a unipotent element with Jordan form $[J_2,J_1]$), so $\mathcal{Q}(G,2) \leqs a_1^2/b_1 < 8q^{-1}$
and the result follows.
\end{proof}

This completes the proof of Proposition \ref{p:nonpar}.

\section{Exceptional groups: Non-parabolic actions}\label{s:excep}

Here we complete the proof of Theorem \ref{t:as} by handling the almost simple groups where $G_0$ is an exceptional group of Lie type and $H$ is non-parabolic. As explained in Remark \ref{r:isom}, we may (and will) assume that $G_0 \ne {}^2G_2(3)', G_2(2)'$. Our main result is the following.

\begin{prop}\label{p:nonpar2}
Let $G \leqs {\rm Sym}(\O)$ be a finite almost simple primitive group with socle $G_0$ and soluble point stabiliser $H$. Assume $G_0$ is an exceptional group of Lie type and $H$ is non-parabolic. Then $b(G,H) = 2$ and $\mathcal{P}(G,2) \to 1$ as $|G| \to \infty$.
\end{prop}

\begin{proof}
By inspecting \cite[Table 20]{LZ}, we deduce that $H$ is a maximal rank subgroup of $G$. More precisely, either $H = N_G(T)$ is the normaliser of a maximal torus $T<G_0$, or one of the following holds:
\begin{itemize}\addtolength{\itemsep}{0.2\baselineskip}
\item[{\rm (a)}] $G_0 = G_2(3)$, $H$ is of type ${\rm SL}_{2}(3)^2$.
\item[{\rm (b)}] $G_0 = {}^3D_4(2)$, $H$ is of type $3 \times {\rm SU}_{3}(2)$.
\item[{\rm (c)}] $G= {}^2F_4(2)$, $H = {\rm SU}_{3}(2).2$.
\item[{\rm (d)}] $G_0 = F_4(2)$, $H$ is of type ${\rm SU}_{3}(2)^2$.
\item[{\rm (e)}] $G_0 = {}^2E_6(2)$, $H$ is of type ${\rm SU}_{3}(2)^3$.
\item[{\rm (f)}] $G = E_8(2)$, $H$ is of type ${\rm SU}_{3}(2)^4$.
\end{itemize}

If $H = N_G(T)$ or if $(G,H)$ is one of the cases labelled (c)--(f), then the result follows immediately from \cite[Proposition 4.2]{BTh}. Cases (a) and (b) can be handled using {\sc Magma}.
\end{proof}

\vs

By combining Proposition \ref{p:nonpar2} with Propositions \ref{p:altsp}, \ref{p:psl2}, \ref{p:par} and \ref{p:nonpar}, we conclude that the proof of Theorem \ref{t:as} is complete. The same sequence of propositions also establishes Theorem \ref{t:prob}, while Corollaries \ref{c:odd}, \ref{c:nilp} and \ref{c:ex} follow by inspection. 

\section{Proof of Theorem \ref{t:main}}\label{s:main}

In this section we complete the proof of Theorem \ref{t:main}. Let $G \leqs {\rm Sym}(\O)$ be a finite primitive permutation group with soluble stabiliser $H$. By \cite[Theorem 1.1]{LZ}, one of the following holds:
\begin{itemize}\addtolength{\itemsep}{0.2\baselineskip}
\item[{\rm (a)}] $G = V{:}H$ is an affine group, where $V = \mathbb{F}_p^d$  and $H \leqs {\rm GL}(V)$ is irreducible.
\item[{\rm (b)}] $T^m \normeq G \leqs L \wr S_m$ and $G$ acts on $\O = \Gamma^m$ with the product action, where $L \leqs {\rm Sym}(\Gamma)$ is almost simple and primitive with socle $T$ and a soluble point stabiliser.
\item[{\rm (c)}] $G$ is almost simple and the possibilities for $(G,H)$ are recorded in \cite[Tables 14-20]{LZ}. 
\end{itemize}

In view of Theorem \ref{t:as}, we may assume $G$ is an affine or product-type group as described in cases (a) and (b).

\subsection{Affine groups}\label{ss:aff}

Let $G = V{:}H$ be a primitive affine group, where $V = \mathbb{F}_p^d$ and $H \leqs {\rm GL}(V)$ is irreducible and soluble. Since $G$ itself is soluble, we can apply the following theorem of Seress \cite{Seress} (note that every soluble primitive permutation group is of affine type).

\begin{thm}[Seress \cite{Seress}]\label{t:aff}
Let $G \leqs {\rm Sym}(\O)$ be a finite soluble primitive permutation group with point stabiliser $H$. Then $b(G,H) \leqs 4$. Moreover, $b(G,H) \leqs 3$ if $|H|$ is odd.
\end{thm}

As noted by Seress \cite[p.244]{Seress}, both bounds are sharp in a strong sense. Indeed, a theorem of P\'{a}lfy \cite{Pal1} states that if $G$ is a soluble primitive group of degree $n$, then $|G| \leqs 24^{-1/3}n^c$ with $c = 1 + \log_9(48.24^{1/3}) = 3.243...$, and equality is attained for infinitely many values of $n$. In these cases, Lemma \ref{l:easy} gives $b(G,H) \geqs 4$ and thus the main bound in Theorem \ref{t:aff} is achieved infinitely often. Similarly, if $|H|$ is odd then another result of P\'{a}lfy \cite{Pal2} gives $|G| \leqs 3^{-1/2}n^c$ with $c = 2.278...$ and once again there are infinitely many examples where this bound is attained. 

There are also strong base size results for affine groups in the so-called coprime setting with $(|V|,|H|)=1$. For example, a theorem of Vdovin \cite{Vdovin} gives $b(G,H) \leqs 3$ in this situation (with $H$ soluble), which extends an earlier result of Moret\'{o} and Wolf \cite{MW} in the case where $H$ has odd order. It is worth noting that Vdovin's result has in turn been extended by Halasi and Podoski \cite{HP}, who have proved that $b(G,H) \leqs 3$ for \emph{all} affine groups of the form $G = V{:}H$ with $(|V|,|H|)=1$.

\subsection{Product-type groups}\label{ss:prod}

Let $L \leqs {\rm Sym}(\Gamma)$ be an almost simple primitive group with socle $T$ and soluble point stabiliser $K$. Set $\Omega = \Gamma^m$ with $m \geqs 2$ and consider the product action of $L \wr S_m$ on $\O$. Let $G$ be a subgroup of $L \wr S_m$ with socle $T^m$ such that 
\[
T^m \normeq G \leqs L \wr P
\]
and $P \leqs S_m$ is a transitive permutation group induced by the conjugation action of $G$ on the factors of $T^m$. Then $G \leqs {\rm Sym}(\O)$ is a primitive group of product-type with soluble point stabiliser $H = G \cap (K \wr P)$. As explained in \cite{LZ}, every primitive product-type group with a soluble point stabiliser is of this form. Note that $G = T^mH$, so $H$ also induces $P$ on the factors of $T^m$ and thus the solubility of $H$ implies that $P$ is also soluble.

\begin{thm}\label{t:pa}
Let $G \leqs {\rm Sym}(\O)$ be a finite primitive group of product-type with soluble point stabiliser $H$. Then $b(G,H) \leqs 5$.
\end{thm}

\begin{proof}
As above, write $G \leqs L \wr P \leqs {\rm Sym}(\O)$ where $\O = \Gamma^m$, $m \geqs 2$ and $L \leqs {\rm Sym}(\Gamma)$ is almost simple. Let $d(P)$ be the distinguishing number of $P$, which is the minimal number of colours needed to colour the elements of $\{1,\ldots, m\}$ in such a way that the stabiliser in $P$ of this colouring is trivial. Then by the proof of \cite[Lemma 3.8]{BS} we have
\[
b(G,H) \leqs \left\lceil\frac{\lceil \log d(P) \rceil}{\lfloor \log |\Gamma| \rfloor}\right\rceil+b(L,K).
\]
Now Theorem \ref{t:as} gives $b(L,K) \leqs 5$ and the solubility of $P$ implies that $d(P) \leqs 5$ by \cite[Theorem 1.2]{Seress}. Since $|\Gamma| \geqs 5$, it follows that $b(G,H) \leqs 5$ if $b(L,K) \leqs 3$. 

Now assume $b(L,K) = 4$. If $|\Gamma| \geqs 8$ then the above bound yields 
$b(G,H) \leqs 5$, so we may assume $|\Gamma|<8$ and thus $L = S_5$ and $K = S_4$ by Theorem \ref{t:as}. For a positive integer $d$, let ${\rm reg}(L,d)$ denote the number of regular orbits of $L$ with respect to its natural action on $\Gamma^d$. Since $d(P) \leqs 5$, \cite[Theorem 2.13]{BC} implies that $b(G,H) \leqs 5$ if and only if ${\rm reg}(L,5) \geqs 5$ (Vdovin makes the same observation in \cite{Vdovin2}). Using {\sc Magma}, it is easy to check that 
${\rm reg}(L,5) = 11$ and thus $b(G,H) \leqs 5$ as required.

To complete the proof, we may assume $b(L,K)=5$. Here Theorem \ref{t:as} implies that one of the following holds, where $T$ denotes the socle of $L$:
\begin{itemize}\addtolength{\itemsep}{0.2\baselineskip}
\item[{\rm (a)}] $L = S_8$, $K = S_4 \wr S_2$ and $|\Gamma| = 35$.
\item[{\rm (b)}] $T = {\rm L}_{4}(3)$, $K = P_2$ and $|\Gamma| = 130$.
\item[{\rm (c)}] $T = {\rm U}_{5}(2)$, $K=P_1$ and $|\Gamma| = 165$.
\end{itemize}
We claim that ${\rm reg}(L,5) \geqs 5$ in each of these cases, which gives $b(G,H) \leqs 5$ as above.

In case (a), the proof of \cite[Theorem 2]{VZ} gives ${\rm reg}(L,5) \geqs 12$ and thus $b(G,H) \leqs 5$ as required. In fact, a straightforward {\sc Magma} computation shows that ${\rm reg}(L,5)=600$ in this case. To handle cases (b) and (c), write $K = L_{\gamma}$ for some fixed $\gamma \in \Gamma$ and let $t$ be the number of tuples of the form $(\gamma,\l_1,\l_2,\l_3,\l_4) \in \Gamma^5$ with $\bigcap_{i}K_{\l_i} = 1$. Then ${\rm reg}(L,5) \geqs 5t/|L|$ and so we just need to verify the bound $t \geqs |L|$ for $L = {\rm Aut}(T)$. Using {\sc Magma}, we calculate that
\[
t = 100776960 > |{\rm Aut}({\rm L}_{4}(3))| = 24261120
\]
in case (b) and similarly
\[
t = 496668672 > |{\rm Aut}({\rm U}_{5}(2))| = 27371520
\]
in case (c). The result follows. 
\end{proof}

\begin{rem}\label{r:best}
It is easy to see that there are infinitely many finite primitive groups $G$ with a soluble stabiliser $H$ and $b(G,H)=5$. For example, we can take any group of the form $L \wr C_m$ with its product action on $\O = \Gamma^m$, where $m$ is a positive integer and $L \leqs {\rm Sym}(\Gamma)$ is one of the groups in (a), (b) or (c) above. We can also take $G = S_5 \wr C_m$ acting on $5^m$ points for any $m \geqs 2$. Indeed, in this case $b(G,H) \leqs 5$ by Theorem \ref{t:pa}, while \cite[Theorem 2.13]{BC} implies that $b(G,H) \geqs 5$ since $L = S_5$ is $4$-transitive on $\Gamma$ and therefore ${\rm reg}(L,4) = 1 < d(C_m)$. 
\end{rem}

By combining Theorems \ref{t:aff} and \ref{t:pa} with Theorem \ref{t:as}, we conclude that the proof of Theorem \ref{t:main} is complete.

\section{The tables}\label{s:tables}

In this final section, we present the tables referred to in the statement of Theorem \ref{t:as}. First we record some remarks on their content.

\begin{rem}\label{r:cases}
In Tables \ref{tab:as3} and \ref{tab:as4}, we exclude the almost simple groups  with socle $G_0$, where $G_0$ is one of the following:
\[
{\rm L}_{2}(4), \, {\rm L}_{2}(5), \, {\rm L}_{2}(9), \, {\rm L}_{3}(2), \, {\rm L}_{4}(2), \, {\rm PSp}_{4}(2)', \, {\rm PSp}_{4}(3).
\]
This is justified by the existence of the following isomorphisms:
\[
{\rm L}_{2}(4) \cong {\rm L}_{2}(5) \cong A_5,\, {\rm L}_{2}(9) \cong {\rm PSp}_{4}(2)' \cong A_6, \, {\rm L}_{3}(2) \cong {\rm L}_{2}(7),
\] 
\[
{\rm L}_{4}(2) \cong A_8, \, {\rm PSp}_{4}(3) \cong {\rm U}_{4}(2).
\]
So for example, a reader who is interested in the groups with socle ${\rm L}_{4}(2)$ should consult Table \ref{tab:as1} and the cases with $G=S_8$ or $A_8$.

Similarly, since ${}^2G_2(3)' \cong {\rm L}_{2}(8)$ and $G_2(2)' \cong {\rm U}_{3}(3)$, we also exclude the groups with socle ${}^2G_2(3)'$ or $G_2(2)'$ in Table \ref{tab:as2}.
\end{rem}

\begin{rem}\label{r:class}
Let us record some additional comments on Tables \ref{tab:as3} and \ref{tab:as4}.  
\begin{itemize}\addtolength{\itemsep}{0.2\baselineskip}
\item[{\rm (i)}] In Table \ref{tab:as3}, we adopt the notation for parabolic subgroups described in Remark \ref{r:nota}. In addition, if $q=p^f$ with $p$ prime, then $\phi$ denotes a field automorphism of order $f$.
\item[{\rm (ii)}] Suppose $G_0 = {\rm L}_{2}(q)$ and $H$ is of type $P_1$ (see Table \ref{tab:as3}). Here $b(G,H) \in \{3,4\}$, with $b(G,H) = 3$ if and only if 
$G \leqs {\rm PGL}_{2}(q)$, or 
\[
\mbox{$q=p^f$, $p \geqs 3$, $f$ is even and $G = \la G_0, \delta\phi^{f/2} \ra  = G_0.2$,}
\]
where $\delta$ is a diagonal automorphism of $G_0$ (see \eqref{e:aut}).
\item[{\rm (iii)}] In Table \ref{tab:as3}, suppose $G_0 = {\rm U}_{4}(q)$ and $H$ is of type $P_1$, so the solubility of $H$ implies that $q \in \{2,3\}$. If $q=2$ then $b(G,H)=4$. For $q=3$ we have 
\[
b(G,H) = \left\{\begin{array}{ll}
4 & G \in \{G_0.D_8, G_0.[4], G_0.2_1\} \\
3 & \mbox{otherwise}
\end{array}\right.
\]
where $G_0.2_1$ is the unique index-two subgroup of ${\rm PGU}_{4}(3)$.
\item[{\rm (iv)}] Suppose $G_0 = {\rm L}_{3}(q)$ and $H$ is of type ${\rm GU}_{3}(q^{1/2})$ (see Table \ref{tab:as4}). Here $q=4$ since $H$ is soluble and we get 
\[
b(G,H) = \left\{\begin{array}{ll}
3 & \mbox{if $|G:G_0| \geqs 3$ or $G = G_0.2_2$} \\
2 & \mbox{otherwise}
\end{array}\right.
\]
where $G_0.2_2$ contains involutory field automorphisms. 
\item[{\rm (v)}] If $G_0 = {\rm L}_{4}(q)$ and $H$ is of type ${\rm GL}_{2}(q) \wr S_2$, then the solubility and maximality of $H$ implies that $q=3$ and $G$ is one of $G_0.2^2$, $G_0.2_1 = {\rm PGL}_{4}(3)$ or $G_0.2_3$ (the latter group contains an involutory graph automorphism $x$ with $C_{G_0}(x) = {\rm PSO}_{4}^{-}(3).2$). We get $b(G,H) = 3$ in every case.
\end{itemize}
\end{rem}

\begin{table}
\[
\begin{array}{cll} \hline
b & G & H \\ \hline
5 & S_8 & S_4 \wr S_2 \\
4 & S_5 & S_4 \\
& A_6.2^2 & {\rm AGL}_{1}(9).2 \\
& S_6 & S_4 \times S_2, S_2 \wr S_3, S_3 \wr S_2  \\
& A_8 & (S_4 \wr S_2) \cap G  \\
3 & A_5 & A_4, D_{10} \\
& S_5 & S_3 \times S_2,\, 5{:}4 \\
& A_6.2^2 & D_{20}.2,\, [32]  \\
& {\rm PGL}_{2}(9) & D_{20},\, 3^2{:}Q_8 \\
& {\rm M}_{10} & {\rm AGL}_{1}(9) \\
& A_6 & (S_4 \times S_2) \cap G,\, (S_2 \wr S_3) \cap G,\, (S_3 \wr S_2) \cap G  \\
& S_7 & S_4 \times S_3  \\
& A_7 & (S_4 \times S_3) \cap G \\
& S_8 & S_2 \wr S_4  \\
& S_9 & S_3 \wr S_3,\, {\rm AGL}_{2}(3)  \\ 
& A_9 & (S_3 \wr S_3) \cap G  \\
& S_{12} & S_3 \wr S_4,\, S_4 \wr S_3  \\
& A_{12} & (S_3 \wr S_4) \cap G,\, (S_4 \wr S_3) \cap G \\
& S_{16} & S_4 \wr S_4  \\
& A_{16} & (S_4 \wr S_4) \cap G  \\
& {\rm M}_{11} & 3^2{:}Q_8.2 \\ 
& {\rm M}_{12} & 3^2{:}2S_4,\, 2^{1+4}{:}S_3,\, 4^2{:}D_{12} \\
& {\rm M}_{12}.2 & 2^{1+4}{:}S_3.2,\, 4^2{:}D_{12}.2,\, 3^{1+2}{:}D_8 \\
& {\rm J}_{2} & 2^{2+4}{:}(3 \times S_3)  \\
& {\rm J}_2.2 & 2^{2+4}{:}(3 \times S_3).2  \\
& {\rm Fi}_{22} & 3^{1+6}{:}2^{3+4}{:}3^2{:}2  \\
& {\rm Fi}_{22}.2 & 3^{1+6}{:}2^{3+4}{:}3^2{:}2.2  \\
& {\rm Fi}_{23} & 3^{1+8}.2^{1+6}.3^{1+2}.2S_4 \\
\hline
\end{array}
\]
\caption{Alternating and sporadic groups}
\label{tab:as1}
\end{table}

\begin{table}
\[
\begin{array}{clll} \hline
b & G_0 & H \cap G_0 & \mbox{Conditions} \\ \hline
3 & F_4(q) & [2^{22}]{:}S_3^2 & G = F_4(2).2 \\
& G_2(q) & [q^6]{:}C_{q-1}^2 & \mbox{$p=3$, $G \not\leqs \la G_0, \phi \ra$} \\
& & [3^5]{:}{\rm GL}_{2}(3) & G = G_2(3) \\
& {}^3D_4(q) & [q^{11}]{:}((q^3-1) \circ {\rm SL}_{2}(q)).(2,q-1) & q=2,3 \\
& {}^2F_4(q)' & [2^9]{:}5{:}4,\, [2^{10}]{:}S_3 & q=2 \\
& {}^2B_2(q) & [q^{2}]{:}C_{q-1} & \\
& {}^2G_2(q) & [q^{3}]{:}C_{q-1} & q \geqs 27 \\ \hline
\end{array}
\]
\caption{Exceptional groups}
\label{tab:as2}
\end{table}

\begin{table}
\[
\begin{array}{clll} \hline
b & G_0 & \mbox{Type of $H$} & \mbox{Conditions} \\ \hline
5 & {\rm L}_{4}(q) & P_2 & q=3 \\
& {\rm U}_{5}(q) & P_1 & q=2 \\
4 & {\rm L}_{2}(q) & P_1 & \mbox{See Remark \ref{r:class}(ii)} \\
& {\rm L}_{3}(q) & P_1, P_2 & G = {\rm L}_{3}(3) \\
& & P_{1,2} & G = {\rm L}_{3}(4).D_{12} \\
& {\rm L}_{4}(q)  & P_{1,3} & G = {\rm L}_{4}(3).2^2 \\
& {\rm U}_{4}(q) & P_1 & \mbox{$q = 2,3$; see Remark \ref{r:class}(iii)} \\
& {\rm L}_{5}(q) & P_{2,3} & G = {\rm L}_{5}(3).2 \\ 
& {\rm L}_{6}(q) & P_{2,4} & G = {\rm L}_{6}(3).2^2 \\
& {\rm PSp}_{6}(q) & P_2 & G = {\rm PGSp}_{6}(3) \\
& \O_7(q) & P_2 & G = {\rm SO}_{7}(3) \\
& {\rm P\O}_{8}^{+}(q) & P_2 & \mbox{$q=2,3$ and $G \ne G_0$} \\
3 & {\rm L}_{2}(q) & P_1 &  \mbox{See Remark \ref{r:class}(ii)} \\
& {\rm L}_{3}(q) & P_{1,2} & G \not\leqs \la {\rm PGL}_{3}(q),\phi \ra, G \ne {\rm L}_{3}(4).D_{12} \\ 
& {\rm U}_{3}(q) & P_1 & \\
& {\rm L}_{4}(q) & P_{1,3} & G = {\rm L}_{4}(3).2 \ne {\rm PGL}_{4}(3) \\
& {\rm U}_{4}(q) & P_1 & \mbox{$q=2,3$; see Remark \ref{r:class}(iii)} \\
& {\rm Sp}_{4}(q) & [q^4]{:}C_{q-1}^2 & \mbox{$q \geqs 4$ even, $G \not\leqs \la G_0, \phi \ra$} \\ 
& {\rm L}_{5}(q) & P_{2,3} & G = {\rm L}_{5}(2).2 \\
& {\rm L}_{6}(q) & P_{2,4} & \mbox{$G = {\rm L}_{6}(2).2$ or ${\rm L}_{6}(3).2 \ne {\rm PGL}_{6}(3)$} \\
& {\rm PSp}_{6}(q) & P_2 & \mbox{$q=2$ or $G = {\rm PSp}_{6}(3)$} \\
& \O_7(q) & P_2 & G = \O_{7}(3) \\
& {\rm P\O}_{8}^{+}(q) & P_2 & \mbox{$q=2,3$ and $G = G_0$} \\
& & P_{1,3,4} & \mbox{$q=2,3$, $G \not\leqs {\rm PGO}_{8}^{+}(q)$} \\ \hline
\end{array}
\]
\caption{Classical groups in parabolic actions}
\label{tab:as3}
\end{table}

\begin{table}
\[
\begin{array}{clll} \hline
b & G_0 & \mbox{Type of $H$} & \mbox{Conditions} \\ \hline
4 
& {\rm U}_{4}(q) & {\rm GU}_{3}(q) \times {\rm GU}_{1}(q) & q = 2 \\
& & {\rm GU}_{1}(q) \wr S_4 & q = 2 \\
3 & {\rm L}_{2}(q) & {\rm GL}_{1}(q) \wr S_2 & {\rm PGL}_{2}(q) < G \\
& & {\rm GL}_{1}(q^2) & {\rm PGL}_{2}(q) \leqs G \\
& & 2^{1+2}_{-}.{\rm O}_{2}^{-}(2) & q=7 \\
& {\rm L}_{3}(q) & {\rm GL}_{2}(q) \times {\rm GL}_{1}(q) & G = {\rm L}_{3}(3).2 \\
& & {\rm GL}_{1}(q^3) & G = {\rm L}_{3}(3).2 \\ 
& & {\rm GU}_3(q^{1/2}) & \mbox{$q=4$; see Remark \ref{r:class}(iv)} \\
& {\rm U}_{3}(q) & {\rm GU}_{2}(q) \times {\rm GU}_{1}(q) & q=3 \\
& & {\rm GU}_{1}(q) \wr S_3 & \mbox{$q=3$, or $q=4$ and $G \ne G_0$} \\ 
& {\rm L}_{4}(q) & {\rm GL}_{2}(q) \wr S_2 & \mbox{$q=3$; see Remark \ref{r:class}(v)} \\
& & {\rm O}_{4}^{+}(q) & G = {\rm L}_{4}(3).2^2 \\
& {\rm U}_{4}(q) & {\rm GU}_{1}(q) \wr S_4 & G = {\rm U}_{4}(3).D_8 \\
& & {\rm GU}_{2}(q) \wr S_2 & \mbox{$q=3$ and $G \ne G_0$} \\
& {\rm U}_{5}(q) & {\rm GU}_{3}(q) \times {\rm GU}_{2}(q) & q=2 \\
& {\rm U}_{6}(q) & {\rm GU}_{3}(q) \wr S_2 & q = 2 \\
& {\rm PSp}_{6}(q) & {\rm Sp}_{2}(q) \wr S_3 & G = {\rm PGSp}_{6}(3) \\
& {\rm P\O}_{8}^{+}(q) & {\rm O}_{4}^{+}(q) \wr S_2 & \mbox{$q = 3$ and $|G:G_0| \geqs 6$} \\
& & {\rm O}_{2}^{-}(q) \wr S_4 & q = 2 \\
& & {\rm O}_{2}^{-}(q) \times {\rm GU}_{3}(q) & G = \O_{8}^{+}(2).S_3 \\
\hline
\end{array}
\]
\caption{Classical groups in non-parabolic actions}
\label{tab:as4}
\end{table}


\clearpage

\begin{thebibliography}{999}

\bibitem{AS}
M. Aschbacher and G.M. Seitz, \emph{Involutions in Chevalley groups over fields of even order}, Nagoya Math. J. \textbf{63} (1976), 1--91.

\bibitem{Babai} 
L. Babai, \emph{On the order of uniprimitive permutation groups}, Annals of Math. \textbf{113} (1981), 553--568. 

\bibitem{BGP}
L. Babai, A. Goodman and L. Pyber, \emph{Groups without faithful transitive permutation representations of small degree}, J. Algebra \textbf{195} (1997), 1--29.

\bibitem{Bay1}
A.A. Baikalov, \emph{Intersection of conjugate solvable subgroups in classical groups of Lie type}, preprint (2018), arxiv:1703.00124.

\bibitem{Bay2}
A.A. Baikalov, \emph{Intersection of conjugate solvable subgroups in symmetric groups}, Algebra Logic \textbf{56} (2017), 87--97.

\bibitem{BC}
R.F. Bailey and P.J. Cameron, \emph{Base size, metric dimension and other invariants of groups and graphs}, Bull. Lond. Math. Soc. \textbf{43} (2011),  209--242.

\bibitem{BCN}
C. Benbenishty, J.A. Cohen and A.C. Niemeyer, \emph{The minimum length of a base for the symmetric group acting on partitions}, European J. Comb. \textbf{28} (2007), 1575--1581.

\bibitem{Blaha} 
K.D. Blaha, \emph{Minimum bases for permutation groups: the greedy approximation}, J. Algorithms \textbf{13} (1992), 297--306.

\bibitem{Boch} 
A. Bochert, \emph{{\"{U}}ber die {Z}ahl verschiedener {W}erte, die eine {F}unktion gegebener {B}uchstaben durch {V}ertauschung derselben erlangen kann}, Math. Ann. \textbf{33} (1889), 584--590.
  
\bibitem{magma} 
W. Bosma, J. Cannon and C. Playoust, \emph{The {\textsc{Magma}} algebra system I: The user language}, J. Symb. Comput. \textbf{24} (1997), 235--265.

\bibitem{BHR}
J.N. Bray, D.F. Holt and C.M. Roney-Dougal, \emph{The maximal subgroups of the low-dimensional finite classical groups}, London Math. Soc. Lecture Notes Series, vol. 407, Cambridge University Press, 2013.

\bibitem{Bur18} 
T.C. Burness, \emph{On base sizes for almost simple primitive groups}, J. Algebra \textbf{516} (2018), 38--74.

\bibitem{Bur181}
T.C. Burness, \emph{Simple groups, fixed point ratios and applications}, in Local representation theory and simple groups, 267--322, EMS Ser. Lect. Math., Eur. Math. Soc., Z\"{u}rich, 2018.

\bibitem{Bur2} 
T.C. Burness, \emph{Fixed point ratios in actions in finite classical groups, II}, J. Algebra \textbf{309} (2007), 80--138.

\bibitem{Bur3}
T.C. Burness, \emph{Fixed point ratios in actions of finite classical groups, III},  J. Algebra \textbf{314} (2007), 693--748.

\bibitem{Bur4}
T.C. Burness, \emph{Fixed point ratios in actions of finite classical groups, IV},  J. Algebra \textbf{314} (2007), 749--788.

\bibitem{Bur7}
T.C. Burness, \emph{On base sizes for actions of finite classical groups}, J.
  London Math. Soc. \textbf{75} (2007), 545--562.
  
\bibitem{BGL} 
T.C. Burness, M. Garonzi and A. Lucchini, \emph{Finite groups, minimal bases and the intersection number}, submitted (2020), arxiv:2009.10137.
  
\bibitem{BG_saxl}
T.C. Burness and M. Giudici, \emph{On the Saxl graph of a permutation group}, Math. Proc. Cambridge Philos. Soc. \textbf{168} (2020), 219--248.

\bibitem{BG}
T.C. Burness and M. Giudici, \emph{Classical groups, derangements and primes}, Australian Mathematical Society Lecture Series, vol. 25, Cambridge University Press, Cambridge, 2016.

\bibitem{BGS}
T.C. Burness, R.M. Guralnick and J. Saxl, \emph{On base sizes for symmetric groups}, Bull. Lond. Math. Soc. \textbf{43} (2011), 386--391.

\bibitem{BH_UDN}
T.C. Burness and S. Harper, \emph{Finite groups, $2$-generation and the uniform domination number}, Israel J. Math. \textbf{239} (2020), 271--367.

\bibitem{BLS} 
T.C. Burness, M.W. Liebeck and A. Shalev, \emph{Base sizes for simple groups and a conjecture of Cameron}, Proc. Lond. Math. Soc. \textbf{98} (2009), 116--162.

\bibitem{BOW}
T.C. Burness, E.A. O'Brien, R.A. Wilson, \emph{Base sizes for sporadic simple groups}, Israel J. Math. \textbf{177} (2010), 307--333.

\bibitem{BS}
T.C. Burness and \'{A}. Seress, \emph{On Pyber's base size conjecture},
Trans. Amer. Math. Soc. \textbf{367} (2015), 5633--5651.

\bibitem{BTh}
T.C. Burness and A.R. Thomas, \emph{The classification of extremely primitive groups}, Int. Math. Res. Not. IMRN, to appear. 

\bibitem{CamPG} 
P.J. Cameron, \emph{Permutation Groups}, London Math. Soc. Student Texts \textbf{45}, Cambridge University Press, 1999.

\bibitem{CK} 
P.J. Cameron and W.M. Kantor, \emph{Random permutations: some group-theoretic aspects}, Combin. Probab. Comput. \textbf{2} (1993), 257--262.

\bibitem{Carter} 
R.W. Carter, \emph{Finite Groups of Lie Type: Conjugacy Classes and Complex Characters}, John Wiley, London, 1985.

\bibitem{Chang}
B. Chang, \emph{The conjugate classes of Chevalley groups of type $(G_2)$},  J. Algebra \textbf{9} (1968), 190--211.

\bibitem{Atlas}
J.H. Conway, R.T. Curtis, S.P. Norton, R.A. Parker, and R.A. Wilson, \emph{Atlas of finite groups}, Oxford University Press, 1985.

\bibitem{DHM}
H. Duyan, Z. Halasi and A. Mar\'{o}ti, \emph{A proof of Pyber's base size conjecture}, Adv. Math. \textbf{331} (2018), 720--747.

\bibitem{Eno}
H. Enomoto, \emph{The characters of the finite Chevalley group $G_2(q)$, $q=3^f$}, Japan. J. Math. \textbf{2} (1976), 191--248.

\bibitem{FI} I.A. Farad\^{z}ev and A.A. Ivanov, \emph{Distance-transitive representations of groups $G$ with ${\rm PSL}_2(q) \normeq G \leqs {\rm P\Gamma L}_2(q)$}, Europ. J. Combinatorics \textbf{11} (1990), 347--356.

\bibitem{GM}
M. Geck and G, Malle, \emph{The Character Theory of Finite Groups of Lie Type}, Cambridge Studies in Advanced Mathematics, vol. 133, Cambridge University Press, Cambridge, 2020.

\bibitem{Halasi} 
Z. Halasi, \emph{On the base size for the symmetric group acting on subsets}, Studia Sci. Math. Hungar. \textbf{49} (2012), 492--500.

\bibitem{HLM}
Z. Halasi, M.W. Liebeck and A. Mar\'{o}ti, \emph{Base sizes of primitive groups: bounds with explicit constants}, J. Algebra \textbf{521} (2019), 16--43.

\bibitem{HP}
Z. Halasi and K. Podoski, \emph{Every coprime linear group admits a base of size two}, Trans. Amer. Math. Soc. \textbf{368} (2016), 5857--5887.

\bibitem{James} 
J.P. James, \emph{Partition actions of symmetric groups and regular bipartite graphs}, Bull. London Math. Soc. \textbf{38} (2006), 224--232.

\bibitem{KL} 
P.B. Kleidman and M.W. Liebeck, \emph{The {S}ubgroup {S}tructure of the {F}inite {C}lassical {G}roups}, London Math. Soc. Lecture Note Series, vol. 129, Cambridge University Press, 1990.

\bibitem{LLS2}
R. Lawther, M.W. Liebeck and G.M. Seitz, \emph{Fixed point ratios in actions of finite exceptional groups of Lie type}, Pacific J. Math. \textbf{205} (2002), 393--464.

\bibitem{LZ}
C.H. Li and H. Zhang, \emph{The finite primitive groups with soluble stabilizers, and the edge-primitive $s$-arc transitive graphs}, Proc. Lond. Math. Soc. \textbf{103} (2011), 441--472.

\bibitem{L10}
M.W. Liebeck, \emph{On minimal degrees and base sizes of primitive permutation groups}, Arch. Math. \textbf{43} (1984), 11--15.

\bibitem{LS_book} 
M.W. Liebeck and G.M. Seitz, \emph{Unipotent and {N}ilpotent {C}lasses in {S}imple {A}lgebraic {G}roups and {L}ie {A}lgebras}, Amer. Math. Soc. Monographs and Surveys series, volume 180, 2012.

\bibitem{LSh3}
M.W. Liebeck and A. Shalev, \emph{Bases of primitive permutation groups}, in Groups, combinatorics \& geometry (Durham, 2001), 147--154, World Sci. Publ., River Edge, NJ, 2003.

\bibitem{LSh2} 
M.W. Liebeck and A. Shalev, \emph{Simple groups, permutation groups, and probability}, J. Amer. Math. Soc. \textbf{12} (1999), 497--520.


\bibitem{Lu2} 
F. L\"{u}beck, \emph{Centralizers and numbers of semisimple classes in exceptional groups of Lie type}, \\
\texttt{http://www.math.rwth-aachen.de/$\sim$Frank.Luebeck/chev/CentSSClasses}

\bibitem{Lub} 
F. L\"{u}beck, \emph{Generic {C}omputations in {F}inite {G}roups of {L}ie {T}ype}, book in preparation.

\bibitem{Kou} 
V.D. Mazurov and E.I. Khukhro, \emph{Unsolved problems in group theory: The Kourovka notebook, no. 19 (English version)} (2019), arxiv:1401.0300.

\bibitem{MW}
A. Moret\'{o} and T.R. Wolf, \emph{Orbit sizes, character degrees and Sylow subgroups}, Adv. Math. \textbf{184} (2004), 18--36.

\bibitem{MS}
J. Morris and P. Spiga, \emph{On the base size of the symmetric and the alternating group acting on partitions}, submitted (2021), arxiv:2102.10428.

\bibitem{Pal1} 
P.P. P\'{a}lfy, \emph{A polynomial bound for the orders of primitive solvable groups}, J. Algebra \textbf{77} (1982), 127--137.

\bibitem{Pal2}
P.P. P\'{a}lfy, \emph{Bounds for linear groups of odd order}, Proc. Second Internat. Group Theory Conf., Bressanone/Brixen 1989, Suppl. Rend. Circ. Mat. Palermo \textbf{23} (1990), 253--263.

\bibitem{Pyber}
L. Pyber, \emph{Asymptotic results for permutation groups}, in Groups and Computation (eds. L. Finkelstein and W. Kantor), DIMACS Series, vol. 11, pp.197--219, 1993.

\bibitem{Seress}
\'{A}. Seress, \emph{The minimal base size of primitive solvable permutation groups}, J. London Math. Soc. \textbf{53} (1996), 243--255.

\bibitem{SF}
W.A. Simpson and J.S. Frame, \emph{The character tables for ${\rm SL}(3,q)$,  ${\rm SU}(3,q^2)$, ${\rm PSL}(3,q)$, ${\rm PSU}(3,q^2)$}, Canadian J. Math. \textbf{25} (1973), 486--494.

\bibitem{Vdovin3}
E.P. Vdovin, \emph{On intersections of solvable Hall subgroups in finite simple exceptional groups of Lie type}, Tr. Inst. Mat. Mekh. \textbf{19} (2013), 62--70.

\bibitem{Vdovin2}
E.P. Vdovin, \emph{On the base size of a transitive group with solvable point stabilizer}, J. Algebra Appl. \textbf{11} (2012), 1250015, 14 pp.

\bibitem{Vdovin}
E.P. Vdovin, \emph{Regular orbits of solvable linear $p'$-groups}, 
Sib. \`{E}lektron. Mat. Izv. \textbf{4} (2007), 345--360.

\bibitem{VZ}
E.P. Vdovin and V.I. Zenkov, \emph{On the intersections of solvable Hall subgroups in finite groups}, Proc. Steklov Inst. Math. \textbf{267} (2009), suppl. 1, S234–S243.

\end{thebibliography}
\end{document}